\documentclass[12pt]{article}
\usepackage{geometry} 
\usepackage[ utf8 ]{inputenc}
\usepackage[T1]{fontenc}
\usepackage[english,french]{babel}
\usepackage{amsmath}
\usepackage{amsfonts}
\usepackage{amssymb}
\usepackage{amsthm}
\usepackage{textcomp}
\usepackage{eucal}
\usepackage{array}
\theoremstyle{plain}
\usepackage[english,french]{babel}
\newtheorem{theorem}{Théorème}
\newtheorem{corollary}{Corollaire}

\newtheorem{proposition}{Proposition}
\newtheorem{notation}{Notation}
\theoremstyle{definition}

\theoremstyle{remark}
\newtheorem{remark}{Remarque}
\date{}

\title{ Représentations de réflexion de groupes de Coxeter\\Deuxième partie: outils pour des exemples}
\author{François ZARA
}
\begin{document}
\maketitle
\begin{abstract}
Cette partie est composée de trois sections. Dans la première section, nous étudions la famille de polynômes dont les racines sont $4\cos^{2}\frac{k\pi}{n}$, $(n\geqslant 3,1\leqslant k<\frac{n}{2})$. Nous obtenons ainsi une famille de polynômes orthogonaux. Cela nous permettra d'étudier en détail les exemples qui suivent. Dans la deuxième section nous donnons des formules techniques  pour ne pas refaire les calculs à chaque fois. Dans la troisième section nous donnons des applications, d'abord lorsque le corps $K$ est un sous-corps de $\mathbb{R}$ (étude de présentations de $W(H_{3})$ et $W(H_{4})$) et ensuite dans le cas complexe (étude des groupes de réflexion complexes $G(p,p,n)$, $G_{24}$ et $G_{27}$).
\end{abstract}
\begin{otherlanguage}{english}
\begin{abstract}
This part is made of three sections. In the first section we study the family of polynomials whose roots are $4\cos^{2}\frac{k\pi}{n}$, $(n\geqslant 3,1\leqslant k<\frac{n}{2})$. We obtain in this manner a family of orthogonal polynomials. This will permit us to study in details all the examples which follow. In the second section, we give technical formulae in order noto repeat calculations. In the third section, we give applications, first when the field $K$ is a sub-field of $\mathbb{R}$ (presentations of $W(H_{3})$ and $W(H_{4})$) then in the complex case (study of the complex reflection group $G(p,p,n)$, $G_{24}$ and $G_{27}$).
\end{abstract}
\end{otherlanguage}
\let\thefootnote\relax\footnote{Mots clés et phrases: groupes de Coxeter, groupes de réflexion, polynômes orthogonaux.}
\let\thefootnote\relax
\section{Une famille de polynômes orthogonaux}
On donne quelques formules concernant les polynômes $u_{n}(X)$  et leurs racines qui nous seront utiles pour étudier les exemples. 
\subsection{Les polynômes $u_{n}(X)$.}
On commence par rappeler quelques résultats de \cite{Z}. 
On considère la suite $(u_{n})_{n\in \mathbb{N}}$ de polynômes à coefficients entiers:

\begin{gather}
u_{2n+1}(X) :=\sum_{k=0}^{n} (-1)^{k}\binom{2n-k}{k}X^{n-k} \qquad (n\geqslant 0),\\
u_{2n+2}(X) :=\sum_{k=0}^{n} (-1)^{k}\binom{2n+1-k}{k}X^{n-k} \qquad (n\geqslant 0),\\
u_{0}(X) :=0.
\end{gather}
Les premiers polynômes sont: $u_{0}(X)=0$, $u_{1}(X)=u_{2}(X)=1$, $u_{3}(X)=X-1$, $u_{4}(X)=X-2$, $u_{5}(X)=X^{2}-3X+1$, $u_{6}(X)=X^{2}-4X+3$, $u_{7}(X)=X^{3}-5X^{2}+6X-1$, etc. ...\\
Nous définissons $u_{n}(X)$ pour $n<0$ par $u_{n}(X)=-u_{-n}(X)$.
\begin{proposition} \label{a1}
1) Nous avons les formules de récurrence:
\begin{gather} 
\forall n \in \mathbb{Z}, \quad u_{2n+2}(X)-u_{2n+1}(X)+u_{2n}(X)=0;\tag{A1}\\
\forall n \in \mathbb{Z}, \quad u_{2n+1}(X)-Xu_{2n}(X)+u_{2n-1}(X)=0.\tag{A2}\\
\forall n \in \mathbb{Z},\quad u_{n+2}(X)-(X-2)u_{n+1}(X)+u_{n}(X)=0.\tag{AR}
\end{gather}
Les suites $n \mapsto (u_{2n}(X))$ et $n \mapsto (u_{2n+1}(X))$ forment chacune une base du $\mathbb{Z}$-module des solutions de la récurrence $(R)$.
\end{proposition}
\begin{proposition}
1) Nous avons:
\begin{gather}
\forall n \in \mathbb{Z}, \quad u_{2n+1}(4\cos^{2}\theta)=\frac{\sin (2n+1)\theta}{\sin \theta},\\
\forall n \in \mathbb{Z},\quad  u_{2n}(4\cos^{2}\theta)=\frac{\sin (2n)\theta}{\sin 2\theta},
\end{gather}
 2) Pour chaque entier $n$, les racines de $u_{2n+1}(X)$ dans $\mathbb{R}$ sont $4\cos^{2} \frac{k\pi}{2n+1}$ $(1 \leqslant k \leqslant n)$ et celles de $u_{2n}(X)$ sont  $4\cos^{2} \frac{k\pi}{2n}$ $(1 \leqslant k \leqslant n-1)$.
\end{proposition}
Nous allons maintenant obtenir une factorisation en facteurs irréductibles de $u_{n}(X)$ dans $\mathbb{Z}[X]$.

Soit $n$ un entier $\geqslant 1$. Nous définissons
\begin{itemize}
\item si $n$ est impair $P_{n}(X):=\frac{X^{n}-1}{X-1}$,
\item si $n$ est pair $P_{n}(X):=\frac{X^{n}-1}{X^{2}-1}$.
\end{itemize}
Alors les $P_{n}(X)$ sont des polynômes symétriques et
\begin{itemize}
\item si $n$ est impair, $n=2m+1$, $\frac{1}{X^{m}}P_{n}(X)$ est un polynôme en $X+X^{-1}$,
\item si $n$ est pair, $n=2m$, $\frac{1}{X^{m-1}}P_{n}(X)$ est un polynôme en $X+X^{-1}$.
\end{itemize}
Soit $\zeta$ une racine primitive n-ième de l'unité, par exemple $\zeta =\exp 2i\pi /n$, alors $\zeta$ est une racine de $P_{n}(X)$ car, pour $n>2$, $\zeta \ne -1$. Nous avons $\zeta + \zeta ^{-1} =2\cos 2\pi /n$ et $\gamma=\zeta + \zeta ^{-1}+2=4\cos^{2}\pi/n$.
Nous voyons ainsi que $\gamma$ est une racine de $u_{n}(X)$ et les autres racines de $u_{n}(X)$ sont obtenues en prenant les autres racines n-ièmes de l'unité ($\ne \pm 1$).
\begin{proposition}\label{a3}
Avec les hypothèses et notations précédentes, nous avons:
\begin{enumerate}
\item L'application $\delta:\zeta \mapsto \zeta + \zeta ^{-1}+2$ induit une bijection entre l'ensemble des racines de $P_{n}(X)$ et l'ensemble des racines de $u_{n}(X)$ (on définit $P_{0}(X):=0$).
\item $\delta$ induit une bijection $\Delta$ entre l'ensemble des facteurs unitaires de $P_{n}(X)$ dans $\mathbb{Z}[X]$ et l'ensemble des facteurs unitaires de $u_{n}(X)$ dans $\mathbb{Z}[X]$.
\item Soit $\Phi_{n}(X)$ le n-ième polynôme cyclotomique. On pose $v_{n}(X):=\Delta(\Phi_{n}(X))$. Alors:
\begin{center}
\[
P_{n}(X)=\prod_{d|n,n\geqslant 3}\Phi_{d}(X) \qquad \mbox{et} \qquad u_{n}(X)=\prod_{d|n,n\geqslant 3}v_{d}(X).
\]
\end{center}
\end{enumerate}
\end{proposition}
On appelle $v_{n}(X)$ le \emph{facteur primitif} de $u_{n}(X)$.\\
On définit une application $n\mapsto n':\mathbb{N}\to\mathbb{N}$ de la manière suivante:
$n'=2n$ si $n$ est impair; $n'=\frac{n}{2}$ si $n \,\equiv \,2 \pmod{4}$; $n'=n$ si $n \,\equiv \,0 \pmod{4}$.\\
Il est facile de vérifier que si $\mathbb{Q}(\zeta_{n})$ est le n-ième corps cyclotomique, alors $\mathbb{Q}(\zeta_{n'})=\mathbb{Q}(\zeta_{n})$.\\
Si $\alpha$ est une racine de $v_{n}(X)$, alors $\alpha'=4-\alpha$ est une racine de $v_{n'}(X)$ et l'application $\alpha \mapsto \alpha'$ établit une bijection entre l'ensemble des racines de $v_{n}(X)$ et l'ensemble des racines de $v_{n'}(X)$.

On donne maintenant des propriétés de la suite $(u_{n}(X))_{n \in \mathbb{Z}}$ dont on aura besoin dans la suite.\\
\begin{proposition}\label{a4}
On a les formules: $ \forall(n,m)\in \mathbb{Z}^2,$
\begin{gather}
  u_{n+2m}(X)+u_{n-2m}(X)=(u_{2m+1}(X)-u_{2m-1}(X))u_{n}(X) ;\\
u_{n+2m+1}(X)+u_{n-(2m+1)}(X) = X(u_{2m+2}(X)-u_{2m}(X))u_{n}(X) \text{si $n$ est pair}\\
u_{n+2m+1}(X)+u_{n-(2m+1)}(X) = (u_{2m+2}(X)-u_{2m}(X))u_{n}(X) \text{si $n$ est impair}.
\end{gather}
\end{proposition}
\begin{corollary}\label{a5}
On a les formules $\forall n \in \mathbb{Z}$:\\
\begin{align}
   u_{2n}(X) &=(u_{n+1}(X)-u_{n}(X))u_{n}(X); \\
  u_{2n}(X) &=u_{n+1}(X)(u_{n}(X)-u_{n-2}(X))-1\\ &=u_{n-1}(X)(u_{n+2}(X)-u_{n}(X))+1;\\
  \text{ si $n$ est pair}\notag\\ u_{2n+1}(X) &=u_{n+1}(X)(u_{n+1}(X)-u_{n-1}(X))-1\\ &=Xu_{n}(X)(u_{n+2}(X)-u_{n}(X))+1\\
\text{si  $n$ est impair}\notag\\ u_{2n+1}(X) & = Xu_{n+1}(X)(u_{n+1}(X)-u_{n-1}(X))-1\\ &=u_{n}(X)(u_{n+2}(X)-u_{n}(X))+1.
\end{align}
\end{corollary}
\begin{proof}
\end{proof}
Les résultats qui suivent seront utilisés dans différentes parties du travail.
\begin{proposition}\label{a6}
On a la formule:
\begin{equation}
 \forall n \in \mathbb{Z},\,u_{2n}(X)=(-1)^{n-1}u_{2n}(4-X).
\end{equation}
\end{proposition}
\begin{proof}
On procède par récurrence sur $n$ si $n$ est positif, la formule pour $n$ négatif se déduisant immédiatement de celle pour $n$ positif car $u_{-n}(X)=-u_{n}(X)$.
Si $n=0$, le résultat est vrai car $u_{0}(X)=0$. Si $n=1$, le résultat est vrai car $u_{2}(X)=1$. On suppose le résultat vrai pour tout $p$, ($0\leqslant p \leqslant n$).\\
On a la relation $(R)$: $u_{2n+2}(X)=(X-2)u_{2n}(X)-u_{2n-2}(X)$. Nous transformons de deux manières le deuxième membre de cette égalité.

Posons $Y:=4-X$. Nous obtenons 
\[
u_{2n+2}(4-Y)=(2-Y)u_{2n}(4-Y)-u_{2n-2}(4-Y),
\]
d'où: 
\[
u_{2n+2}(4-X)=(2-X)u_{2n}(X)-u_{2n-2}(4-X).
\]
D'un autre coté, d'après l'hypothèse de récurrence, nous avons:\\
$u_{2n+2}(X)  =  (X-2)(-1)^{n-1}u_{2n}(4-X)+(-1)^{n-1}u_{2n-2}(4-X)$\\
$u_{2n+2}(X)  =(-1)^{n-1}((X-2)u_{2n}(4-X)+u_{2n-2}(4-X))$\\
$u_{2n+2}(X)  =(-1)^{n}((2-X)u_{2n}(4-X)-u_{2n-2}(4-X))$\\
$u_{2n+2}(X)  =(-1)^{n}u_{2n+2}(4-X)$.

Le résultat est donc vrai pour tout $n$.
\end{proof}
\begin{corollary}
Si $p$ est impair, on a:
\begin{equation}
u_{2p}(X) =(-1)^{\frac{p-1}{2}}u_{p}(X)u_{p}(4-X). 
\end{equation}
\end{corollary}
\begin{proof}
On a $u_{2p}(X)=u_{p}(X)(u_{p+1}(X)-u_{p-1}(X))$ (formule (A5)). Si $p=2q+1$, alors $u_{p+1}(X)=u_{2q+2}(X)=(-1)^{q}u_{2q+2}(4-X)$ et $u_{p-1}(X)=u_{2q}(X)=(-1)^{q-1}u_{2q}(4-X)$ donc $u_{p+1}(X)-u_{p-1}(X)=(-1)^{q}(u_{2q+2})(4-X)+u_{2q}(4-X))=(-1)^{q}u_{2q+1}(4-X)$ d'après (A1).\\
Finalement, $u_{2p}(X)=(-1)^{\frac{p-1}{2}}u_{p}(X)u_{p}(4-X)$
\end{proof}
 \begin{corollary}
Si $\alpha$ est une racine de $v_{n}(X)$, alors $4-\alpha$ est une racine de $v_{n'}(X)$
\end{corollary}
\begin{proof}
C'est clair d'après la proposition 5 et les corollaires 1 et 2.
\end{proof}
\begin{proposition}\label{a7}
On a les formules: $ \forall(n,p)\in \mathbb{Z}^{2}$,

\begin{align}
  u_{2n}(X) &=u_{p}(X)u_{2n+1-p}(X)-u_{p-1}(X)u_{2n-p}(X). \\
  u_{2n+1}(X) &=u_{2p+1}(X)u_{2n+1-2p}(X)-Xu_{2p}(X)u_{2n-2p}(X)\notag\\
  &= Xu_{2p+2}(X)u_{2n-2p}(X)-u_{2p+1}(X)u_{2n-2p-1}(X) \\
 u_{2n-1}(X) &=u_{2p-1}(X)u_{2n+1-2p}(X)-Xu_{2p-2}(X)u_{2n-2p}(X)\notag\\
  &=  Xu_{2p}(X)u_{2n-2p}(X)-u_{2p-1}(X)u_{2n-2p-1}(X).
  \end{align}
\end{proposition}
\begin{proof}
1) On pose, pour cette démonstration seulement, 
\[
u_{2n,p}(X):=u_{p}(X)u_{2n+1-p}(X)-u_{p-1}(X)u_{2n-p}(X)
\]
et on montre que $u_{2n,p}(X)$ ne dépend pas de $p$. Nous distinguons deux cas suivant la parité de $p$. 
\begin{itemize}
\item Si $p$ est pair, on a $u_{p}(X)=u_{p-1}(X)-u_{p-2}(X)$ d'après (A1) donc
\begin{align*}
u_{2n,p}(X) & =  (u_{p-1}(X)-u_{p-2}(X))u_{2n+1-p}(X)-u_{p-1}(X)u_{2n-p}(X)\\
&=  u_{p-1}(X)(u_{2n+1-p}(X)-u_{2n-p}(X))-u_{p-2}(X)u_{2n+1-p}(X)\\
&=  u_{p-1}(X)u_{2n+2-p}(X)-u_{p-2}(X)u_{2n+1-p}(X)\\
&=  u_{2n,p-1}(X)
\end{align*}
en utilisant (A1).
\item Si $p$ est impair, on a $u_{p}(X)=Xu_{p-1}(X)-u_{p-2}(X)$ d'après (A2) donc
\begin{align*}
u_{2n,p}(X) & =  (Xu_{p-1}(X)-u_{p-2}(X))u_{2n+1-p}(X)u_{p-1}(X)u_{2n-p}(X)\\
&=  u_{p-1}(X)(Xu_{2n+1-p}(X)-u_{2n-p}(X))-u_{p-2}(X)u_{2n+1-p}(X)\\
&=  u_{p-1}(X)u_{2n+2-p}(X)-u_{p-2}(X)u_{2n+1-p}(X)\\
&=  u_{2n,p-1}(X)
\end{align*}
en utilisant (A2).
\end{itemize}
Il en résulte que $u_{2n,p}(X)$ ne dépend pas de $p$. Comme $u_{2n,0}(X)=u_{0}(X)u_{2n+1}(X)-u_{-1}(X)u_{2n}(X)=u_{2n}(X)$ car $u_{0}(X)=0$ et $-u_{-1}(X)=u_{1}(X)=1$, nous avons le résultat.\\
2) On pose, pour cette démonstration seulement,
\[
u_{2n+1,2q}(X):=u_{2q+1}(X)u_{2n+1-2q}(X)-Xu_{2q}(X)u_{2n-2q}(X)
\]
et
\[
u_{2n+1,2q-1}(X):=Xu_{2q}(X)u_{2n+2-2q}(X)-u_{2q-1}(X)u_{2n+1-2q}(X).
\]
Comme $u_{2q+1}(X)=Xu_{2q}(X)-u_{2q-1}(X)$, nous obtenons
\begin{align*}
u_{2n+1,2q}(X) &= Xu_{2q}(X)(u_{2n+1-2q}(X)-u_{2n-2q}(X))-u_{2q-1}(X)u_{2n+1-2q}(X)\\
&=  Xu_{2q}(X)u_{2n+2-2q}(X)-u_{2q-1}(X)u_{2n+1-2q}(X)\\
&=  u_{2n,2q-1}.
\end{align*}
Comme $u_{2q}(X)=u_{2q-1}(X)-u_{2q-2}(X)$, nous obtenons
\begin{align*}
u_{2n+1,2q-1}(X) &=  X(u_{2q-1}(X)-u_{2q-2}(X))u_{2n+2-2q}(X)-u_{2q-1}(X)u_{2n+1-2q}(X)\\
&=  u_{2q-1}(X)(Xu_{2n+2-2q}(X)-u_{2n+1-2q}(X))-Xu_{2q-2}(X)u_{2n+2-2q}(X)\\
&=  u_{2q-1}(X)u_{2n+3-2q}(X)-Xu_{2q-2}(X)u_{2n+2-2q}(X)\\
&=  u_{2n+1,2(q-1)}(X).
\end{align*}
Il en résulte que $u_{2n+1,p}(X)$ ($p$ pair ou impair) ne dépend pas de $p$. Nous avons 
\[
u_{2n+1,0}(X)=u_{1}(X)u_{2n+1}(X)-Xu_{0}(X)u_{2n}(X)=u_{2n+1}(X)
\]
d'où le résultat.\\
3) A partir des formules (19), on change $n$ en $n-1$ et $p$ en $p-2$ pour obtenir les formules (20).
\end{proof}
\begin{corollary}\label{a8}
On a les formules: $\forall n \in \mathbb{Z}$,
\begin{align}
u_{4n+1}(X)&=u_{2n+1}^{2}(X)-Xu_{2n}^{2}(X)\notag\\ &=Xu_{2n+2}(X)u_{2n}(X)-u_{2n+1}(X)u_{2n-1}(X)\\
u_{4n-1}(X)&=Xu_{2n}^{2}(X)-u_{2n-1}^{2}(X)\notag\\&=u_{2n+1}(X)u_{2n-1}(X)-Xu_{2n}(X)u_{2n-2}(X).
\end{align}
\end{corollary}
\begin{proof}
Pour (21), on prend $n=2p$ dans (19). La formule (22) s'obtient à partir de (21) en changeant $n$ en $-n$.
\end{proof}
\begin{proposition}\label{a9}
On a les formules: $\forall n \in \mathbb{Z}$,
\begin{align}
u_{2n+1}^{2}(X)-1&=Xu_{2n}(X)u_{2n+2}(X)\\
Xu_{2n}^{2}(X)-1&=u_{2n-1}(X)u_{2n+1}(X).
\end{align}
\end{proposition}
\begin{proof}
La formule (24) est conséquence immédiate des formules (21) et (23).\\
Pour démontrer (23), on procède par récurrence sur $n$ si $n\geqslant 0$.\\
Si $n=0$, $1^{2}-1=0=Xu_{0}(X)u_{2}(X)$ car $u_{0}(X)=0$.\\
Supposons (23) vraie pour $n$: $u_{2n+1}^{2}(X)-1=Xu_{2n}(X)u_{2n+2}(X)$. Nous avons:
\begin{align*}
u_{2n+3}^{2}(X)-1 &=  (Xu_{2n+2}(X)-u_{2n+1}(X))^{2}-1\\
&=  X^{2}u_{2n+2}^{2}(X)-2Xu_{2n+2}(X)u_{2n+1}(X)+u_{2n+1}^{2}(X)-1\\
&=  X^{2}u_{2n+2}^{2}(X)-2Xu_{2n+2}(X)u_{2n+1}(X)+Xu_{2n}(X)u_{2n+2}(X)\\
&=  Xu_{2n+2}(X)(Xu_{2n+2}(X)-2u_{2n+1}(X)+u_{2n}(X))\\
&=  Xu_{2n+2}(X)(Xu_{2n+2}(X)-u_{2n+1}(X)-(u_{2n+1}(X)-u_{2n}(X)))\\
&=  Xu_{2n+2}(X)(u_{2n+3}(X)-u_{2n+2}(X))\\
&=  Xu_{2n+2}(X)u_{2n+4}(X)
\end{align*}
en utilisant $(A1)$ et $(A2)$.\\
Comme $u_{-m}(X)=-u_{m}(X)$, on vérifie sans peine que(23) est vraie pour $n$ négatif.
\end{proof}
\begin{corollary}\label{a10}
On a les formules: $\forall n \in \mathbb{Z}$,
\begin{align}
u_{4n+1}(X)-u_{4n-1}(X)&=2-X(4-X)u_{2n}^{2}(X);\\
u_{4n+3}(X)-u_{4n+1}(X)&=2-(4-X)u_{2n+1}^{2}(X).
\end{align}

\end{corollary}
\begin{proof}
D'après les formules (21) et (22) nous avons:
\begin{align*}
u_{4n+1}(X)-u_{4n-1}(X) &=  Xu_{2n}(X)(u_{2n+2}(X)+u_{2n-2}(X))-2u_{2n+1}(X)u_{2n-1}(X)\\
 &=  X(X-2)u_{2n}^{2}(X)-2u_{2n+1}(X)u_{2n-1}(X)
\end{align*}
en utilisant la relation $(AR)$.\\
D'après (24), on a $2u_{2n+1}(X)u_{2n-1}(X)=2Xu_{2n}^{2}(X)-2$ donc\\
\[
u_{4n+1}(X)-u_{4n-1}(X)=2+X(X-4)u_{2n}^{2}(X)=2-X(4-X)u_{2n}^{2}(X).
\]
De même nous avons:
\begin{align*}
u_{4n+3}(X)-u_{4n+1}(X) &=  u_{2n+3}(X)u_{2n+1}(X)+u_{2n+1}(X)u_{2n-1}(X)-2Xu_{2n+2}(X)u_{2n}(X)\\
 &=  u_{2n+1}(X)(u_{2n+3}(X)+u_{2n+1}(X))-2(u_{2n+1}^{2}(X)-1)\\
 &=  2-(4-X)u_{2n+1}^{2}(X)
\end{align*}
en procédant comme ci-dessus.
\end{proof}
\begin{proposition}\label{a11b}
On a les formules: $\forall n \in \mathbb{Z}$:
Si $n$ est pair
\begin{equation}
  u_{n-1}(X)u_{2n}(X)-u_{n}(X)u_{2n-1}(X) = -u_{n}(X). 
\end{equation}
\begin{equation}
Xu_{n}(X)u_{2n}(X)-u_{n+1}(X)u_{2n-1}(X)=-u_{n-1}(X).
\end{equation}
Si $n$ est impair
\begin {equation}
 Xu_{n-1}(X)u_{2n}(X)-u_{n}(X)u_{2n-1}(X)=-u_{n}(X), 
\end{equation}
\begin{equation}
 u_{n}(X)u_{2n}(X)-u_{n+1}(X)u_{2n-1}(X)=-u_{n-1}(X). 
\end{equation}
\end{proposition}
\begin{proof}
1) on suppose que $n$ est pair.\\ Posons, pour cette démonstration seulement,
\[
A:=u_{n-1}(X)u_{2n}(X)-u_{n}(X)u_{2n-1}(X).
\]
Comme $u_{2n}(X)=u_{n}(X)(u_{n+1}(X)-u_{n-1}(X))$, nous obtenons
\[
A=u_{n}(X)(u_{n+1}(X)u_{n-1}(X)-u_{n-1}^{2}(X)-u_{2n-1}(X)).
\]
D'après (22), $u_{2n-1}(X)+u_{n-1}^{2}(X)=Xu_{n}^{2}(X)$ donc \\
$A=u_{n}(X)(u_{n+1}(X)u_{n-1}(X)-Xu_{n}^{2}(X))=-u_{n}(X)$.\\
Posons:
\[
B:=Xu_{n}(X)u_{2n}(X)-u_{n+1}(X)u_{2n-1}(X).
\]Alors:
\begin{align*}
B &=  Xu_{n}^{2}(X)(u_{n+1}(X)-u_{n-1}(X))-u_{n+1}(X)u_{2n-1}(X)\\
&=  u_{n+1}(X)(Xu_{n}^{2}(X)-u_{2n-1}(X))-Xu_{n}^{2}(X)u_{n-1}(X)\\
&=  u_{n+1}(X)u_{n-1}^{2}(X)-Xu_{n}^{2}(X)u_{n-1}(X)\\
&=  u_{n-1}(X)(u_{n+1}(X)u_{n-1}(X)-Xu_{n}^{2}(X))\\
&=  -u_{n-1}(X).
\end{align*}
2) On suppose que $n$ est impair.\\
Posons:
\[
C:=u_{n-1}(X)u_{2n}(X)-u_{n}(X)u_{2n-1}(X).
\]Alors:
\begin{align*}
C &=  Xu_{n-1}(X)u_{n}(X)(u_{n+1}(X)-u_{n-1}(X))-u_{n}(X)u_{2n-1}(X)\\
&=  u_{n}(X)(Xu_{n-1}(X)u_{n+1}(X)-Xu_{n-1}^{2}(X)-u_{2n-1}(X))\\
&=  u_{n}(X)(-1+u_{n}^{2}(X)-Xu_{n-1}^{2}(X)-u_{2n-1}(X))\\
&=  -u_{n}(X).
\end{align*}
Posons:
\[
D:=u_{n}(X)u_{2n}(X)-u_{n+1}(X)u_{2n-1}(X).
\]Alors:
\begin{align*}
D &=  u_{n}^{2}(X)(u_{n+1}(X)-u_{n-1}(X))-u_{n+1}(X)u_{2n-1}(X)\\
&=  u_{n+1}(X)(u_{n}^{2}(X)-u_{2n-1}(X))-u_{n}^{2}(X)u_{n-1}(X)\\
&=  u_{n+1}(X)Xu_{n-1}^{2}(X)-u_{n}^{2}(X)u_{n-1}(X)\\
&=  u_{n-1}(X)(Xu_{n+1}(X)u_{n-1}(X)-u_{n}^{2}(X))\\
&=  -u_{n-1}(X)
\end{align*}
en utilisant (23) et (21).
\end{proof}
\subsection{Propriétés des racines de $v_{p}(X)$.}
Soit $n$ un entier et soit $v_{n}(X)$ le facteur primitif de $u_{n}(X)$.
\begin{itemize}
\item Si $n=2p+1$, les racines de $v_{n}(X)$ sont les $4\cos^{2} \frac{k\pi}{2p+1}$ avec $1 \leqslant k \leqslant p$ et $(k,2p+1)=1$;
\item si $n=2p$, les racines de $v_{n}(X)$ sont les $4\cos^{2} \frac{k\pi}{2p}$ avec $1 \leqslant k \leqslant p-1$ et $(k,2p)=1$.
\end{itemize}
\begin{proposition}\label{a13}
Soient $p$ un entier $\geqslant 2$ et $\gamma$ une racine de $v_{2p}(X)$. Alors:
\begin{enumerate}
\item $\forall k, 0\leqslant k \leqslant p, u_{2p-k}(\gamma)=u_{k}(\gamma);$
\item $\forall k, 0\leqslant k <p,\forall l \in \mathbb{N}, u_{2lp+k}(\gamma)=(-1)^{l}u_{k}(\gamma)$.
\end{enumerate}

\end{proposition}
\begin{proof}
1) On a $u_{2p}(\gamma))=u_{p}(\gamma)(u_{p+1}(\gamma)-u_{p-1}(\gamma))$ (formule (9)), donc comme $u_{p}(\gamma)\ne 0$ et $u_{2p}(\gamma)=v_{2p}(\gamma)=0$, nous obtenons $u_{p+1}(\gamma)=u_{p-1}(\gamma)$.\\
Pour obtenir le résultat, nous effectuons une récurrence descendante sur $k$. Si $k=p$, $2p-k=p$ et si $k=p-1$, c'est la remarque initiale.\\
Supposons le résultat vrai pour tout $l$, $k+1 \leqslant l \leqslant p$: $u_{2p-l}(\gamma)=u_{l}(\gamma)$.
\begin{itemize}
\item Si $k$ est pair,
\begin{align*}
u_{2p-k}(\gamma) &=  u_{2p-(k+1)}(\gamma)-u_{2p-(k+2)}(\gamma)\\
&=  u_{k+1}(\gamma)-u_{k+2}(\gamma)\\
&=  u_{k}(\gamma)
\end{align*}
en utilisant $(A1)$.
\item Si $k$ est impair,
\begin{align*}
u_{2p-k}(\gamma) &=  \gamma u_{2p-(k+1)}(\gamma)-u_{2p-(k+2)}(\gamma)\\
&=  \gamma u_{k+1}(\gamma)-u_{k+2}(\gamma)\\
&=  u_{k}(\gamma)
\end{align*}
en utilisant $(A2)$.

\end{itemize}
2) On voit ainsi que $u_{2p-1}(\gamma)=u_{1}(\gamma)=1$ et, comme $0=u_{2p+1}(\gamma)-\gamma u_{2p}(\gamma)+u_{2p-1}(\gamma)$, nous obtenons $u_{2p+1}(\gamma)=-1$.\\
D'après les formules(6) et (7), nous avons $\forall m \in \mathbb{Z},u_{p+2m}(\gamma)+u_{p-2m}(\gamma)=0$.\\
Si $0<m<2p$, nous en déduisons que $u_{2p+m}(\gamma)=-u_{2p-m}(\gamma)=-u_{m}(\gamma)$.\\
Nous supposons que $2p$ ne divise pas $m$ et que $2p<m$. Divisons $m$ par $2p$: $m=2qp+r$ avec $0<r<2p$. Nous obtenons alors:
\[
u_{2p+m}(\gamma)=u_{2p(q+1)+r}(\gamma)=-u_{2p-2qp-r}(\gamma)=u_{2p(q-1)+r}(\gamma).
\]
Si $q$ est pair, nous aurons $u_{2p+m}(\gamma)=u_{2p+r}(\gamma)=-u_{r}(\gamma)$ tandis que si $q$ est impair, nous aurons $u_{2p+m}(\gamma)=u_{r}(\gamma)$.\\
Donc $u_{2p(q+1)}(\gamma)=(-1)^{q+1}u_{r}(\gamma)$. comme $u_{2p-r}(\gamma)=u_{r}(\gamma)$ d'après le 1), nous avons le résultat.
\end{proof}
\begin{proposition}\label{a14}
Soit  $\gamma$ une racine de $v_{2p}(X)$. Alors:\\
1) Si $p$ est impair, on a pour $k \in \{0,1,\ldots,p\}$, 
\begin{equation}
\left (\frac{4- \gamma}{2}\right )u_{p}(\gamma)u_{p-k}(\gamma)=u_{k+1}(\gamma)-u_{k-1}(\gamma).
\end{equation}
En particulier $(4-\gamma)u_{p}^{2}(\gamma)=4$ et $(4-\gamma)u_{p}(\gamma)u_{p-1}(\gamma)=2$.\\
2) Si $p$ est pair, on a pour $0\leqslant k < \frac{p}{2}$,
\begin{equation}
\frac{\gamma (4-\gamma)}{2}u_{p}(\gamma)u_{p-2k}(\gamma)=u_{2k+1}(\gamma)-u_{2k-1}(\gamma)
\end{equation}
et 
\begin{equation} 
\left (\frac{4-\gamma}{2}\right )u_{p}(\gamma)u_{p-(2k+1)}(\gamma)=u_{2k+2}(\gamma)-u_{2k}(\gamma)
\end{equation}
En particulier $\gamma(4-\gamma)u_{p}^{2}(\gamma)=4$ et $(4-\gamma)u_{p}(\gamma)u_{p-1}(\gamma)=2$.

\end{proposition}
\begin{proof}
1) Supposons $p$ impair. Nous montrons la relation par récurrence sur $k$.\\
On a $u_{p-1}(\gamma)=u_{p+1}(\gamma)$ et, comme $p$ est impair, $u_{p}(\gamma)=2u_{p-1}(\gamma)$ en utilisant $(A1)$.\\
Maintenant $4u_{p}^{2}(\gamma)-4=4\gamma u_{p-1}(\gamma)u_{p+1}(\gamma)$ (formule (23)), d'où $4u_{p}^{2}(\gamma)-4=\gamma u_{p}^{2}(\gamma)$ et nous avons le résultat $(4-\gamma)u_{p}^{2}(\gamma)=4$.\\
Comme $u_{1}(\gamma)-u_{-1}(\gamma)=2$, c'est la formule cherchée pour $k=0$.\\
Remplaçons $u_{p}(\gamma)$ par $2u_{p-1}(\gamma)$dans la formule ci-dessus\\ pour obtenir $(4-\gamma)u_{p}(\gamma)u_{p-1}(\gamma)=2$.\\
Comme $u_{2}(\gamma)-u_{0}(\gamma)=1$, nous obtenons le résultat pour $k=1$.\\
Nous supposons maintenant le résultat vrai jusqu'à $k$.\\
 Si $k$ est pair, $p-(k+1)$ est pair, donc $u_{p-(k+1)}(\gamma)=u_{p-k}(\gamma)-u_{p-(k-1)}(\gamma)$, d'où
\begin{align*}
\left (\frac{4-\gamma}{2}\right )u_{p}(\gamma)u_{p-(k+1)}(\gamma) &=  \left (\frac{4-\gamma}{2}\right )u_{p}(\gamma)(u_{p-k}(\gamma)-u_{p-(k-1)}(\gamma))\\
&=  (u_{k+1}(\gamma)-u_{k-1}(\gamma))-(u_{k}(\gamma)-u_{k-2}(\gamma))\\
&= u_{k+2}(\gamma)-u_{k}(\gamma).
\end{align*}
 Si $k$ est impair, $p-(k+1)$ est impair, donc $u_{p-(k+1)}(\gamma)=\gamma u_{p-k}(\gamma)-u_{p-(k-1)}(\gamma)$, d'où
\begin{align*}
\left (\frac{4-\gamma}{2}\right )u_{p}(\gamma)u_{p-(k+1)}(\gamma) &=  \left (\frac{4-\gamma}{2}\right )u_{p}(\gamma)(\gamma u_{p-k}(\gamma)-u_{p-(k-1)}(\gamma))\\
&= (\gamma u_{k+1}(\gamma)-\gamma u_{k-1}(\gamma))-(u_{k}(\gamma)-u_{k-2}(\gamma))\\
&=  u_{k+2}(\gamma)-u_{k}(\gamma).
\end{align*}
Le résultat est donc vrai pour tout $k$.\\
2) Supposons $p$ pair. Comme ci-dessus on a $u_{p-1}(\gamma)=u_{p+1}(\gamma)$ et, comme $p$ est pair, $\gamma u_{p}(\gamma)=2u_{p-1}(\gamma)$.\\
Maintenant $4\gamma u_{p}^{2}(\gamma)-4=\gamma^{2}u_{p}^{2}(\gamma)$ et nous avons le résultat: $\gamma(4-\gamma)u_{p}^{2}(\gamma)=4$. C'est la formule annoncée pour $k=0$.\\
Pour obtenir le cas $k=1$ , on remplace $\gamma u_{p}(\gamma)$ par $2u_{p-1}(\gamma)$ dans la formule ci-dessus.

Nous montrons les relations par récurrence sur $k$. Nous supposons maintenant le résultat vrai jusqu'à $k-1$.\\
Nous avons $u_{p-2k}(\gamma)=u_{p-(2k-1)}(\gamma)-u_{p-(2k-2)}(\gamma)$ donc
\begin{align*}
u_{p}(\gamma)u_{p-2k}(\gamma) &=  u_{p}(\gamma)u_{p-(2k-1)}(\gamma)-u_{p}(\gamma)u_{p-(2k-2)}(\gamma)\\
&=  \left (\frac{2}{4-\gamma}\right )(u_{2k}(\gamma)-u_{2k-2}(\gamma))-\frac{2}{\gamma (4-\gamma)}(u_{2k-1}(\gamma)-u_{2k-3}(\gamma))\\
&=  \frac{2}{\gamma (4-\gamma)}((\gamma u_{2k}(\gamma)-u_{2k-1}(\gamma))-(\gamma u_{2k-2}(\gamma)-u_{2k-3}(\gamma)))\\
&=  \frac{2}{\gamma (4-\gamma)}(u_{2k+1}(\gamma)-u_{2k-1}(\gamma)).
\end{align*}
Nous avons $u_{p-(2k+1)}(\gamma)=\gamma u_{p-2k}(\gamma)-u_{p-(2k-1)}(\gamma)$ donc
\begin{align*}
u_{p}(\gamma)u_{p-(2k+1)}(\gamma)&= \left (\frac{2}{4-\gamma}\right )(u_{2k+1}(\gamma)-u_{2k-1}(\gamma))-\left (\frac{2}{4-\gamma}\right )(u_{2k}(\gamma)-u_{2k-2}(\gamma))\\
&=  \left (\frac{2}{4-\gamma}\right )((u_{2k+1}(\gamma)-u_{2k}(\gamma))-(u_{2k-1}(\gamma)-u_{2k-2}(\gamma)))\\
&=  \left (\frac{2}{4-\gamma}\right )(u_{2k+2}(\gamma)-u_{2k}(\gamma)).
\end{align*}
Le résultat est donc vrai pour tout $k$.

\end{proof}
\begin{proposition}\label{a15}
Soit $\gamma$ une racine de $v_{r}(X)$ avec $r=2r_{1}+1$. :\\
	1) Pour $0\leqslant k \leqslant r_{1}-1$ on a 
	\begin{eqnarray*}
u_{r-(2k+1)}(\gamma) & = & u_{2k+1}(\gamma)u_{r-1}(\gamma);\\
u_{r-(2k+2)}(\gamma) & = & \gamma u_{2k+2}(\gamma)u_{r-1}(\gamma).
\end{eqnarray*}\\
	2) On se place dans un sous-corps de $\mathbb{R}$. alors si $\gamma=4\cos^{2}\frac{k\pi}{r}$\\
	(i) on a $\sqrt{\gamma}u_{r-1}(\gamma)=(-1)^{k-1}$ où $\sqrt{\gamma}>0$.\\
	(ii)Supposons que $r\equiv 1\pmod{4}$: $r=4r_{2}+1$, alors on a pour $0\leqslant l\leqslant r_{2}$:
\begin{eqnarray}
u_{\frac{r-1}{2}-2l}(\gamma) & = & (u_{2l+1}(\gamma)-(-1)^{(k-1)}\sqrt{\gamma}u_{2l}(\gamma))u_{\frac{r-1}{2}}(\gamma);\\
u_{\frac{r-1}{2}-(2l-1)}(\gamma) & = &( \gamma u_{2l}(\gamma)-(-1)^{(k-1)}\sqrt{\gamma}u_{2l-1}(\gamma))u_{\frac{r-1}{2}}(\gamma).
\end{eqnarray}
Supposons que $r\equiv 3\pmod{4}$: $r=4r_{2}+3$, alors on a pour $0\leqslant l\leqslant \frac{r-3}{8}$:
\begin{eqnarray}
u_{\frac{r-1}{2}-2l}(\gamma) & = & (u_{2l+1}(\gamma)-(-1)^{(k-1)}\sqrt{\gamma}u_{2l}(\gamma))u_{\frac{r-1}{2}}(\gamma);\\
u_{\frac{r-1}{2}-(2l-1)}(\gamma) & = &( u_{2l}(\gamma)-(-1)^{(k-1)}\frac{1}{\sqrt{\gamma}}u_{2l-1}(\gamma))u_{\frac{r-1}{2}}(\gamma).
\end{eqnarray}

\end{proposition}
\begin{proof}
1) On procède par récurrence sur $k$ en utilisant les relations $(A1)$ et $(A2)$. Si $k=0$, $u_{r-1}(\gamma)=u_{1}(\gamma)u_{r-1}(\gamma)$ car $u_{1}(X)=1$; $u_{r-2}(\gamma)=\gamma u_{r-1}(\gamma)-u_{r}(\gamma)=\gamma u_{r-1}(\gamma)$ car $u_{r}(\gamma)=0$.\\
Supposons le résultat vrai pour $k\geqslant 1$. On a 
\begin{eqnarray*}
u_{r-(2k+3)}(X) & = & u_{r-(2k+2)}(X)-u_{r-(2k+1)}\, \text{d'après} \,(A1)\, \text{donc}\\
u_{r-(2k+3)}(\gamma) & = & \gamma u_{2k+2}(\gamma)u_{r-1}(\gamma)-u_{2k+1}(\gamma)u_{r-1}(\gamma)\\
& = & (\gamma u_{2k+2}(\gamma)-u_{2k+1}(\gamma))u_{r-1}(\gamma)\\
& = & u_{2k+3}(\gamma)u_{r-1}(\gamma)\, \text{d'après} \,(A2).
\end{eqnarray*}
\begin{eqnarray*}
u_{r-(2k+4)}(X) & = & Xu_{r-(2k+3)}(X)-u_{r-(2k+2)}(X)\, \text{d'après} \,(A2)\, \text{donc}\\
u_{r-(2k+4)}(\gamma) & = & \gamma u_{r-(2k+3)}(\gamma)-u_{r-(2k+2)}(\gamma)\\
& = & \gamma u_{2k+3}(\gamma)u_{r-1}(\gamma)-\gamma u_{2k+2}(\gamma)u_{r-1}(\gamma)\\
& = & \gamma(u_{2k+3}(\gamma)-u_{2k+2}(\gamma))u_{r-1}(\gamma)\\
& = & \gamma u_{2k+4}(\gamma)u_{r-1}(\gamma)\, \text{d'après} \,(A1).
\end{eqnarray*}
Le résultat est donc vrai pour tout $k$.\\

2) On suppose que nous sommes dans un sous-corps de $\mathbb{R}$.\\
(i) Les racines de $u_{r-1}(X)$ encadrent les racines de $u_{r}(X)$. Plus précisément on a 
\[
\forall l,\quad \frac{2l\pi}{r-1}<\frac{(2l+1)\pi}{r}<\frac{(2l+1)\pi}{r-1}
\]
d'où, comme l'application $\cos$ est décroissante:
\[
4\cos^{2}\frac{(2l+1)\pi}{r-1}<4\cos^{2}\frac{(2l+1)\pi}{r}<4\cos^{2}\frac{2l\pi}{r-1}.
\]
Le polynôme $u_{r-1}(X)$ n'a que des racines simples et $u_{r-1}(4)>0$, nous obtenons:
\begin{eqnarray*}
u_{r-1}(a)>0 & \text{si} & a \in \left ]4\cos^{2}\frac{(2l+1)\pi}{r-1},4\cos^{2}\frac{2l\pi}{r-1}\right [;\\
u_{r-1}(a)<0) & \text{si} & a \in \left ]4\cos^{2}\frac{(2l+2)\pi}{r-1},4\cos^{2}\frac{(2l+1)\pi}{r-1}\right [.
\end{eqnarray*}
D'après la formule (25) nous avons $\gamma u_{r-1}^{2}(\gamma)=1$ et de ce qui précède, nous déduisons que $\sqrt{\gamma}u_{r-1}(\gamma)=(-1)^{k-1}$.\\
(ii) 
La démonstration se fait par récurrence sur $l$ et n'offre pas de difficultés en utilisant la relation $(\mathcal{R})$

\end{proof}
\begin{proposition}\label{a16}
Soient $K$ un corps et $K'$ un sur-corps de $K$. Soit $\varphi$ un élément de $K^{*}$. Alors:\\
1) Si $a\in K'-K$ satisfait à $a^{2}=\varphi a+1$, on a $\forall n \in \mathbb{Z}$,:
\begin{align*}
  a^{2n}  &=  (-1)^{n-1}\varphi u_{2n}(-\varphi^{2})a+(-1)^{n-1}u_{2n-1}(-\varphi^{2}),\\
 a^{2n+1}  &=  (-1)^{n}u_{2n+1}(-\varphi^{2})a+(-1)^{n-1}\varphi u_{2n}(-\varphi^{2}).
\end{align*}
En particulier, $a^{n} \in K$ si et seulement si $u_{n}(-\varphi^{2})=0$.\\
2) Si $a\in K'-K$ satisfait à $a^{2}=\varphi a-1$, on a $\forall n \in \mathbb{Z}$,:
\begin{align*}
 a^{2n} &= \varphi u_{2n}(\varphi^{2})a-u_{2n-1}(\varphi^{2}),\\
 a^{2n+1} &= u_{2n+1}(\varphi^{2})a-\varphi u_{2n}(\varphi^{2}).
\end{align*}
En particulier, $a^{n} \in K$ si et seulement si $u_{n}(\varphi^{2})=0$.\\
3) On suppose que $K$ est un sous-corps de $\mathbb{C}$ et que $a\in K'-K$ satisfait à $a^{2}=\varphi a+1$. Alors:
\begin{itemize}
\item Si $-\varphi^{2}$ est racine de $v_{2n+1}(X)$, il existe $\epsilon \in \{-1,+1\}$ et un entier $k$, premier à $2n+1$, tels que $\varphi =2\epsilon i \cos \frac{k\pi}{2n+1}$ et l'on a 
\[
a^{2n+1}=(-1)^{n+k}\epsilon i,\, a^{2(2n+1)}=-1\, \mbox{et} \; a^{4(2n+1)}=1.
\]
\item Si $-\varphi^{2}$ est racine de $v_{2n}(X)$, il existe $\epsilon \in \{-1,+1\}$ et un entier $k$, premier à $2n$ tels que $\varphi =2\epsilon \cos \frac{k\pi}{2n}$ et l'on a $a^{2n}=(-1)^{n-1}$.
\end{itemize}
4) On suppose que $K$ est un sous-corps de $\mathbb{R}$ et que $a\in K'-K$ satisfait à $a^{2}=\varphi a-1$. Alors:
\begin{itemize}
\item Si $\varphi^{2}$ est racine de $v_{2n+1}(X)$, il existe $\epsilon \in \{-1,+1\}$ et un entier $k$, premier à $2n+1$, tels que $\varphi =2\epsilon \cos \frac{k\pi}{2n+1}$ et l'on a 
\[
a^{2n+1}=\epsilon(-1)^{k} \quad \mbox{et} \quad a^{2(2n+1)}=1.
\]
\item Si $\varphi^{2}$ est racine de $v_{2n}(X)$, il existe $\epsilon \in \{-1,+1\}$ et un entier $k$, premier à $2n$, tels que $\varphi =2\epsilon \cos \frac{k\pi}{2n}$ et l'on a 
\[
a^{2n}=-1 \quad \mbox{et} \quad a^{4n}=1.
\]
\end{itemize}

\end{proposition}
\begin{proof}
1) On suppose d'abord $n\geqslant 0$. La preuve se fait par récurrence sur $n$.\\
Si $n=0$, $u_{0}(-\varphi^{2})=0$, $u_{-1}(-\varphi^{2})=-1$, $u_{1}(-\varphi^{2})=1$ d'où le résultat dans ce cas.\\
Si $n=1$, nous avons $a^{2}=\varphi u_{2}(-\varphi^{2})a+u_{1}(-\varphi^{2})=\varphi a+1$ et\\
$a^{3}=\varphi a^{2}+a=\varphi(\varphi a+1)+a=(\varphi^{2}+1)a+\varphi=-u_{3}(-\varphi^{2})a+\varphi u_{2}(-\varphi^{2})$ car $u_{3}(X)=X-1$ et $u_{2}(X)=1$.\\
Supposons que l'on ait $a^{2n}=(-1)^{n-1}\varphi u_{2n}(-\varphi^{2})a+(-1)^{n-1}u_{2n-1}(-\varphi^{2})$. Alors:
\begin{align*}
a^{2n+1} &=  (-1)^{n-1}\varphi u_{2n}(-\varphi^{2})(\varphi a+1)+(-1)^{n-1}u_{2n-1}(-\varphi^{2})a\\
&= (-1)^{n-1}(\varphi^{2}u_{2n}(-\varphi^{2})+u_{2n-1}(-\varphi^{2}))a+(-1)^{n-1}\varphi u_{2n}(-\varphi^{2})\\
&=  (-1)^{n}((-\varphi^{2})u_{2n}(-\varphi^{2})-u_{2n-1}(-\varphi^{2}))a+(-1)^{n-1}\varphi u_{2n}(-\varphi^{2})\\
&=  (-1^{n}u_{2n+1}(-\varphi^{2})a+(-1)^{n-1}\varphi u_{2n}(-\varphi^{2})
\end{align*}
en utilisant (A2).\\
Supposons que l'on ait $a^{2n+1}=(-1)^{n}\varphi u_{2n+1}(-\varphi^{2})a+(-1)^{n-1}\varphi u_{2n}(-\varphi^{2})$. Alors:
\begin{align*}
a^{2n+2} &=  (-1)^{n}\varphi u_{2n+1}(-\varphi^{2})(\varphi a+1)+(-1)^{n-1}\varphi u_{2n}(-\varphi^{2})a\\
&=  (-1)^{n}\varphi(u_{2n+1}(-\varphi^{2})-u_{2n}(-\varphi^{2}))a+(-1)^{n}u_{2n+1}(-\varphi^{2})\\
&=  (-1)^{n}\varphi u_{2n+2}(-\varphi^{2})a+(-1)^{n}u_{2n+1}(-\varphi^{2})
\end{align*}
en utilisant (A1).\\
Le résultat est donc vrai pour tout entier positif ou nul.\\
Dans l'expression de $a^{2n}$, nous changeons $n$ en $-n$ pour obtenir:
\[
b:=(-1)^{n}\varphi u_{2n}(-\varphi^{2})a+(-1)^{n}u_{2n+1}(-\varphi^{2})
\]
car par définition, $u_{-n}(X)=-u_{n}(X)$. nous avons alors:
\begin{multline*}
ba^{2n}=(-1)(\varphi^{2}u_{2n}^{2}(-\varphi^{2})a^{2}+\varphi u_{2n}(-\varphi^{2})(u_{2n+1}(-\varphi^{2})+u_{2n-1}(-\varphi^{2}))a\\+u_{2n+1}(-\varphi^{2})u_{2n-1}(-\varphi^{2}))
\end{multline*}
Nous avons $a^{2}=\varphi a+1$,  $u_{2n+1}(-\varphi^{2})+u_{2n-1}(-\varphi^{2})=-\varphi^{2}u_{2n}(-\varphi^{2})$ d'après $(A2)$, et $u_{2n+1}(-\varphi^{2})u_{2n-1}(-\varphi^{2}=-1-\varphi^{2}u_{2n}^{2}(-\varphi^{2})$ d'après $(A15)$ donc
\[
ba^{2n}=(-1)(\varphi^{3}u_{2n}^{2}(-\varphi^{2})a+\varphi^{2}u_{2n}^{2}(-\varphi^{2})-\varphi^{3}u_{2n}^{2}(-\varphi^{2})-\varphi^{2}u_{2n}^{2}(-\varphi^{2})-1,
\]
donc $ba^{2n}=1$ et $b=a^{-2n}$.\\
Nous procédons de même pour $a^{-(2n+1)}$.\\
2) La preuve est la même que la précédente, en un peu plus simple car ici, il n'y a pas de questions de signes.\\
3) Si $-\varphi^{2}$ est racine de $v_{2n+1}(X)$, alors il existe $\epsilon \in \{-1,+1\}$ et un entier $k$ premier à $2n+1$ tels que $\varphi=2\epsilon i\cos \frac{k\pi}{2n+1}$.
Nous avons:
\[
u_{2n}(4\cos^{2} \frac{k\pi}{2n+1})=\frac {\sin \frac{2nk\pi}{2n+1}}{\sin\frac{2k\pi}{2n+1}}
\]
d'après la proposition 19. Nous avons:
\[
\sin\frac{2nk\pi}{2n+1}=\sin (k\pi -\frac{k\pi}{2n+1})=(-1)^{k-1}\sin \frac{k\pi}{2n+1}
\]
d'où:
\[
u_{2n}(4\cos^{2} \frac{k\pi}{2n+1})=(-1)^{k-1}\frac{1}{2\cos \frac{k\pi}{2n+1}}.
\]
Il en résulte que $a^{2n+1}=(-1)^{n-1}\epsilon i(-1)^{k-1}=(-1)^{n+k}\epsilon i$, d'où $a^{2(2n+1)}=-1$ et $a^{4(2n+1)}=1$.\\
Si $-\varphi^{2}$est racine de $v_{2n}(X)$, comme $a^{2n}=(-1)^{n-1}u_{2n-1}(-\varphi^{2})$ d'après la proposition \ref{a13}, $u_{2n-1}(-\varphi^{2})=1$ et $a^{2n}=(-1)^{n-1}$.\\
4) La preuve est la même que celle du 3).
\end{proof}
\begin{notation}\label{a17}
Soit $K$ un corps commutatif. Si $P(X)\in K[X]$, on note $\theta (P(X))$ le produit des racines de $P(X)$.
\end{notation}
Il est clair que si $P(X)$ et $Q(X)$ sont deux éléments de $K[X]$, on a $\theta (P(X)Q(X))=\theta (P(X))\theta (Q(X))$.\\
On applique ce qui précède à la famille des polynômes $u_{n}(X)$, $n \in \mathbb{N}$.\\
On voit que $\theta(u_{2n+1}(X))=1$ et $\theta(u_{2n}(X))=n$. On a la décomposition en produit de facteurs irréductibles de $u_{n}(X)$ dans $\mathbb{Z}[X]$:
\[
u_{n}(X)=\prod_{d|n,d>0}v_{d}(X)
\]
où $v_{1}(X)=v_{2}(X)=1$ et chaque $v_{n}(X)$ est un polynôme unitaire.\\
Le but de la proposition suivante est de calculer $\theta (v_{n}(X))$ pour tout $n$.
\begin{proposition}\hfill\label{a18}
\begin{enumerate}
\item Si $n=2p^{m}$ avec $p$ premier et $m\geqslant 1$, on a $\theta(v_{n}(X))=p$.
\item Si $n$ n'est pas de la forme précédente, on a $\theta(v_{n}(X))=1$.
\end{enumerate}

\end{proposition}
\begin{proof}
Nous pouvons remarquer que comme $v_{n}(X)$ est un polynôme à coefficients entiers et comme toutes les racines de $v_{n}(X)$ sont strictement positives, on a $\theta (v_{n}(X))\in \mathbb{N}-\{0\}$.\\
D'après la remarque initiale, si $n$ est impair, on a $\theta(v_{n}(X))=1$.\\
1) Nous montrons par récurrence sur $m$, que si $n=2p^{m}$ avec $p$ premier, alors $\theta(v_{n}(X))=p$.\\
supposons d'abord que $m=1$.\\
Si $p=2$, $u_{4}(X)=X-2=v_{4}(X)$ et $\theta(v_{4}(X))=2$.\\
Si $p>2$, alors $u_{2p}(X)=v_{1}(X)v_{2}(X)v_{p}(X)v_{2p}(X)$. D'après la remarque initiale, nous savons que $\theta(u_{2p}(X))=p$ et que $\theta(v_{1}(X))=\theta(v_{2}(X))=\theta(v_{p}(X))=1$, donc $\theta(v_{2p}(X))=p$.\\
Supposons maintenant $m>1$ et le résultat vrai pour tout entier $m'$ tel que $1\leqslant m'<m$.\\
Si $p=2$, 
\[
u_{2^{m+1}}(X)=\prod_{i=0}^{m+1}v_{2^{i}}(X)=(\prod_{i=0}^{m}v_{2^{i}}(X))v_{2^{m+1}}(X)=u_{2^{m}}(X)v_{2^{m+1}}(X).
\]
Nous avons $\theta (u_{2^{m}}(X))=2^{m-1}$ et $\theta (u_{2^{m+1}}(X))=2^{m}$, donc $\theta (v_{2^{m+1}}(X))=2$.\\
Si $p>2$, 
\[
u_{2p^{m}}(X)=u_{p^{m}}(X)\prod_{i=0}^{m}v_{2p^{i}}(X)
\]
d'où:
\[
\theta (u_{2p^{m}})(X)=p^{m}=\theta(u_{p^{m}}(X))(\prod_{i=0}^{m-1}\theta(v_{2p^{i}}(X)))\theta(v_{2p^{m}}(X)).
\]
Nous avons $\theta(u_{p^{m}}(X))=1$ et, par hypothèse de récurrence, $\theta(v_{2p^{i}}(X))=p$ pour $1\leqslant i\leqslant m-1$, d'où $\theta(v_{2p^{m}}(X))=p$.\\
2) Supposons que $n=2^{\beta}p_{1}^{\alpha_{1}}\ldots p_{k}^{\alpha_{k}}$, les $p_{i}$ étant des nombres premiers impairs, distincts deux à deux, $k\geqslant1$, $\beta\geqslant1$ ou $\beta\geqslant2$ si $k=1$ et $\alpha_{i}\geqslant 1$ pour tout $i$.\\
Parmi les diviseurs de $n$, il y a tous les $2^{\gamma}$, $1\leqslant \gamma \leqslant \beta$ et nous avons vu que $\theta(v_{2^{\gamma}}(X))=2$ ($2\leqslant \gamma \leqslant \beta$) donc, dans le produit $\prod_{d|n,d>1}\theta(v_{d}(X))$, le produit $\prod_{1\leqslant \gamma \leqslant \beta}\theta(v_{2^{\gamma}}(X))$ contribue $2^{\beta-1}$.\\
Parmi les diviseurs de $n$, il y a aussi tous les $2p_{i}^{\beta_{i}}$ ($1 \leqslant i \leqslant k, \, 1 \leqslant \beta_{i} \leqslant \alpha_{i}$) et pour chacun de ces diviseurs, on a $\theta(v_{2p_{i}^{\beta_{i}}}(X))=p_{i}$, donc dans $\theta (u_{n}(X))$, ils contribuent $\prod_{i=1}^{k}p_{i}^{\alpha_{i}}$.\\
En combinant ces résultats, nous voyons que tous ces diviseurs donnent une contribution de $2^{\beta-1}p_{1}^{\alpha_{1}}\ldots p_{k}^{\alpha_{k}}=\theta(u_{n}(X))$. Il en résulte que pour tous les autres facteurs $v_{d}(X)$, on a $\theta(v_{d}(X))=1$. En particulier $\theta(v_{n}(X))=1$.
\end{proof}Nous en déduisons le résultat suivant:
\begin{corollary}\label{a19}
Soit $\gamma$ une racine de $v_{r}(X)$. Alors:\\
\begin{enumerate}
\item \begin{enumerate}
\item Si $r$ est impair, $\gamma$ est inversible dans $K_{0}$ et l'on a $\gamma u_{r-1}^{2}(\gamma)=1$.
\item Si $r$ est pair,\begin{itemize}
\item si $r=2r_{1}$, avec $r_{1}$ puissance d'un nombre premier $p$, alors $\gamma$ n'est pas inversible et $N_{K_{0}/ \mathbb{Q}}(\gamma)=p$: il existe un polynôme unitaire $P(X) \in \mathbb{Z}[X]$ tel que $\gamma P(\gamma)=p$;
\item si $r_{1}$ n'est pas une  puissance d'un nombre premier, alors $\gamma$ est inversible et il existe un polynôme unitaire $P(X) \in \mathbb{Z}[X]$ tel que $\gamma P(\gamma)=1$.
\end{itemize}

\end{enumerate}
\item $4-\gamma$ n'est pas inversible si et seulement si $r$ est une puissance d'un nombre premier. Si $r=p^{n}$ avec $p$ premier, il existe un polynôme unitaire $P(X) \in \mathbb{Z}[X]$ tel que $(4-\gamma )P(\gamma)=p$.\\
Si $r$ n'est pas une puissance d'un nombre premier, il existe un polynôme unitaire $P(X) \in \mathbb{Z}[X]$ tel que $(4-\gamma )P(\gamma)=1$.
\end{enumerate}

\end{corollary}
\begin{proof}
Le 1) est une conséquence immédiate de la proposition \ref{a18} et du fait que tous les conjugués de $\gamma$ sont des polynômes en $\gamma$ car l'extension $K_{0}/ \mathbb{Q}$ est cyclique.\\
2) Il existe un entier $k$ tel que $(r,k)=1$, $1 \leqslant k < \frac{r}{2}$ et $\gamma = 4\cos^{2} \frac{k\pi}{r}$. Nous avons alors:
\[
4-\gamma=4\sin^{2} \frac{k\pi}{r}=4\cos^{2} (\frac{\pi}{2}-\frac{k\pi}{r})=4\cos^{2} (\frac{r-2k}{2r})\pi.
\]
\begin{itemize}
\item Si $r$ est impair, alors $(r-2k,2r)=1$ et $4-\gamma$ est racine de $v_{2r}(X)$.
\item Si $r$ est pair, $r=2r_{1}$, $4-\gamma=4\cos^{2} (\frac{r_{1}-k}{r})\pi$, $(k,r_{1})=1$ et $k$ est impair.
\begin{itemize}
\item Si $r_{1}$ est impair, $2|(r_{1}-k)$, donc $4-\gamma$ est racine de $v_{r_{1}}(X)$.
\item Si $r_{1}$ est pair, $r_{1}-k$ est impair, donc $4-\gamma$ est racine de $v_{r}(X)$.

\end{itemize}

\end{itemize}
De ceci on déduit tous les résultats de l'énoncé.
\end{proof}
\begin{proposition}
Soient $p$, $q$, $r$ des entiers $\geqslant 3$. Soient $\alpha$ une racine de $v_{p}(X)$, $\beta$ une racine de $v_{q}(X)$ et $\gamma$ une racine de $v_{r}(X)$.

1) On suppose que $\alpha\beta=4\gamma$. Alors on a $p=q=4$, $r=3$, $\alpha=\beta=2$ et $\gamma=1$.

2) On suppose que $4-\alpha-\beta-\gamma=0$. Alors on a deux possibilités (à l'ordre près de $p$, $q$, $r$).
\begin{enumerate}
  \item \[
  p=4,q=r=3,\alpha=2, \beta=\gamma=1
  \]
  \item 
  \[
  p=q=5,r=3,\alpha=\tau,\beta=3-\tau,\gamma=1
  \]
\end{enumerate}
où $\tau=\frac{3+\sqrt{5}}{2}$ est une racine du polynôme $u_{5}(X)$.
\end{proposition}
\begin{proof}
1) Comme $\alpha\beta=4\gamma$, on voit que $2$ divise $N_{K_{0}/\mathbb{Q}}(\alpha)$ et $N_{K_{0}/\mathbb{Q}}(\beta)$ donc $N_{K_{0}/\mathbb{Q}}(\alpha)=N_{K_{0}/\mathbb{Q}}(\beta)=2$ et $p$ et $q$ sont des puissances de $2$ et nous pouvons supposer que $p \leqslant q$; on voit aussi que, dans ces conditions, $\gamma$ est inversible, donc $r$ n'est pas une puissance de $2$. $\alpha$ et $\beta$ sont des éléments d'un sous-corps $K$ réel d'un corps cyclotomique $\mathbb{Q}(\zeta)$ où $\zeta$ est une racine primitive q-ième de l'unité. On a $\gamma \in K$ par hypothèse, mais $v_{r}(X)$ reste irréductible dans $K$, donc $v_{r}(X)$ est de degré $1$: $v_{r}(X)=X-1$, $r=3$ et $\gamma=1$. On a alors $\alpha=\beta=2$.\\
2) Supposons d'abord que $r=3$. Alors $\gamma=1$ et on a la relation: $\alpha+\beta=3$. On peut avoir $\alpha=2$ et $\beta=1$ donc $p=4$ et $q=3$, ou bien $p=q=5$, $\alpha=\tau$, $\beta=3-\tau$. Si l'un de $p$,$q$, $r$ est dans $\{3,4,6\}$ on est dans l'un des cas précédents, nous pouvons donc supposer que $p$,$q$,et $r$ sont $\geqslant 5$ mais alors aucune somme $\alpha+\beta+\gamma$ n'est entière (égale à $4$).
\end{proof}
\section{Formules en rang $3$}
Dans cette section on suppose que $(W(p,q,r),S)$ est un système de Coxeter de rang $3$. On pose $S:=\{s_{1},s_{2},s_{3}\}$. Soit $R \to GL(M)$ une représentation de réflexion. Le but de cette section est, entre autres, de donner des formules pour les éléments

$(s_{i}s_{j})^{k}, s_{i}(s_{i}s_{j})^{k}, t_{k}:=s_{1}(s_{2}s_{3})^{k}, x_{k}:=s_{2}(s_{3}s_{1})^{k}, y_{k}:=s_{3}(s_{1}s_{2})^{k},$
qui nous seront utiles dans les exemples étudiés.\\
\emph{Il faut remarquer que les formules obtenues sont valables sans hypothèses sur $\alpha$, $\beta$ et $\gamma$.}
\begin{proposition}\label{b12}
On a les formules: $\forall k \in \mathbb{Z}$,\\
1) 
\begin{align}
(s_{1}s_{2})^{2k} &=  \begin{pmatrix}
u_{4k+1}(\alpha) & -\alpha u_{4k}(\alpha) & \beta \alpha u_{2k}^{2}(\alpha)+\alpha lu_{2k+1}(\alpha)u_{2k}(\alpha)\\
u_{4k}(\alpha) & -u_{4k-1}(\alpha) & \alpha lu_{2k}^{2}(\alpha)+\beta u_{2k}(\alpha)u_{2k-1}(\alpha)\\
0 & 0 & 1
\end{pmatrix}\\
(s_{1}s_{2})^{2k+1} &=  \begin{pmatrix}
u_{4k+3}(\alpha) & -\alpha u_{4k+2}(\alpha) & \beta  u_{2k+1}^{2}(\alpha)+\alpha lu_{2k+2}(\alpha)u_{2k+1}(\alpha)\\
u_{4k+2}(\alpha) & -u_{4k+1}(\alpha) &  lu_{2k+1}^{2}(\alpha)+\beta u_{2k+1}(\alpha)u_{2k}(\alpha)\\
0 & 0 & 1
\end{pmatrix}.
\end{align}\\
En particulier si $p$ est pair: $p=2p_{1}$, on a:
\begin{align}
(s_{1}s_{2})^{p_{1}} &= 
\begin{pmatrix}
-1 & 0 & \frac{2(2\beta+\alpha l)}{4-\alpha}\\
0 & -1 & \frac{2(\beta+2l)}{4-\alpha}\\
0 & 0 &1
\end{pmatrix}.
\end{align}
2)
\begin{align}
(s_{1}s_{3})^{2k} &= \begin{pmatrix}
u_{4k+1}(\beta) & \alpha\beta u_{2k}^{2}(\beta)+\beta mu_{2k+1}(\beta)u_{2k}(\beta) & -\beta u_{4k}(\beta)\\
0 & 1 & 0\\
u_{4k}(\beta) & \beta m u_{2k}^{2}(\beta)+\alpha u_{2k}(\beta)u_{2k-1}(\beta) & -u_{4k-1}(\beta)
\end{pmatrix}\\
(s_{1}s_{3})^{2k+1} &=  \begin{pmatrix}
u_{4k+3}(\beta) & \alpha u_{2k+1}^{2}(\beta)+\beta mu_{2k+2}(\beta)u_{2k+1}(\beta) & -\beta u_{4k+2}(\beta)\\
0 & 1 & 0\\
u_{4k+2}(\beta) & m u_{2k+1}^{2}(\beta)+\alpha u_{2k+1}(\beta)u_{2k}(\beta) & -u_{4k+1}(\beta)
\end{pmatrix}.
\end{align}\\
En particulier si $q$ est pair: $q=2q_{1}$, on a:
\begin{align}
(s_{1}s_{3})^{q_{1}} &= 
\begin{pmatrix}
-1 &  \frac{2(2\alpha+\beta m)}{4-\beta} & 0 \\
0 & 1 & 0 \\
0 & \frac{2(\alpha+2m)}{4-\beta} & -1
\end{pmatrix}.
\end{align}
3) \begin{align}
(s_{2}s_{3})^{2k}  &=  \begin{pmatrix}
1 & 0 & 0\\
\gamma u_{2k}^{2}(\gamma)+lu_{2k+1}(\gamma)u_{2k}(\gamma) & u_{4k+1}(\gamma) & -lu_{4k}(\gamma)\\
\gamma u_{2k}^{2}(\gamma)+mu_{2k}(\gamma)u_{2k-1 }(\gamma) & mu_{4k}(\gamma) & -u_{4k-1}(\gamma)
\end{pmatrix}\\
(s_{2}s_{3})^{2k+1} &=  \begin{pmatrix}
1 & 0 & 0\\
 u_{2k+1}^{2}(\gamma)+lu_{2k+2}(\gamma)u_{2k+1}(\gamma) & u_{4k+3}(\gamma) & -lu_{4k+2}(\gamma)\\
 u_{2k+1}^{2}(\gamma)+mu_{2k+1}(\gamma)u_{2k }(\gamma) & mu_{4k+2}(\gamma) & -u_{4k+1}(\gamma)
\end{pmatrix}.
\end{align}
En particulier si $r$ est pair: $r=2r_{1}$, on a:
\begin{align}
(s_{2}s_{3})^{r_{1}} &= 
\begin{pmatrix}
1 & 0 & 0 \\
 \frac{2(l+2)}{4-\gamma}& -1 & 0 \\
 \frac{2(m+2)}{4-\gamma} &0 & -1
\end{pmatrix}.
\end{align}
\end{proposition}
\begin{proof}
Nous ne démontrons que les formules du 1), les autres se montrant de manière analogue.\\
Nous supposons d'abord $k \geqslant 0$ et nous procédons par récurrence sur $k$.\\
Nous avons:\\ $s_{1}=
\begin{pmatrix}
-1 & \alpha & \beta\\
0 & 1 & 0\\
0 & 0 & 1
\end{pmatrix}$, $s_{2}=
\begin{pmatrix}
1 & 0 & 0\\
1 & -1 & l\\
0 & 0 & 1
\end{pmatrix}$ donc $s_{1}s_{2}=
\begin{pmatrix}
\alpha-1 & -\alpha & \alpha l+\beta\\
1 & -1 & l\\
0 & 0 & 1
\end{pmatrix}$.
Lorsque $k=0$, $u_{4k+1}(\alpha)=u_{1}(\alpha)=1$, $u_{4k}(\alpha)=u_{0}(\alpha)=0$, $u_{4k-1}(\alpha)=-u_{1}(\alpha)=-1$, $u_{2k}(\alpha)=0$, donc le résultat est vrai pour $k=0$.\\
Comme $u_{3}(\alpha)=\alpha-1$ et $u_{2}(\alpha)=1$, le résultat est aussi vrai pour $k=1$.\\
Nous posons $(s_{1}s_{2})^{k}=(a_{ij}(k))_{1\leqslant i,j\leqslant3}$.\\
Il est clair que pour tout $k$, $a_{31}(k)=a_{32}(k)=0$ et $a_{33}(k)=1$.\\
Les calculs pour $a_{11}$, $a_{21}$, $a_{12}$, et $a_{22}$ ont déjà été faits dans \cite{Z}.\\
Soit $k\geqslant 1$. Supposons que $(s_{1}s_{2})^{2k}$ a la forme indiquée.\\
Avons nous:
\[
a_{13}(2k+1)=(l\alpha+\beta)a_{11}(2k)+la_{12}(2k)+a_{13}(2k)?
\]
On montre que 
\begin{multline*}
\beta u_{2k}^{2}(\alpha) +l\alpha u_{2k+2}(\alpha)u_{2k+1}(\alpha)=\\
(l\alpha+\beta)u_{4k+1}(\alpha)-l\alpha u_{4k}(\alpha)+\beta u_{2k+1}^{2}(\alpha)+l\alpha u_{2k+1}(\alpha)u_{2k}(\alpha)
\end{multline*}
ce qui est équivalent à 
\begin{multline*}
l\alpha[u_{4k+1}(\alpha)-u_{4k}(\alpha)-u_{2k+2}(\alpha)u_{2k+1}(\alpha)+u_{2k+1}(\alpha)u_{2k}(\alpha)]\\
+\beta[u_{4k+1}(\alpha)+\alpha u_{2k}^{2}(\alpha)-u_{2k+1}^{2}(\alpha)]=0.
\end{multline*}
Dans cette égalité, le coefficient de $\beta$ est $0$ par (21).
Nous avons:
\[
u_{4k+1}(\alpha)-u_{4k}(\alpha)-u_{2k+1}(\alpha)(u_{2k+2}(\alpha)-u_{2k}(\alpha))=u_{4k+1}(\alpha)-u_{4k}(\alpha)-u_{4k+2}(\alpha)=0
\]
en utilisant (9) puis (A1). Il en résulte que $a_{13}(2k+1)$ a bien la forme indiquée. Nous procédons de même pour $a_{23}$. Enfin c'est la même méthode pour passer de $(s_{1}s_{2})^{2k-1}$ à $(s_{1}s_{2})^{2k}$ en utilisant (9), (A2)et (22).\\
Les formules pour $k < 0$ s'obtiennent à partir des formules précédentes en changeant $k$ en $-k$ comme on le vérifie en utilisant (21), (22), (23) et (24).

\end{proof}
Du résultat précédent, nous déduisons les matrices des réflexions contenues dans les sous-groupes paraboliques $<s_{i},s_{j}>$ ($1\leqslant i < j \leqslant 3$).
\begin{proposition}\label{b13}
On a les formules: $\forall k \in \mathbb{Z}$,\\
1)\begin{align*}
s_{1}(s_{1}s_{2})^{2k} &=  \begin{pmatrix} 
u_{4k-1}(\alpha) & -\alpha u_{4k-2}(\alpha) & u_{2k-1}(\alpha)(\alpha lu_{2k}(\alpha)+\beta u_{2k-1}(\alpha))\\
u_{4k}(\alpha) & -u_{4k-1}(\alpha) & u_{2k}(\alpha)(\alpha lu_{2k}(\alpha)+\beta u_{2k-1}(\alpha))\\
0 & 0 & 1
\end{pmatrix}\\
s_{1}(s_{1}s_{2})^{2k+1} &=  \begin{pmatrix} 
u_{4k+1}(\alpha) &  -\alpha u_{4k}(\alpha) & \alpha u_{2k}(\alpha)(\beta u_{2k}(\alpha)+lu_{2k+1}(\alpha))\\
u_{4k+2}(\alpha) & -u_{4k+1}(\alpha) & u_{2k+1}(\alpha)(\beta u_{2k}(\alpha)+lu_{2k+1}(\alpha))\\
0 & 0 & 1
\end{pmatrix}\\
H(s_{1}(s_{1}s_{2})^{2k}) &=  <b_{3},\alpha(u_{2k}(\alpha)-u_{2k-2}(\alpha))a_{1}+(u_{2k+1}(\alpha)-u_{2k-1}(\alpha))a_{2}>\\
v_{s_{1}(s_{1}s_{2})^{2k}}^{-} &=  u_{2k-1}(\alpha)a_{1}+u_{2k}(\alpha)a_{2}\\
H(s_{1}(s_{1}s_{2})^{2k+1}) &=  <b_{3},(u_{2k+1}(\alpha)-u_{2k-1}(\alpha))a_{1}+(u_{2k+2}(\alpha)-u_{2k}(\alpha))a_{2}>\\
v_{s_{1}(s_{1}s_{2})^{2k+1}}^{-} &=  \alpha u_{2k}(\alpha)a_{1}+u_{2k+1}(\alpha)a_{2}.
\end{align*}
2)\begin{align*}
s_{1}(s_{1}s_{3})^{2k} &= 
\begin{pmatrix}
u_{4k-1}(\beta) & u_{2k-1}(\beta)(\alpha u_{2k-1}(\beta)+\beta mu_{2k}(\beta)) & -\beta u_{4k-2}(\beta)\\
0 & 1 & 0\\
u_{4k}(\beta) & u_{2k}(\beta)(\alpha u_{2k-1}(\beta)+\beta mu_{2k}(\beta)) & -u_{4k-1}(\beta)
\end{pmatrix}\\
s_{1}(s_{1}s_{3})^{2k+1} &= 
\begin{pmatrix}
u_{4k+1}(\beta) & \beta u_{2k}(\beta)(mu_{2k+1}(\beta)+\alpha u_{2k}(\beta)) & -\beta u_{4k}(\beta) \\
0 & 1 & 0 \\
u_{4k+2}(\beta) & u_{2k+1}(\beta)(mu_{2k+1}(\beta)+\alpha u_{2k}(\beta)) & -u_{4k+1}(\beta)
\end{pmatrix}\\
H(s_{1}(s_{1}s_{3})^{2k}) &=  <b_{2}, \beta(u_{2k}(\beta)-u_{2k-2}(\beta))a_{1}+(u_{2k+1}(\beta)-u_{2k-1}(\beta))a_{3}>\\
v_{s_{1}(s_{1}s_{3})^{2k}}^{-} &=  u_{2k-1}(\beta)a_{1}+u_{2k}(\beta)a_{3}\\
H(s_{1}(s_{1}s_{3})^{2k+1}) &=  <b_{2},(u_{2k+1}(\beta)-u_{2k-1}(\beta))a_{1}+(u_{2k}(\beta)-u_{2k-2}(\beta))a_{2}>\\
v_{s_{1}(s_{1}s_{3})^{2k+1}}^{-} &=  \beta u_{2k}(\beta)a_{1}+u_{2k+1}(\beta)a_{3}.
\end{align*}
3)\begin{align*}
s_{2}(s_{2}s_{3})^{2k} &= 
\begin{pmatrix}
1 & 0 & 0\\
u_{2k-1}(\gamma)(u_{2k-1}(\gamma)+lu_{2k}(\gamma)) & u_{4k-1}(\gamma) & -lu_{4k-2}(\gamma)\\
mu_{2k}(\gamma)(u_{2k-1}(\gamma)+lu_{2k}(\gamma)) & mu_{4k}(\gamma) & -u_{4k-1}(\gamma)
\end{pmatrix}\\
s_{2}(s_{2}s_{3})^{2k+1} &= 
\begin{pmatrix}
1 & 0 & 0\\
lu_{2k}(\gamma)(u_{2k+1}(\gamma)+mu_{2k}(\gamma)) & u_{4k+1}(\gamma) & -lu_{4k}(\gamma)\\
u_{2k+1}(\gamma)(u_{2k+1}(\gamma)+mu_{2k}(\gamma)) & mu_{4k+2}(\gamma) & -u_{4k+1}(\gamma) 
\end{pmatrix}\\
H(s_{2}(s_{2}s_{3})^{2k}) &= <b_{1},l(u_{2k}(\gamma)-u_{2k-2}(\gamma))a_{2}+(u_{2k+1}(\gamma)-u_{2k-1}(\gamma))a_{3}>\\
v_{s_{2}(s_{2}s_{3})^{2k}}^{-} &=  u_{2k-1}(\gamma)a_{2}+mu_{2k}(\gamma)a_{3}\\
H(s_{2}(s_{2}s_{3})^{2k+1}) &=  <b_{1},(u_{2k+1}(\gamma)-u_{2k-1}(\gamma))a_{2}+m(u_{2k+2}(\gamma)-u_{2k}(\gamma))a_{3}>\\
v_{s_{2}(s_{2}s_{3})^{2k+1}}^{-} &=  lu_{2k}(\gamma)a_{2}+u_{2k+1}(\gamma)a_{3}.
\end{align*}

\end{proposition}
\begin{proof}
Elle ne présente pas de difficultés en utilisant la proposition 32 et les formules de la section 1.\\
\end{proof}
\begin{remark}
Soit $s$ une réflexion de $G(\alpha, \beta,\gamma;l)$. On suppose que $v^{-}(s)=\lambda_{1}a_{1}+\lambda_{2}a_{2}+\lambda_{3}a_{3}$. Alors on a:
\[
s=Id_{M}+\begin{pmatrix}
\lambda_{1}&0&0\\
0&\lambda_{2}&0\\
0&0&\lambda_{3}
\end{pmatrix}
\begin{pmatrix}
c(s,s_{1})&c(s,s_{2})&c(s,s_{3})\\
c(s,s_{1})&c(s,s_{2})&c(s,s_{3})\\
c(s,s_{1})&c(s,s_{2})&c(s,s_{3})
\end{pmatrix}
\]
\end{remark}
\begin{proof}
On a $s(a_{i})=a_{i}+c(s,s_{i})v^{-}(s)$ par définition et on en déduit tout de suite le résultat.
\end{proof}
\begin{proposition}
On a $\forall k \in \mathbb{Z}$:
\begin{enumerate}
\item  \begin{align}
C(s_{3},s_{1}(s_{1}s_{2})^{2k})&=  (u_{2k-1}(\alpha)+mu_{2k}(\alpha))(\alpha lu_{2k}(\alpha)+\beta u_{2k-1}(\alpha))\\
&=  \alpha\gamma u_{2k}^{2}(\alpha)+\beta u_{2k-1}^{2}(\alpha)+(\alpha l+\beta m)u_{2k}(\alpha)u_{2k-1}(\alpha)
\end{align}
 \begin{align}
C(s_{3},s_{1}(s_{1}s_{2})^{2k+1})&= (\alpha u_{2k}(\alpha)+mu_{2k+1}(\alpha))(\beta u_{2k}(\alpha)+lu_{2k+1}(\alpha))\\
&=  \gamma u_{2k+1}^{2}(\alpha)+\alpha\beta u_{2k}^{2}(\alpha)+(\alpha l+\beta m)u_{2k+1}(\alpha)u_{2k}(\alpha)
\end{align}
En particulier si $p$ est pair: $p=2p_{1}$, on a:
 \begin{align}
 C(s_{3},s_{1}(s_{1}s_{2})^{p_{1}})&= 4-\beta - \frac{2\Delta}{4-\alpha}\\
 C(s_{3},s_{2}(s_{1}s_{2})^{p_{1}})&= 4-\gamma - \frac{2\Delta}{4-\alpha}
 \end{align}
\item \begin{align}
C(s_{2},s_{1}(s_{1}s_{3})^{2k})&= (u_{2k-1}(\beta)+lu_{2k}(\beta))(\alpha u_{2k-1}(\beta)+\beta mu_{2k}(\beta))\\
&=  \beta\gamma u_{2k}^{2}(\beta)+\alpha u_{2k-1}^{2}(\beta)+(\alpha l+\beta m)u_{2k}(\beta)u_{2k-1}(\beta)\\
C(s_{2},s_{1}(s_{1}s_{3})^{2k+1})&= (\beta u_{2k}(\beta)+lu_{2k+1}(\beta))(mu_{2k+1}(\beta)+\alpha u_{2k}(\beta))\\
&=  \gamma u_{2k+1}^{2}(\beta)+\alpha\beta u_{2k}^{2}(\beta)+(\alpha l+\beta m)u_{2k+1}(\beta)u_{2k}(\beta)
\end{align}
En particulier si $q$ est pair: $q=2q_{1}$, on a:
\begin{align}
C(s_{2},s_{1}(s_{1}s_{3})^{q_{1}})&= 4-\alpha-\frac{2\Delta}{4-\beta}\\
C(s_{2},s_{3}(s_{1}s_{3})^{q_{1}})&= 4-\gamma-\frac{2\Delta}{4-\beta}\\
\end{align}
\item \begin{align}
C(s_{1},s_{2}(s_{2}s_{3})^{2k})&= (\alpha u_{2k-1}(\gamma)+\beta mu_{2k}(\gamma))(u_{2k-1}(\gamma)+lu_{2k}(\gamma))\\
&=  \beta\gamma u_{2k}^{2}(\gamma)+\alpha u_{2k-1}^{2}(\gamma)+(\alpha l+\beta m)u_{2k}(\gamma)u_{2k-1}(\gamma)
\end{align}
\begin{align}
C(s_{1},s_{2}(s_{2}s_{3})^{2k+1})&= (\alpha lu_{2k}(\gamma)+\beta u_{2k+1}(\gamma))(u_{2k+1}(\gamma)+mu_{2k}(\gamma))\\
&=  \beta u_{2k+1}^{2}(\gamma)+\alpha\gamma u_{2k}^{2}(\gamma)+(\alpha l+\beta m)u_{2k+1}(\gamma)u_{2k}(\gamma)
\end{align}
En particulier si $r$ est pair: $r=2r_{1}$, on a:
\begin{align}
C(s_{1},s_{2}(s_{2}s_{3})^{r_{1}})&= 4-\alpha-\frac{2\Delta}{4-\gamma}\\
C(s_{1},s_{3}(s_{2}s_{3})^{r_{1}})&= 4-\beta-\frac{2\Delta}{4-\gamma}\\
\end{align}
\item On suppose que $p$ et $q$ sont pairs: $p=2p_{1}$ et $q=2q_{1}$, alors on a:
\begin{align*}
C(s_{1}(s_{1}s_{2})^{p_{1}}),s_{1}(s_{1}s_{3})^{q_{1}}))&=4\left(\frac{\beta+2l}{4-\alpha}\right)\left(\frac{\alpha+2m}{4-\beta}\right)\\&=4-\frac{8\Delta}{(4-\alpha)(4-\beta)};\\
C(s_{1}(s_{1}s_{2})^{p_{1}}),s_{3}(s_{1}s_{3})^{q_{1}}))&=\beta\left(\frac{\beta+2l}{4-\alpha}\right)\left(\frac{\alpha+2m}{4-\beta}\right)\\&=\beta-\frac{2\beta\Delta}{(4-\alpha)(4-\beta)};\\
C(s_{2}(s_{1}s_{2})^{p_{1}}),s_{1}(s_{1}s_{3})^{q_{1}}))&=\alpha\left(\frac{\beta+2l}{4-\alpha}\right)\left(\frac{\alpha+2m}{4-\beta}\right)\\&=\alpha-\frac{2\alpha\Delta}{(4-\alpha)(4-\beta)};\\
C(s_{2}(s_{1}s_{2})^{p_{1}}),s_{3}(s_{1}s_{3})^{q_{1}}))&=\left(\frac{-8+2\alpha+2\beta+\alpha l}{4-\alpha}\right)\left(\frac{-8+2\alpha+2\beta+\beta m}{4-\beta}\right)\\&=\gamma+\frac{\Delta(8-2\alpha-2\beta)}{(4-\alpha)(4-\beta)}.
\end{align*}
\item On suppose que $p$ et $r$ sont pairs: $p=2p_{1}$ et $r=2r_{1}$, alors on a:
\begin{align*}
C(s_{1}(s_{1}s_{2})^{p_{1}}),s_{2}(s_{2}s_{3})^{r_{1}}))&=\alpha\left(\frac{2\beta+\alpha l}{4-\alpha}\right)\left(\frac{m+2}{4-\beta}\right)\\&=\alpha-\frac{2\alpha\Delta}{(4-\alpha)(4-\beta)};\\
C(s_{1}(s_{1}s_{2})^{p_{1}}),s_{3}(s_{2}s_{3})^{r_{1}}))&=\left(\frac{-8+2\alpha+2\gamma+\beta m}{4-\alpha}\right)\left(\frac{-8+2\alpha+2\gamma+\alpha l}{4-\gamma}\right)\\&=\beta+\frac{\Delta(8-2\alpha-2\gamma)}{(4-\alpha)(4-\gamma)};\\
C(s_{2}(s_{1}s_{2})^{p_{1}}),s_{2}(s_{2}s_{3})^{r_{1}}))&=4\left(\frac{2\beta+\alpha l}{4-\alpha}\right)\left(\frac{m+2}{4-\gamma}\right)\\&=4-\frac{8\Delta}{(4-\alpha)(4-\gamma)};\\
C(s_{2}(s_{1}s_{2})^{p_{1}}),s_{3}(s_{2}s_{3})^{r_{1}}))&=\gamma\left(\frac{2\beta+\alpha l}{4-\alpha}\right)\left(\frac{m+2}{4-\beta}\right)\\&=\gamma-\frac{2\gamma\Delta}{(4-\alpha)(4-\gamma)}.
\end{align*}
\item On suppose que $q$ et $r$ sont pairs: $q=2q_{1}$ et $r=2r_{1}$, alors on a:
\begin{align*}
C(s_{1}(s_{1}s_{3})^{q_{1}}),s_{2}(s_{2}s_{3})^{r_{1}}))&=\left(\frac{-8+2\beta+2\gamma+\alpha l}{4-\beta}\right)\left(\frac{-8+2\beta+2\gamma+\beta m}{4-\gamma}\right)\\&=\alpha+\frac{\Delta(8-2\beta-2\gamma)}{(4-\beta)(4-\gamma)};\\
C(s_{1}(s_{1}s_{3})^{q_{1}}),s_{3}(s_{2}s_{3})^{r_{1}}))&=\beta\left(\frac{2\alpha+\beta m}{4-\beta}\right)\left(\frac{l+2}{4-\gamma}\right)\\&=\beta-\frac{2\beta\Delta}{(4-\beta)(4-\gamma)};\\
C(s_{3}(s_{1}s_{3})^{q_{1}}),s_{2}(s_{2}s_{3})^{r_{1}}))&=\gamma\left(\frac{2\alpha+\beta m}{4-\beta}\right)\left(\frac{l+2}{4-\gamma}\right)\\&=\gamma-\frac{2\gamma\Delta}{(4-\beta)(4-\gamma)};\\
C(s_{3}(s_{1}s_{3})^{q_{1}}),s_{3}(s_{2}s_{3})^{r_{1}}))&=4\left(\frac{2\alpha+\beta m}{4-\beta}\right)\left(\frac{l+2}{4-\gamma}\right)\\&=4-\frac{8\Delta}{(4-\beta)(4-\gamma)}.
\end{align*}

\end{enumerate}

\end{proposition}
\begin{proof}
Les calculs se font sans difficultés en utilisant la remarque 3.\\
Pour les formules lorsque $p$ (ou $q$, ou $r$) est pair ($p=2p_{1}$) nous ne montrons que
\[ 
C(s_{3},s_{1}(s_{1}s_{2})^{p_{1}})=4-\beta-\frac{2\Delta}{4-\alpha} 
\]
avec $p_{1}$ pair, les autres cas se traitant de la même manière. On a:
\[ C(s_{3},s_{1}(s_{1}s_{2})^{p_{1}})=\gamma\alpha u_{p_{1}}^{2}(\alpha)+\beta u_{p_{1}-1}^{2}(\alpha)+(\alpha l+\beta m)u_{p_{1}}(\alpha)u_{p_{1}-1}(\alpha).
\]
D'après la proposition \ref{a14} nous avons: $\alpha u_{p_{1}}^{2}(\alpha)=\frac{4}{4-\alpha}$, $u_{p_{1}-1}^{2}(\alpha)=\frac{\alpha}{4-\alpha}$ et $u_{p_{1}}(\alpha)u_{p_{1}-1}(\alpha)=\frac{2}{4-\alpha}$ donc 
\[
C(s_{3},s_{1}(s_{1}s_{2})^{p_{1}})=(\frac{1}{4-\alpha})(4\gamma+\alpha\beta+2(\alpha l+\beta m)).
\]
Mais $\Delta=8-2\alpha-2\beta-2\gamma-(\alpha l+\beta m)$ donc $4\gamma+2(\alpha l+\beta m)=16-4\alpha-4\beta-2\Delta$ et
\[
C(s_{3},s_{1}(s_{1}s_{2})^{p_{1}})=(\frac{1}{4-\alpha})(16-4\alpha-4\beta+\alpha\beta-2\Delta)
\]
Comme $16-4\alpha-4\beta+\alpha\beta=(4-\alpha)(4-\beta)$ nous avons le résultat.\\
Les démonstrations de toutes les formules en 4), 5), 6) sont semblables. Nous ne démontrons que la quatrième du 4). On a dans la base $\mathcal{A}$ de $M$:
\[
s_{2}(s_{1}s_{2})^{p_{1}}=\begin{pmatrix} 1 & 0 & 0\\ 1 & -1 & l \\ 0 & 0 & 1
\end{pmatrix}\begin{pmatrix} -1 & 0 & \frac{2\beta+\alpha l}{4-\alpha} \\ 
0 & -1 &\frac{2(\beta+2l)}{4-\alpha}\\ 0 & 0 & 1
\end{pmatrix} = \begin{pmatrix} -1 & 0 & \frac{2(2\beta+\alpha l)}{4-\alpha} \\ -1 & 1 & \frac{2\beta+\alpha l}{4-\alpha} \\ 0 & 0 & 1

\end{pmatrix}
\]
et nous obtenons:
\begin{align*}
s_{2}(s_{1}s_{2})^{p_{1}}(a_{1}) &=  a_{1}-(2a_{1}+a_{2})\\
s_{2}(s_{1}s_{2})^{p_{1}}(a_{2}) &=  a_{2}\\
s_{2}(s_{1}s_{2})^{p_{1}}(a_{3}) &=  a_{3}+\left(\frac{2\beta+\alpha l}{4-\alpha}\right)(2a_{1}+a_{2}).
\end{align*}
Nous voyons de la même manière que l'on a:
\begin{align*}
s_{3}(s_{1}s_{3})^{q_{1}}(a_{1}) &=  a_{1}-(2a_{1}+a_{3})\\
s_{3}(s_{1}s_{3})^{q_{1}}(a_{2}) &=  a_{2}+\left(\frac{2\alpha+\beta m}{4-\beta}\right)(2a_{1}+a_{3})\\
s_{3}(s_{1}s_{2})^{q_{1}}(a_{3}) &=  a_{3}.
\end{align*}
Nous obtenons:
\[
s_{2}(s_{1}s_{2})^{p_{1}}(2a_{1}+a_{3})  =  (2a_{1}+a_{3})+\left(\frac{-8+2\alpha+2\beta+\alpha l}{4-\alpha}\right)(2a_{1}+a_{2})
\]
\[
s_{3}(s_{1}s_{3})^{q_{1}}(2a_{1}+a_{2})  =  (2a_{1}+a_{2})\left(\frac{-8+2\alpha+2\beta+\beta m}{4-\beta}\right)(2a_{1}+a_{3}).
\]
Nous avons alors:
\[
C(s_{2}(s_{1}s_{2})^{p_{1}},s_{3}(s_{1}s_{3})^{q_{1}})=\left(\frac{-8+2\alpha+2\beta+\alpha l}{4-\alpha}\right)\left(\frac{-8+2\alpha+2\beta+\beta m}{4-\beta}\right)
\]
et un calcul simple donne le résultat:
\[
C(s_{2}(s_{1}s_{2})^{p_{1}},s_{3}(s_{1}s_{3})^{q_{1}})=\gamma+\frac{\Delta(8-2\alpha-2\beta)}{(4-\alpha)(4-\beta)}.
\]
\end{proof}
\begin{notation}
\begin{enumerate}
\item Pour tout $k$ dans $\mathbb{Z}$, on pose $t_{k}:=s_{1}(s_{2}s_{3})^{k}$, $x_{k}:=s_{2}(s_{3}s_{1})^{k}$ et  $y_{k}:=s_{3}(s_{1}s_{2})^{k}$.
\item On pose $\theta:=-4+\alpha+\beta+\gamma+\alpha l$ et $\theta':=4-\alpha-\beta-\gamma-\beta m$.
\end{enumerate}

\end{notation}
On peut remarquer que $\theta'-\theta=\Delta$ et que $\alpha l-\beta m=\theta+\theta'$. Il en résulte que $\Delta=0$ équivaut à $\theta'=\theta$.
\begin{proposition}\label{b15}
\begin{enumerate}
\item Le polynôme caractéristique de $t_{k}$ (resp. de $x_{k}$, resp. de $y_{k}$) est:
\begin{align}
P_{t_{2k}}(X) &=  X^{3}-(1+\theta\gamma u_{2k}^{2}(\gamma)-(\theta+\theta')u_{2k}(\gamma)u_{2k-1}(\gamma))X^{2}\notag\\
    &+ (-1+\theta'\gamma u_{2k}^{2}(\gamma)-(\theta+\theta')u_{2k}(\gamma)u_{2k-1}(\gamma))X+1\\
P_{t_{2k+1}}(X) &=  X^{3}-(1+\theta u_{2k+1}^{2}(\gamma)-(\theta+\theta')u_{2k+1}(\gamma)u_{2k}(\gamma))X^{2}\notag\\
   &+(-1+\theta' u_{2k+1}^{2}(\gamma)-(\theta+\theta')u_{2k+1}(\gamma)u_{2k}(\gamma))X+1
\end{align}
\begin{align}
P_{x_{2k}}(X) &=  X^{3}-(1+\theta\beta u_{2k}^{2}(\beta)-(\theta+\theta')u_{2k}(\beta)u_{2k-1}(\beta))X^{2}\notag\\
  & +(-1+\theta'\beta u_{2k}^{2}(\beta)-(\theta+\theta')u_{2k}(\beta)u_{2k-1}(\beta))X+1\\
P_{x_{2k+1}}(X) &=  X^{3}-(1+\theta u_{2k+1}^{2}(\beta)-(\theta+\theta')u_{2k+1}(\beta)u_{2k}(\beta))X^{2}\notag\\
  & +(-1+\theta'u_{2k+1}^{2}(\beta)-(\theta+\theta')u_{2k+1}(\beta)u_{2k}(\beta))X+1
\end{align}
\begin{align}
P_{y_{2k}}(X) &=  X^{3}-(1+\theta\alpha u_{2k}^{2}(\alpha)-(\theta+\theta')u_{2k}(\alpha)u_{2k-1}(\alpha))X^{2}\notag\\
  & +(-1+\theta'\alpha u_{2k}^{2}(\alpha)-(\theta+\theta')u_{2k}(\alpha)u_{2k-1}(\alpha))X+1\\
P_{y_{2k+1}}(X) &=  X^{3}-(1+\theta u_{2k+1}^{2}(\alpha)-(\theta+\theta')u_{2k+1}(\alpha)u_{2k}(\alpha))X^{2}\notag\\
 & +(-1+\theta'u_{2k+1}^{2}(\alpha)-(\theta+\theta')u_{2k+1}(\alpha)u_{2k}(\alpha))X+1
\end{align}
\item Valeur en $1$ des polynômes caractéristiques.\\
\parbox{5cm}{\begin{align*}
P_{t_{2k}}(1)&=  \Delta \gamma u_{2k}^{2}(\gamma),\\
P_{x_{2k}}(1)&=  \Delta \beta u_{2k}^{2}(\beta),\\
P_{y_{2k}}(1)&=  \Delta \alpha u_{2k}^{2}(\alpha),
\end{align*}}
\hfill \parbox{5cm}{\begin{align*}
P_{t_{2k+1}}(1)&=  \Delta u_{2k+1}^{2}(\gamma);\\
P_{x_{2k+1}}(1)&=  \Delta u_{2k+1}^{2}(\beta);\\
P_{y_{2k+1}}(1)&=  \Delta u_{2k+1}^{2}(\alpha).\\
\end{align*}}\\
En particulier, si $\Delta=0$, c'est à dire si la représentation $R(\alpha,\beta,\gamma;l)$ est réductible, on a:
\begin{align}
P_{t_{k}}(X) &=  (X-1)(X^{2}-\theta u_{2k}(\gamma)X-1);\\
P_{x_{k}}(X) &=  (X-1)(X^{2}-\theta u_{2k}(\beta)X-1);\\
P_{y_{k}}(X) &=  (X-1)(X^{2}-\theta u_{2k}(\alpha)X-1).
\end{align}
\item Valeur en $-1$ des polynômes caractéristiques.\\
\begin{align*}
P_{t_{k}}(-1) &=  -(\alpha l-\beta m)u_{2k}(\gamma);\\
P_{x_{k}}(-1) &=  -(\alpha l-\beta m)u_{2k}(\beta);\\
P_{y_{k}}(-1) &=  -(\alpha l-\beta m)u_{2k}(\alpha).\\
\end{align*}
En particulier si $\alpha l = \beta m$, on a:
\begin{align*}
P_{t_{2k}}(X)&=(X+1)(X^{2}-(2+\theta\gamma u_{2k}^{2}(\gamma))X+1),\\
P_{t_{2k+1}}(X)&=(X+1)(X^{2}-(2+\theta u_{2k+1}^{2}(\gamma))X+1);\\
P_{x_{2k}}(X)&=(X+1)(X^{2}-(2+\theta\beta u_{2k}^{2}(\beta))X+1),\\
P_{x_{2k+1}}(X)&=(X+1)(X^{2}-(2+\theta u_{2k+1}^{2}(\beta))X+1);\\
P_{y_{2k}}(X)&=(X+1)(X^{2}-(2+\theta\alpha u_{2k}^{2}(\alpha))X+1),\\
P_{y_{2k+1}}(X)&=(X+1)(X^{2}-(2+\theta u_{2k+1}^{2}(\alpha))X+1).
\end{align*}
De même, si $r=2r_{1}$ est pair (resp. $p=2p_{1}$, resp. $q=2q_{1}$), on a
\[
P_{t_{r_{1}}}(X)=(X+1)(X^{2}-2\left(1-\frac{\Delta}{4-\gamma}\right)X+1)
\]
(resp.
\[
P_{x_{q_{1}}}(X)=(X+1)(X^{2}-2\left(1-\frac{\Delta}{4-\beta}\right)X+1),
\]
resp.
\[
P_{y_{p_{1}}}(X)=(X+1)(X^{2}-2\left(1-\frac{\Delta}{4-\alpha}\right)X+1)).
\]
\end{enumerate}
\end{proposition}
\begin{proof}
1) On ne fait les calculs que pour $P_{t_{k}}(X)$, car ceux pour $P_{x_{k}}(X)$ et $P_{y_{k}}(X)$ sont complètement similaires.\\
On a 
\[
t_{k}=s_{1}(s_{2}s_{3})^{k}=\begin{pmatrix}
-1&\alpha&\beta\\
0&1&0\\
0&0&1
\end{pmatrix}\begin{pmatrix}
1&0&0\\
a&c&e\\
b&d&f
\end{pmatrix}=\begin{pmatrix}
-1+\alpha a+\beta b&\alpha c+\beta d&\alpha e+\beta f\\
a&c&e\\
b&d&f
\end{pmatrix}
\]
où les valeurs de $a,b,c,d,e,f$ sont données par (voir la proposition \ref{b12}):\\
\begin{itemize}
\item si $k$ est pair: $a=\gamma u_{k}^{2}(\gamma)+lu_{k+1}(\gamma)u_{k}(\gamma)$, $b=\gamma u_{k}^{2}(\gamma)+mu_{k}(\gamma)u_{k-1}(\gamma)$, $c=u_{2k+1}(\gamma)$, $d=mu_{2k}(\gamma)$, $e=-lu_{k}(\gamma)$, $f=-u_{2k-1}(\gamma)$;
\item si $k$ est impair $a=u_{k}^{2}(\gamma)+lu_{k+1}(\gamma)u_{k}(\gamma)$, $b=u_{k}^{2}(\gamma)+mu_{k}(\gamma)u_{k-1}(\gamma)$, $c=u_{2k+1}(\gamma)$, $d=mu_{2k}(\gamma)$, $e=-lu_{2k}(\gamma)$, $f=-u_{2k-1}(\gamma)$.
\end{itemize}
Nous avons $P_{t_{k}}(X)=X^{3}-tr^{(1)}(t_{k})X^{2}+tr^{(2)}(t_{k})X+1$ où
\[
tr^{(1)}(t_{k})=-1+c+f+\alpha a+\beta b,
\]
\[
tr^{(2)}(t_{k})=\begin{vmatrix}
-1+\alpha a+\beta b&\alpha c+\beta d\\
a&c
\end{vmatrix}+\begin{vmatrix}
-1+\alpha a+\beta b&\alpha e+\beta f\\
b&f
\end{vmatrix}+\begin{vmatrix}
c&e\\
d&f
\end{vmatrix}
\]
car $\det t_{k}=-1$.
Pour calculer $tr^{(1)}(t_{k})$ et $tr^{(2)}(t_{k})$, nous distinguons deux cas suivant la parité de $k$.
Si $k$ est pair, $c+f=2-4\gamma(4-\gamma)u_{k}^{2}(\gamma)$ d'après (25), d'où
\begin{align*}
tr^{(1)}(t_{k})&=1-\gamma(4-\gamma)u_{k}^{2}(\gamma)+\alpha\gamma u_{k}^{2}(\gamma)+\alpha lu_{k+1}(\gamma)u_{k}(\gamma)+\beta\gamma u_{k}^{2}(\gamma)\\
&+\beta mu_{k}(\gamma)u_{k-1}(\gamma)\\
&=1+\gamma(-4+\gamma+\alpha+\beta)u_{k}^{2}(\gamma)+\alpha lu_{k}(\gamma)(\gamma u_{k}(\gamma)-u_{k-1}(\gamma))\\
&+\beta mu_{k}(\gamma)u_{k-1}(\gamma)\\
&=1+\gamma(-4+\alpha+\beta+\gamma+\alpha l)u_{k}^{2}(\gamma)+(\beta m-\alpha l)u_{k}(\gamma)u_{k-1}(\gamma)
\end{align*}
donc
\[ 
tr^{(1)}(t_{k})=1+\theta\gamma u_{k}^{2}(\gamma)-(\theta+\theta')u_{k}(\gamma)u_{k-1}(\gamma).
\]
Nous avons:
\begin{align*}
tr^{(2)}(t_{k})&=\begin{vmatrix}
-1+\beta b & \beta d\\
a & c
\end{vmatrix}+\begin{vmatrix}
-1+\alpha a & \alpha e\\
b & f
\end{vmatrix}
+cf-de\\
&=  cf-de-(c+f)+\alpha(af-be)+\beta(bc-ad).
\end{align*}
Nous avons:
\[
cf-de=\gamma u_{k}^{2}(\gamma)-u_{2k+1}(\gamma)u_{2k-1}(\gamma)=1
\]
d'après (24).\\
Nous avons:
\begin{align*}
\alpha(af-be)&=\alpha\gamma u_{k}(\gamma)(u_{k-1}(\gamma)u_{2k}(\gamma)-u_{k}(\gamma)u_{2k-1}(\gamma))\\
 &+\alpha lu_{k}(\gamma)(\gamma u_{k}(\gamma)u_{2k}(\gamma)-u_{k+1}(\gamma)u_{2k-1}(\gamma)).
\end{align*}
Nous avons, d'après (25) et (26):
\[
\alpha(af-be)=-\alpha\gamma u_{k}^{2}(\gamma)-\alpha lu_{k}(\gamma)u_{k-1}(\gamma).
\]
Nous avons:
\begin{align*}
\beta(bc-ad) &=  \beta\gamma u_{k}^{2}(\gamma)u_{2k+1}(\gamma)+\beta mu_{k}(\gamma)u_{k-1}(\gamma)u_{2k+1}(\gamma)\\
&-\beta\gamma u_{k+1}(\gamma)u_{k}(\gamma)u_{2k}(\gamma)-\beta m\gamma u_{k}^{2}(\gamma)u_{2k}(\gamma)\\
&= \beta\gamma u_{k}(\gamma)(u_{k}(\gamma)u_{2k+1}(\gamma)-u_{k+1}(\gamma)u_{2k}(\gamma)\\
&+\beta mu_{k}(\gamma)(u_{k-1}(\gamma)u_{2k+1}(\gamma)-\gamma u_{k}(\gamma)u_{2k}(\gamma)\\
&=  \beta\gamma u_{k}(\gamma)A'+\beta mu_{k}(\gamma)B'.
\end{align*}
En changeant $n$ en $-n$ dans (25) (resp. dans (26)), nous trouvons que $A'=-u_{k}(\gamma)$ et $B'=-u_{k+1}(\gamma)$. Il en résulte que:
\[
\beta(bc-ad)=-\beta\gamma u_{k}^{2}(\gamma)-\beta mu_{k+1}(\gamma)u_{k}(\gamma).
\]
Nous obtenons ainsi:
\begin{align*}
tr^{(2)}(t_{k}) &=  -1+\gamma(4-\gamma)u_{k}^{2}(\gamma)-\alpha\gamma u_{k}^{2}(\gamma)-\alpha lu_{k}(\gamma)u_{k-1}(\gamma)-\beta\gamma u_{k}^{2}(\gamma)\\
&-\beta mu_{k+1}(\gamma)u_{k}(\gamma)\\
&=  -1+\gamma(4-\alpha-\beta--\gamma)u_{k}^{2}(\gamma)-\beta mu_{k+1}(\gamma)u_{k}(\gamma)-\alpha lu_{k}(\gamma)u_{k-1}(\gamma)\\
&=  -1+\gamma(4-\alpha-\beta-\gamma-\beta m)u_{k}^{2}(\gamma)+(\beta m-\alpha l)u_{k}(\gamma)u_{k-1}(\gamma)\\
&=  -1+\theta'\gamma u_{k}^{2}(\gamma)-(\theta+\theta')u_{k}(\gamma)u_{k-1}(\gamma).
\end{align*}
Nous avons ainsi le résultat annoncé pour $P_{t_{k}}(X)$ lorsque $k$ est pair.

Si $k$ est impair, $c+f=2-(4-\gamma)u_{k}^{2}(\gamma)$ d'après (26), d'où
\begin{align*}
tr^{(1)}(t_{k}) &=  1-(4-\gamma)u_{k}^{2}(\gamma)+\alpha u_{k}^{2}(\gamma)+\alpha lu_{k+1}(\gamma)u_{k}(\gamma)+\beta u_{k}^{2}(\gamma)\\
&+\beta mu_{k}(\gamma)u_{k-1}(\gamma)\\
&= 1+(-4+\alpha+\beta+\gamma)u_{k}^{2}(\gamma)+\alpha lu_{k+1}(\gamma)u_{k}(\gamma)+\beta mu_{k}(\gamma)u_{k-1}(\gamma)\\
&=  1+(-4+\alpha+\beta+\gamma+\alpha l)u_{k}^{2}(\gamma)+(\beta m-\alpha l)u_{k}(\gamma)u_{k-1}(\gamma)\\
&=  1+\theta u_{k}^{2}(\gamma)-(\theta+\theta')u_{k}(\gamma)u_{k-1}(\gamma).
\end{align*}
Nous avons
\[
tr^{(2)}(t_{k})=(cf-de)-(c+f)+\alpha(af-be)+\beta(bc-ad).
\]
D'après $(26)$, 
$cf-de=\gamma u_{2k}^{2}(\gamma)-u_{2k+1}(\gamma)u_{2k-1}(\gamma)=1$.\\
Nous avons:
\begin{align*}
\alpha(af-be) &=  -\alpha u_{k}^{2}(\gamma)u_{2k-1}(\gamma)-\alpha lu_{k+1}(\gamma)u_{k}(\gamma)u_{2k-1}(\gamma)\\
&+\alpha\gamma u_{k}(\gamma)u_{k-1}(\gamma)u_{2k}(\gamma)+\alpha lu_{k}^{2}(\gamma)u_{2k}(\gamma)\\
&=  \alpha u_{k}(\gamma)(-u_{k}(\gamma)u_{2k-1}(\gamma)+\gamma u_{k-1}(\gamma)u_{2k}(\gamma)\\
& +\alpha lu_{k}(\gamma)(u_{k}(\gamma)u_{2k}(\gamma)-u_{k+1}(\gamma)u_{2k-1}(\gamma)\\
&=  -\alpha u_{k}^{2}(\gamma)-\alpha lu_{k}(\gamma)u_{k-1}(\gamma)
\end{align*}
en utilisant (33).\\
Nous avons:
\begin{align*}
\beta(bc-ad)&=  \beta u_{k}^{2}(\gamma)u_{2k+1}(\gamma)+\beta mu_{k}(\gamma)u_{k-1}(\gamma)u_{2k+1}(\gamma)\\
& -\beta\gamma u_{k+1}(\gamma)u_{k}(\gamma)u_{2k}(\gamma)-\beta mu_{k}^{2}(\gamma)u_{2k}(\gamma)\\
&=  \beta u_{k}(\gamma)(u_{k}(\gamma)u_{2k+1}(\gamma)-\gamma u_{k+1}(\gamma)u_{2k}(\gamma))\\
&+\beta mu_{k}(\gamma)(u_{k-1}(\gamma)u_{2k+1}(\gamma)-u_{k}(\gamma)u_{2k}(\gamma))\\
&=  \beta u_{k}(\gamma)C+\beta mu_{k}(\gamma)D.
\end{align*}
En changeant $n$ en $-n$ dans (33) (resp. (34)), nous voyons que $C=-u_{k}(\gamma)$ et $D=-u_{k+1}(\gamma)$.
Il en résulte que:
\[
\beta(bc-ad)=-\beta u_{k}^{2}(\gamma)-\beta mu_{k+1}(\gamma)u_{k}(\gamma)
\]
Nous obtenons ainsi, de la même manière que dans le cas $k$ pair:
\[
tr{(2)}(t_{k})=-1+\theta'u_{k^{2}}(\gamma)-(\theta+\theta')u_{k}(\gamma)u_{k-1}(\gamma).
\]
Nous avons le résultat pour $P_{t_{k}}(X)$ lorsque $k$ est impair.

Les calculs pour trouver $P_{x_{k}}(X)$ et $P_{y_{k}}(X)$ sont identiques à ceux faits pour $P_{t_{k}}(X)$.\\
2) Nous avons:
\begin{align*}
P_{t_{2k}}(1) &=  -\theta\gamma u_{k}^{2}(\gamma)+(\theta+\theta')u_{2k}(\gamma)u_{2k-1}(\gamma)
+\theta'\gamma u_{2k}^{2}(\gamma)-(\theta+\theta')u_{2k}(\gamma)u_{2k-1}(\gamma)\\
&=  (\theta'-\theta)\gamma u_{2k}^{2}(\gamma)\\
&=  \Delta\gamma u_{2k}^{2}(\gamma).
\end{align*}
On trouve de même $P_{t_{2k+1}}(1)=\Delta u_{2k+1}^{2}(\gamma)$.
Si $\Delta=0$, alors $P_{t_{k}}(1)=0$ et nous obtenons sans difficultés:
\begin{itemize}
\item $P_{t_{k}}(X)=(X-1)(X^{2}-\theta u_{2k}(\gamma)X-1);$
\item $P_{x_{k}}(X)=(X-1)(X^{2}-\theta u_{2k}(\beta)X-1);$
\item $P_{y_{k}}(X)=(X-1)(X^{2}-\theta u_{2k}(\alpha)X-1).$
\end{itemize}
3) Les calculs ne présentent pas de difficultés:
\begin{align*}
P_{t_{2k}}(-1) &=  -1-(1+\theta\gamma u_{2k}^{2}(\gamma)-(\theta+\theta')u_{2k}(\gamma)u_{2k-1
}(\gamma)\\
&-(-1+\theta'\gamma u_{2k}^{2}(\gamma)-(\theta+\theta')u_{2k}(\gamma)u_{2k-1}(\gamma))+1\\
&=  -(\theta+\theta')\gamma u_{2k}^{2}(\gamma)+2(\theta+\theta')u_{2k}(\gamma)u_{2k-1}(\gamma)\\
&=  -(\theta+\theta')u_{2k}(\gamma)(\gamma u_{2k}(\gamma)-u_{2k-1}(\gamma)-u_{2k-1}(\gamma))\\
&=  -(\theta+\theta')u_{2k}(\gamma)(u_{2k+1}(\gamma)-u_{2k-1}(\gamma))\\
&=  -(\theta+\theta')u_{4k}(\gamma).
\end{align*}
\begin{align*}
P_{t_{2k+1}}(-1) &=  -1-(1+\theta u_{2k+1}^{2}(\gamma)-(\theta+\theta')u_{2k+1}(\gamma)u_{2k}(\gamma)\\
  & -(-1+\theta'u_{2k+1}^{2}(\gamma)-(\theta+\theta')u_{2k+1}(\gamma)u_{2k}(\gamma))+1\\
&=  -(\theta+\theta')u_{2k+1}(\gamma)(u_{2k+1}(\gamma)-u_{2k}(\gamma)-u_{2k}(\gamma))\\
&=   -(\theta+\theta')u_{4k+2}(\gamma).
\end{align*}
Comme $\theta+\theta'=\alpha l-\beta m$, nous en déduisons toutes les formules annoncées.

Nous supposons maintenant que $r$ est pair: $r=2r'$ et nous ne faisons les calculs que pour $t_{r'}$ car ceux pour $x_{q'}$ et $y_{p'}$ (avec $q=2q'$ et $p=2p'$) sont identiques.

Nous avons en effectuant la division de $P_{t_{k}}(X)$ par $X+1$:
\begin{align*}
P_{t_{2k}}(X)&=(X+1)(X^{2}-(2+\theta\gamma u_{2k}^{2}(\gamma)-(\theta+\theta')u_{2k}(\gamma)u_{2k-1}(\gamma))X\\
&+1+(\theta+\theta')u_{4k}(\gamma))-(\theta+\theta')u_{4k}(\gamma)
\end{align*}
\begin{align*}
P_{t_{2k+1}}(X) &=  (X+1)(X^{2}-(2+\theta u_{2k+1}^{2}(\gamma)-(\theta+\theta')u_{2k+1}(\gamma)u_{2k}(\gamma))X\\
&+1+(\theta+\theta')u_{4k+2}(\gamma))-(\theta+\theta')u_{4k+2}(\gamma).
\end{align*}
Si $r=4k$, alors
\[
P_{t_{2k}}(X)=(X+1)(X^{2}-(2+\theta\gamma u_{2k}^{2}(\gamma)-(\theta+\theta')u_{2k}(\gamma)u_{2k-1}(\gamma))X+1)
\]
d'après la proposition \ref{a14}, nous avons $\gamma u_{2k}^{2}(\gamma)=\frac{4}{4-\gamma}$ et $u_{2k}(\gamma)u_{2k-1}(\gamma)=\frac{2}{4-\gamma}$, nous obtenons ainsi 
\[
P_{t_{2k}}(X)=(X+1)(X^{2}-2\left (1-\frac{\Delta}{4-\gamma}\right )X+1)
\]
car $\Delta=\theta-\theta'$.\\
Si $r=4k+2$, alors
\[
P_{t_{2k+1}}(X)=(X+1)(X^{2}-(2+\theta u_{2k+1}^{2}(\gamma)-(\theta+\theta')u_{2k+1}(\gamma)u_{2k}(\gamma))X+1)
\]
d'après la proposition \ref{a14}, nous avons $u_{2k+1}^{2}(\gamma)=\frac{4}{4-\gamma}$ et $u_{2k+1}(\gamma)u_{2k}(\gamma)=\frac{2}{4-\gamma}$, nous obtenons ainsi
\[
P_{t_{2k+1}}(X)=(X+1)(X^{2}-2\left (1-\frac{\Delta}{4-\gamma}\right )X+1.
\]
Il en résulte que si $r=2r'$, nous avons:
\[
P_{r'}(X)=(X+1)(X^{2}-2\left (1-\frac{\Delta}{4-\gamma}\right )X+1).
\]
\end{proof}
\section{Exemples lorsque le corps $K$ est réel.}
Dans cette section on donne quelques exemples de groupes de réflexion finis réels de petit rang pour montrer comment la théorie précédente s'applique et ceci uniquement lorsque la représentation $R$ est irréductible. Les autres parties de ce travail seront consacrées au cas où $R$ est réductible.

Dans la suite, nous nous intéresserons \textbf{principalement} au cas où le groupe de Coxeter $W$ est de rang $3$. On a $S:=\{s_{1},s_{2},s_{3}\}$ et
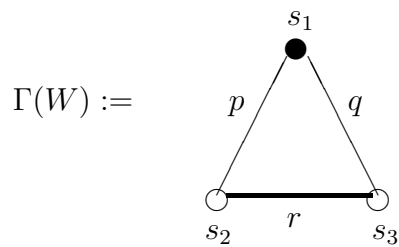
\[
\begin{picture}(150,68)
\put(-30,35){$\Gamma(W):=$}
\put(60,63){$s_{1}$}
\put(63,55){\circle*{7}}
\put (37,5){\circle{7}}
\put (90,5){\circle{7}}
\put(33,-8){$s_{2}$}
\put(88,-8){$s_{3}$}
\linethickness{0.6mm}
\put(40,6,5){\line(1,0){48}}
\put(37,6,8){\line(1,2){23}}
\put(90,6,8){\line(-1,2){23}}
\put(41,35){$p$}
\put(80,35){$q$}
\put(60,-4){$r$}
\end{picture}
\]
Pour préciser dans la notation l'ordre de $s_{i}s_{j}$, nous notons $W(p,q,r)$ le groupe $W$ avec $p:=m_{s_{1}s_{2}}$, $q:=m_{s_{1}s_{3}}$ et $r:=m_{s_{2}s_{3}}$. On a $p,q\geqslant 3$ et $r\geqslant 2$. Ayant choisi $\alpha$ une racine de $V_{p}(X)$, $\beta$ une racine de $v_{q}(X)$ et $\gamma$ une racine de $v_{r}(X)$ et $l$ et $m$ dans $K=K_{0}(l)$ tels que $lm=\gamma$ (si $r=2$, $l=m=\gamma =0$), on a la représentation de réflexion $R \to GL(M)$ obtenue grâce à la construction fondamentale. On pose $G:=Im \,R=G(\alpha,\beta,\gamma ;l)$ et $R:=R(\alpha,\beta,\gamma;l)$ et on appelle $\textit{P}(G):=\textit{P}(\alpha,\beta,\gamma ;\alpha l,\beta m)$ le système de paramètres de la représentation $R$ de $W$ (ce qui représente une légère modification de la définition initiale). On a :
\[
\Delta=\Delta(G)=8-2\alpha-2\beta-2\gamma -(\alpha l+\beta m).
\]
On a vu dans les sections précédentes les résultats suivants:

\begin{proposition}\hfill
\begin{enumerate}
  	\item \begin{enumerate}
  		\item $Z(G)$ est un groupe cyclique d'ordre un diviseur de $6$; de plus\\ $Z(G)\cap G^{+} \subset Z(G^{+})$ est cyclique d'ordre $1$ ou $3$.
  		\item Si $\Delta(G)=0$, $Z(G)=\{1\}$.
 \end{enumerate}
  	\item Si $dim \,\Phi=1$ et $\Phi=<\varphi>$ alors:
  \begin{enumerate}
  		\item Les deux conditions suivantes sont équivalentes:
  		\begin{enumerate}
  			\item $\varphi$ est symétrique et $\theta=id_{K}$;
  			\item $\alpha l=\beta m$
  		\end{enumerate}
  		\item $\varphi$ est dégénérée si et seulement si $\Delta(G)=0$.
 \end{enumerate}
  	\item On suppose que $\alpha l = \beta m$, alors $K \subset \mathbb{R}$ et:
  \begin{enumerate}
  		\item $dim \,\Phi =1$;
  		\item $\varphi$ est symétrique et $\theta=id_{K}$;
  		\item \begin{enumerate}
  			\item si $\Delta > 0$, $\varphi$ est définie positive;
  			\item si $\Delta = 0$, $\varphi$ est positive dégénérée, de noyau $C_{M}(G)$;
  			\item si $\Delta < 0$, $\varphi$ est de signature $(2,1)$.
		\end{enumerate}
\end{enumerate}
	\item On suppose que $\Delta(G)=0$ et que $l \not \in K_{0}$. Alors on a:
\begin{enumerate}
  \item $dim \Phi =1$;
  \item $K$ est un corps de décomposition du polynôme 
  \[
  Q(X)=X^{2}-2(4-\alpha -\beta -\gamma)X+\alpha\beta\gamma
  \]
   sur le corps $K_{0}$ et $\theta$ est l'élément générateur du groupe de Galois $\mathcal{G}(K/K_{0})$.
   \end{enumerate}
  \item Le système de paramètres de la représentation contragrédiente $R^{*}$ est: $\textit{P}(G):=\textit{P}(\alpha,\beta,\gamma ;\beta m,\alpha l)$ si $R$ est irréductible.
\end{enumerate}
\end{proposition}
Dans toute la suite nous utiliserons de nombreux résultats des sections 1 et 2.
\subsection{Questions de fidélité.}
Nous gardons les hypothèses et notations précédentes. En général, nous ne savons pas si la représentation $R$ est fidèle.
\begin{proposition}
Le sous-groupe $ker \,R$ de $W$ est sans torsion.
\end{proposition}
\begin{proof}
Soit $H$ un sous-groupe fini de $ker\, R$. Alors d'après \cite{B} (page 130, exercice 2 d)), il existe $i,j \, (i\neq j)$, il existe $g\in W$ tels que $gHg^{-1}\subset <s_{i},s_{j}>$. Mais l'hypothèse faite sur $R$ : $<R(s_{i}),R(s_{j})> \simeq <s_{i},s_{j}>$ si $i\neq j$ montre que $ker\,R\,\cap <s_{i},s_{j}>={1}$. Il en résulte aussitôt que $H={1}$. 
\end{proof}
\begin{proposition}
Soit $W(p,q,r),S)$ un système de Coxeter de rang $3$ avec $p,q,r \geqslant 3$. On pose $S:=\{s_{1},s_{2},s_{3}\}$. Soit $s$ une réflexion de $W(p,q,r)$ contenue dans $<s_{j},s_{k}>-\{s_{j},s_{k}\}$. Alors $s_{i}s$ est d'ordre infini si $|\{i,j,k\}|=3$.
\end{proposition}
\begin{proof}
Soit $R(\alpha,\beta,\gamma;l)$ une représentation de réflexion de $W(p,q,r)$ qui satisfait aux conditions (R). Supposons, pour fixer les idées, que $i=1$. Il existe un entier $k$ tel que $s=s_{2}(s_{2}s_{3})^{k}$. On a (voir la section 2):
\[
C(s_{1},s_{2}(s_{2}s_{3})^{2k})=\beta\gamma u_{2k}^{2}(\gamma)+\alpha u_{2k-1}^{2}(\gamma)+(\alpha l+\beta m)u_{2k}(\gamma)u_{2k-1}(\gamma)
\]
\[
C(s_{1},s_{2}(s_{2}s_{3})^{2k+1})=\beta u_{2k+1}^{2}(\gamma)+\alpha\gamma u_{2k}^{2}(\gamma)+(\alpha l+\beta m)u_{2k+1}(\gamma)u_{2k}(\gamma).
\]
Avec les hypothèses faites, on a toujours $u_{2k}(\gamma)u_{2k-1}(\gamma)u_{2k+1}(\gamma)\neq 0$. En effet $s_{2}s_{3}$ est d'ordre $r$ et comme $s \not \in \{s_{2},s_{3}\}$, nécessairement $2k-1$, $2k$ et $2k+1$ sont différents de $0$ et de $r$. Le produit des deux réflexions $s_{1}$ et $s$ est d'ordre fini si et seulement si il existe un nombre rationnel $n$, $0\leqslant n < \frac{1}{2}$ tel que $C(s_{1},s)=4 \cos^{2} \pi n$. Nous pouvons choisir $l$ (donc aussi $m$ puisque $lm=\gamma$) de telle sorte que $C(s_{1},s)$ ne soit pas de la forme $4 \cos^{2} \pi n$. Il en résulte que dans $G$, $s_{1}s$ est d'ordre infini. Il en est de même dans $W(p,q,r)$.
\end{proof}
\begin{corollary}(De la démonstration de la proposition 21.)
Soit $(W(p,q,r),S)$ un système de Coxeter de rang $3$ avec $p,q,r \geqslant 3$. Il y a une infinité dénombrable au moins de représentations de réflexion de $W(p,q,r)$ qui ne sont pas fidèles.
\end{corollary}
\begin{proof}
Dans la démonstration précédente, nous prenons $n$ rationnel tel que $0 < n <\frac{1}{2}$ et nous calculons $l$ de telle sorte que $C(s_{1},s)=4\cos^{2} n\pi$. Cela nous donne une ou deux valeurs pour $l$ car alors nous connaissons $\alpha l+\beta m$ et $\alpha l\beta m=\alpha\beta\gamma$.
\end{proof}
\subsection{Un résultat général.}
Dans toute la suite nous aurons besoin des formules suivantes démontrées dans la section 2 (formules $\mathcal{C}_{i}$).
\begin{align}
C(s_{1},s_{2}^{s_{3}})& =&C(s_{2},s_{1}^{s_{3}})&=&(l+1)(\alpha+\beta m) &=&\alpha+\beta\gamma+(\alpha l+\beta m)\\
C(s_{1},s_{3}^{s_{2}})&=&C(s_{3},s_{1}^{s_{2}})&=&(m+1)(\beta+\alpha l) &=&\beta+\alpha\gamma+(\alpha l+\beta m)\\
C(s_{2},s_{3}^{s_{1}})&=&C(s_{3},s_{2}^{s_{1}})&=&(\alpha+m)(\beta+l) &=&\gamma+\alpha\beta+(\alpha l+\beta m)
\end{align}
\begin{align}
C(s_{1},s_{2}^{s_{3}s_{2}}) &=&(l+u_{3}(\gamma)(\beta m+\alpha u_{3}(\gamma))&=&\beta\gamma+\alpha u_{3}(\gamma)^{2}+u_{3}(\gamma(\alpha l+\beta m)\\
C(s_{1},s_{3}^{s_{2}s_{3}}) &=&(m+u_{3}(\gamma))(\alpha l+\beta u_{3}(\gamma))&=&\alpha\gamma+\beta u_{3}(\gamma)^{2}+u_{3}(\gamma)(\alpha l+\beta m)\\
C(s_{2},s_{1}^{s_{3}s_{1}}) &=&(l+u_{3}(\beta))(\beta m+\alpha u_{3}(\beta))&=&\beta\gamma+\alpha u_{3}(\beta)^{2}+u_{3}(\beta)(\alpha l+\beta m)\\
C(s_{2},s_{3}^{s_{1}s_{3}})&=&(\beta+lu_{3}(\beta))(\alpha+mu_{3}(\beta))&=&\alpha\beta+\gamma u_{3}(\beta)^{2}+u_{3}(\beta)(\alpha l+\beta m)\\
C(s_{3},s_{1}^{s_{2}s_{1}})&=&(\beta u_{3}(\alpha)+\alpha l)(m+u_{3}(\alpha))&=&\alpha\gamma+\beta u_{3}(\alpha)^{2}+u_{3}(\alpha)(\alpha l+\beta m)\\
C(s_{3},s_{2}^{s_{1}s_{2}})&=&(\alpha+mu_{3}(\alpha))(\beta+lu_{3}(\alpha))&=&\alpha\beta+\gamma u_{3}(\alpha)^{2}+u_{3}(\alpha)(\alpha l+\beta m)
\end{align}
Si $\Delta(G)=0$ et si\\
$p=2p_{1},$ $C(s_{3},s_{1}(s_{1}s_{2})^{p_{1}})=4-\beta$, $C(s_{3},s_{2}(s_{1}s_{2})^{p_{1}})=4-\gamma$\\
$q=2q_{1},$ $C(s_{2},s_{1}(s_{1}s_{3})^{q_{1}})=4-\alpha$, $C(s_{2},s_{1}(s_{1}s_{3})^{q_{1}})=4-\gamma$\\
$p=2r_{1},$ $C(s_{1},s_{2}(s_{2}s_{3})^{r_{1}})=4-\alpha$, $C(s_{1},s_{3}(s_{2}s_{3})^{r_{1}})=4-\beta$.

L'un des buts de ce numéro est de démontrer le résultat suivant:
\begin{theorem}
Soit $W(p,p,r)$ un groupe de Coxeter de rang $3$ avec $r$ impair. Alors il existe un quotient de ce groupe isomorphe au groupe de Coxeter de rang $3$, $W(p,r,2)$. Réciproquement, tout groupe de Coxeter de rang $3$, $W(p,r,2)$ avec $r$ impair est isomorphe à un quotient de $W(p,p,r)$.
\end{theorem}
Nous commen\c cons par quelques résultats techniques.\\
On suppose $r$ impair: $r=2r_{1}+1$. D'après la proposition 11 $\exists \epsilon\in \{+1,-1\}$ tel que si $\gamma$ est une racine de $v_{r}(X)$ (avec $\epsilon \in \{-1,+1\}$):
\begin{itemize}
  \item si $r_{1}$ est pair, on ait $u_{r_{1}+1}(\gamma)+\epsilon\sqrt{\gamma}u_{r_{1}}(\gamma)=0$
  \item si $r_{1}$ est impair, on ait $\epsilon\sqrt{\gamma}u_{r_{1}+1}(\gamma)+u_{r_{1}}(\gamma)=0$.
 \end{itemize}
 On a alors:
 \begin{proposition}
Soient $W(p,q,r)$ un groupe de Coxeter de rang $3$ et $R(\alpha,\beta,\gamma ;l)$ une de ses représentations de réflexion. On suppose $r$ impair: $r=2r_{1}+1$. Les deux conditions suivantes sont équivalentes:
\begin{enumerate}
  \item $s_{1}(s_{2}s_{3})^{r_{1}}s_{2}$ est d'ordre $2$;
  \item $q=p$, $\beta = \alpha$ et $l=m=\epsilon\sqrt{\gamma}$ (donc $\alpha l=\beta m$).
 \end{enumerate}
\end{proposition}
\begin{proof}
1) On montre que (1) implique (2).\\
On a $(s_{2}s_{3})^{r_{1}}s_{2}=s_{2}(s_{2}s_{3})^{r_{1}+1}$. On applique les résultats de la proposition 17.\\
Si $r_{1}$ est pair, $C(s_{1},s_{2}(s_{2}s_{3})^{r_{1}+1})=(\alpha l u_{r_{1}}(\gamma)+\beta u_{r_{1}+1}(\gamma))(u_{r_{1}+1}(\gamma)+m u_{r_{1}}(\gamma))$ donc $s_{1}s_{2}(s_{2}s_{3})^{r_{1}+1}$ d'ordre 2 équivaut au système:
\begin{align}
\alpha lu_{r_{1}}(\gamma)+\beta u_{r_{1}+1}(\gamma) &= 0 \notag\\
m u_{r_{1}}(\gamma)+ u_{r_{1}+1}(\gamma) &= 0. \notag
\end{align}
On obtient d'abord $\alpha l=\beta m$. On a aussi l'équation $u_{r_{1}+1}(\gamma)+\epsilon\sqrt{\gamma}u_{r_{1}}(\gamma)=0$ donc $m=\epsilon\sqrt{\gamma}$. Comme $lm=\gamma$ on voit que $l=m$, $\beta = \alpha$ et $q=p$.\\
Si $r_{1}$ est impair $C(s_{1},s_{2}(s_{2}s_{3})^{r_{1}+1})=(\alpha u_{r_{1}}(\gamma)+\beta m u_{r_{1}+1}(\gamma))(u_{r_{1}}(\gamma)+l u_{r_{1}+1}(\gamma))$ donc $s_{1}s_{2}(s_{2}s_{3})^{r_{1}+1}$ d'ordre 2 équivaut au système :
\begin{align}
\alpha u_{r_{1}}(\gamma)+\beta m u_{r_{1}+1}(\gamma) &= 0 \notag\\
u_{r_{1}}(\gamma)+l u_{r_{1}+1}(\gamma) &= 0\notag
\end{align}
On obtient d'abord $\alpha l=\beta m$. On a aussi l'équation $\epsilon\sqrt{\gamma}u_{r_{1}+1}(\gamma)+u_{r_{1}}(\gamma)=0$ donc $l=\epsilon\sqrt{\gamma}$. Comme ci-dessus, $l=m$, $\beta=\alpha$ et $q=p$.\\
2) On montre que (2) implique (1).\\
Si $\beta=\alpha$ et $l=m=\epsilon\sqrt{\gamma}$ alors si $r_{1}$ est pair, le système:
\begin{align}
\alpha lu_{r_{1}}(\gamma)+\beta u_{r_{1}+1}(\gamma) &= 0 \notag\\
m u_{r_{1}}(\gamma)+ u_{r_{1}+1}(\gamma) &= 0. \notag\\
u_{r_{1}+1}(\gamma)+\epsilon\sqrt{\gamma}u_{r_{1}}(\gamma) &=0 \notag
\end{align}
se réduit à la dernière équation qui est satisfaite par hypothèse donc $s_{1}(s_{2}s_{3})^{r_{1}}s_{2}$ est d'ordre $2$.\\
La démonstration est identique si $r_{1}$ est impair.
\end{proof}
Nous avons deux corollaires à ce résultat.
\begin{corollary}
On pose $\gamma_{k}:=4\cos^{2} \frac{k\pi}{r}$, $0\leqslant k < \frac{r}{2}$, $(k,r)=1$. Avec les hypothèses et notations de la proposition 15, on a:
\[
G(\alpha,\alpha,\gamma_{k};\epsilon\sqrt\gamma_{k}) \simeq G(\alpha,\gamma_{k'},0;0)
\]
où
\begin{itemize}
  \item si $k=2k_{1}$, on a $k'=k_{1}$;
  \item si $k=2k_{1}+1$, on a $k'=r_{1}-k_{1}$.
  \end{itemize}
\end{corollary}
\begin{proof}
Posons $s'_{3}=(s_{2}s_{3})^{r_{1}}s_{2}$. On a $< s_{2}, s'_{3}>=<s_{2},s_{3}>$ d'où\\ $G=<s_{1},s_{2},s'_{3}>$. En utilisant la proposition 16 nous voyons que $s'_{3}(a_{1})=a_{1}$, $s'_{3}(a_{2})=a_{3}$ et $s'_{3}(a_{3})=a_{2}$. Nous avons le diagramme:
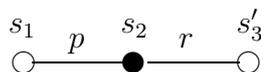
\[
\begin{picture}(150,88)
\put(29,42){\circle{7}}
\put(32,42){\line(1,0){30}}
\put(65,42){\circle*{7}}
\put(70,42){\line(4,0){30}}
\put(103,42){\circle{7}}
\put(24,52){$s_{1}$}
\put(61,52){$s_{2}$}
\put(99,52){$s'_{3}$}
\put(44,47){$p$}
\put(80,47){$r$}
\end{picture}
\]
Nous posons $a'_{1}=s_{1}(a_{2})-a_{2}=\alpha a_{1}$ et $a'_{3}=s'_{3}(a_{2})-a_{2}=a_{3}-a_{2}$ et nous avons la relation $s_{2}(a'_{3})-a'_{3}=(2+\epsilon\sqrt{\gamma})a_{2}$ donc $C(s_{2},s'_{3})=2+\epsilon\sqrt{\gamma}$. Un calcul simple montre que si $\gamma =\gamma_{k}$ on a:
\begin{itemize}
  \item si $k=2k_{1}$ alors $\epsilon=+1$ et $2+\epsilon\sqrt{\gamma}=4\cos^{2}\frac{k_{1}\pi}{r}$;
  \item si $k=2k_{1}+1$ alors $\epsilon=-1$ et $2+\epsilon\sqrt{\gamma}=4\cos^{2}\frac{(r_{1}-k_{1})\pi}{r}$
 \end{itemize}
 d'où le résultat en utilisant la proposition 11.
\end{proof}
\begin{corollary}
On pose $\tau:=\frac{3+\sqrt{5}}{2}=4\cos^{2}\frac{\pi}{5}$. C'est la plus grande racine de $u_{5}(X)$, l'autre étant $3-\tau=\frac{3-\sqrt{5}}{2}=4\cos{2}\frac{2\pi}{5}$. On a:
\begin{enumerate}
  \item $G(1,1,1;-1)\simeq W(A_{3})$;
  \item $G(2,2,2;-1)\simeq W(B_{3})$;
  \item $G(3,3,3;-1)\simeq W(\tilde{G_{2}})$;
  \item $G(1,1,\tau;1-\tau)\simeq W(H_{3})$;
  \item $G(\tau,\tau,1;-1)\simeq W(H_{3})$.
\end{enumerate}
\end{corollary}
\begin{proof}
C'est clair d'après le corollaire 8.
\end{proof}
\begin{proof} (Du théorème 6)\\
Par hypothèse $r$ est impair et $r=2r_{1}+1$.\\
1) Distinguons deux cas suivant que $p=r$ ou non.\\
\textbf{Premier cas} $p\neq r$. Nous partons de la représentation suivante de $W(p,p,r)$: on suppose que $\alpha=\beta=4\cos^{2}\frac{\pi}{p}$, $\gamma=4\cos^{2}\frac{\pi}{r}$ et $\alpha l=\beta m$. D'après le corollaire 4, comme $k=1$ on a $k_{1}=0$ et le groupe obtenu en ajoutant la relation $(s_{1}(s_{2}s_{3})^{r_{1}}s_{2})^{2}=1$ est isomorphe à $G(\alpha,\gamma_{r},0;0)$. maintenant le groupe de Galois $\textit{G}(K_{0}/\mathbb{Q})$ est commutatif et il existe $\sigma \in \textit{G}(K_{0}/\mathbb{Q})$ tel que $\sigma(\alpha)=\alpha$ et $\sigma(\gamma_{r_{1}})=\gamma_{1}$ car les polynômes $v_{p}(X)$ et $v_{q}(X)$ sont premiers entre eux et ils sont irréductibles. Il en résulte que $R(\alpha,\gamma_{r},0;0)$ est conjuguée de la représentation géométrique de $W(p,r,2)$, d'où le résultat dans ce cas.\\
\textbf{Deuxième cas} $p=r$. Nous partons de la représentation suivante de $W(p,p,p)$: on suppose que $\alpha=\beta=4\cos^{2}\frac{\pi}{p}$, $\gamma=4\cos^{2}\frac{2\pi}{r}$ et $\alpha l=\beta m$. D'après le corollaire 4, comme $k=2$, nous voyons que le groupe obtenu en ajoutant la relation $(s_{1}(s_{2}s_{3})^{r_{1}}s_{2})^{2}=1$ est isomorphe à $G(\alpha,\gamma_{1},0;0)$ d'où le résultat dans ce cas.\\
2) Réciproquement nous partons de la représentation géométrique de $W(p,r,2)$: $\alpha=4\cos^{2}\frac{pi}{p}$, $\gamma_{1}=4\cos^{2}\frac{pi}{r}$ et $l=m=0$. Si nous appliquons le corollaire 4 au groupe $W(p,p,r)$ et à sa représentation $R(\alpha,\alpha,\gamma_{2};\sqrt{\gamma_{2}})$ nous avons le résultat.
\end{proof}
\subsection{Présentations du groupe de Coxeter $W(H_{3})$.}
Nous montrons comment obtenir le groupe de Coxeter $W(H_{3})$ comme quotient d'un groupe de Coxeter de rang $3$.\\
\begin{notation}
Soit $W(p,q,r)$ un groupe de Coxeter de rang $3$. Si $G$ est un quotient de $W(p,q,r)$, une présentation de $G$ est notée:
\[
(w(p,q,r),R_{1},\cdots,R_{n})
\]
où $w(p,q,r)$ est la présentation de $W(p,q,r)$ comme groupe de Coxeter:
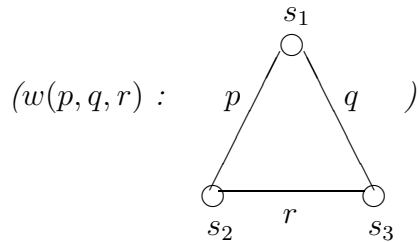
\[
\begin{picture}(150,68)
\put(-30,35){($w(p,q,r)$:}
\put(60,63){$s_{1}$}
\put(63,55){\circle{7}}
\put (37,5){\circle{7}}
\put (90,5){\circle{7}}
\put(35,-8){$s_{2}$}
\put(88,-8){$s_{3}$}
\put(39,6,5){\line(1,0){48}}
\put(36,6,7){\line(1,2){23}}
\put(90,6,8){\line(-1,2){23}}
\put(41,35){$p$}
\put(80,35){$q$}
\put(60,-4){$r$}\put(100,35){)}
\end{picture}
\]
et où les $R_{i}$ sont des mots en $s_{1}$, $s_{2}$ et $s_{3}$.\\
Si $G$ est un groupe de réflexion, dans la mesure du possible, nous choisissons les $R_{i}$ sous la forme $(st)^{m}$ où $s$ et $t$ sont des réflexions de $W(p,q,r)$ et aussi de telle sorte que l'on puisse calculer les éléments du système de paramètres $\textit{P}(G):=\textit{P}(\alpha,\beta,\gamma ;\alpha l,\beta m)$, puisque celui-ci permet de construire la représentation $R$.
\end{notation}
Nous connaissons la structure du groupe $W(H_{3})$: $W(H_{3})\simeq C_{2}\times A_{5}$ où $A_{5}$ est le groupe alterné de degré $5$. Nous connaissons aussi son corps de définition en tant que groupe de réflexion: on a $K=K_{0}=\mathbb{Q}(\tau)$. Nous avons alors:
\begin{theorem}
Les présentations de $W(H_{3})$ comme quotient d'un groupe de Coxeter $W(p,q,r)$ sont les suivantes où l'on a posé $t=s_{1}s_{2}s_{3}$: ($t=t_{1}$ dans la notation 2)
\begin{enumerate}
  \item $(p,q,r)=(3,5,2)$ (présentation de Coxeter)\\
  \[
  (w(3,5,2))
  \]
   on a $t^{5}=-id_{M}$ et $\Delta=2(3-\tau)$. On obtient la représentation $R(1,\tau,0;0)$ de $W(3,5,2)$.
  \item $(p,q,r)=(5,5,2)$
  \[
  (w(5,5,2),\,(s_{2}s_{3}^{s_{1}})^{3}=1)
  \]
  On a $t^{3}=-id_{M}$ et $\Delta=2$. On obtient la représentation $R(\tau,3-\tau,0;0)$ de $W(5,5,2)$.
  \item $(p,q,r)=(3,3,5)$
  \[
 (w(3,3,5), \,(s_{1}s_{2}^{s_{3}s_{2}})^{2}=1)
  \]
  On a $t^{3}=-id_{M}$ et $\Delta=2$. On obtient la représentation $R(1,1,\tau;1-\tau)$ de $W(3,3,5)$.
  \item $(p,q,r)=(5,5,3)$
  \[
  (w(5,5,3),\, (s_{1}s_{2}^{s_{3}})^{2}=1)
  \]
  On a $t^{5}=-id_{M}$ et $\Delta=2(3-\tau)$. On obtient la représentation $R(\tau,\tau,1;-1)$ de $W(5,5,3)$.
  \item $(p,q,r)=(5,5,3)$
  \[
  (w(5,5,3),\,(s_{3}s_{2}^{s_{1}})^{2}=1)
  \]
   On a $t^{3}=-id_{M}$ et $\Delta=2$. On obtient la représentation $R(\tau,3-\tau,1;\tau-3)$ de $W(5,5,3)$.
  \item $(p,q,r)=(5,5,5)$
  \[
  (w(5,5,5),\,(s_{1}s_{2}^{s_{3}})^{3}=1,\,(s_{1}s_{2}^{s_{3}s_{2}})^{2}=1)
  \]
  On a $t^{5}=-id_{M}$ et $\Delta=2(3-\tau)$. On obtient la représentation $R(\tau,\tau,\tau;1-\tau)$ de $W(5,5,5)$.
\end{enumerate}
\end{theorem}
\begin{proof}
Elle est divisée en deux parties. D'abord nous montrons que les présentations données sont bien des présentations du groupe de Coxeter $W(H_{3})$, ensuite nous montrons que ce sont les seules possibilités pour obtenir $W(H_{3})$ comme quotient d'un groupe de Coxeter de rang $3$. Le calcul de $t^{n}$ ne présente pas de difficultés.\\
\textbf{Première partie.}\\
1) C'est la définition de $W(H_{3})$. On a la représentation $R(1,\tau,0;0)$ de $W(3,5,2)$.\\
2) On a $<s_{3},s_{3}^{s_{1}}>=<s_{1},s_{3}>$, donc $G=<s_{3}^{s_{1}},s_{2},s_{3}>$ avec le diagramme:
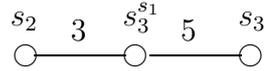
\[
\begin{picture}(150,88)
\put(29,42){\circle{7}}
\put(32,42){\line(1,0){30}}
\put(65,42){\circle{7}}
\put(70,42){\line(4,0){30}}
\put(103,42){\circle{7}}
\put(24,52){$s_{2}$}
\put(61,52){$s_{3}^{s_{1}}$}
\put(99,52){$s_{3}$}
\put(44,47){$3$}
\put(80,47){$5$}
\end{picture}
\]
et $G$ est isomorphe à $W(H_{3})$. En effet nous avons $\{\alpha,\beta\}\subset\{\tau,3-\tau\}$, $l=m=\gamma=0$, puis $C(s_{2},s_{3}^{s_{1}})=\gamma+\alpha\beta+(\alpha l+\beta m)=\alpha\beta=1$, d'où par exemple $\alpha=\tau$ et $\beta=3-\tau$. Nous obtenons la représentation $R(\tau,3-\tau,0;0)$ de $W(5,5,2)$.\\
3) On a $<s_{2},s_{3}>=<s_{2},s_{2}^{s_{3}s_{2}}>$, donc $G=<s_{1},s_{2},s_{2}^{s_{3}s_{2}}>$ avec le diagramme:
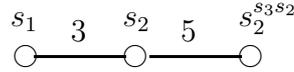
\[
\begin{picture}(150,88)
\put(29,42){\circle{7}}
\put(32,42){\line(1,0){30}}
\put(65,42){\circle{7}}
\put(70,42){\line(4,0){30}}
\put(103,42){\circle{7}}
\put(24,52){$s_{1}$}
\put(61,52){$s_{2}$}
\put(99,52){$s_{2}^{s_{3}s_{2}}$}
\put(44,47){$3$}
\put(80,47){$5$}
\end{picture}
\]
et $G$ est isomorphe à $W(H_{3})$. En effet nous avons $\alpha=\beta=1$, $\gamma=\tau$ et $C(s_{1},s_{2}^{s_{3}s_{2}})=(l+\gamma -1)(m+\gamma -1)$ d'où $l=m=1-\gamma=1-\tau$. Nous obtenons la représentation $R(1,1,\tau;1-\tau)$ de $W(3,3,5)$.\\
4) On a $<s_{2},s_{3}>=<s_{2},s_{2}^{s_{3}}>$ donc $G=<s_{1},s_{2},s_{2}^{s_{3}}>$ avec le diagramme:
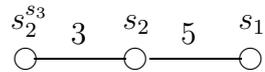
\[
\begin{picture}(150,88)
\put(29,42){\circle{7}}
\put(32,42){\line(1,0){30}}
\put(65,42){\circle{7}}
\put(70,42){\line(4,0){30}}
\put(103,42){\circle{7}}
\put(24,52){$s_{2}^{s_{3}}$}
\put(61,52){$s_{2}$}
\put(99,52){$s_{1}$}
\put(44,47){$3$}
\put(80,47){$5$}
\end{picture}
\]
et $G$ est isomorphe à $W(H_{3})$. En effet nous avons $C(s_{1},s_{2}^{s_{3}})=(\beta m+\alpha)(l+1)$ d'où $\beta m+\alpha=0$ et $l+1=0$ donc $l=-1$; comme $lm=\gamma=1$, $m=-1$ et $\beta=\alpha$. Nous obtenons la représentation $R(\tau,\tau,1;-1)$ de $W(5,5,3)$.\\
5) On a $<s_{2},s_{2}^{s_{1}}>=<s_{2},s_{1}>$ donc $G=<s_{3},s_{2},s_{2}^{s_{1}}>$ avec le diagramme:
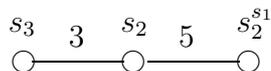
\[
\begin{picture}(150,88)
\put(29,42){\circle{7}}
\put(32,42){\line(1,0){30}}
\put(65,42){\circle{7}}
\put(70,42){\line(4,0){30}}
\put(103,42){\circle{7}}
\put(24,52){$s_{3}$}
\put(61,52){$s_{2}$}
\put(99,52){$s_{2}^{s_{1}}$}
\put(44,47){$3$}
\put(80,47){$5$}
\end{picture}
\]
et $G$ est isomorphe à $W(H_{3})$. En effet nous avons $C(s_{3},s_{2}^{s_{1}})=(\alpha+m)(\beta+l)$ d'où $\alpha+m=\beta+l=0$ donc $m=-\alpha$, $l=-\beta$ et $lm=\gamma=1=\alpha\beta$. Nous obtenons la représentation $R(\tau,3-\tau,1;\tau-3)$ de $W(5,5,3)$.\\
6) On a $<s_{2},s_{2}^{s_{3}s_{2}}>=<s_{2},s_{3}>$, donc $G=<s_{1},s_{2},s_{2}^{s_{3}s_{2}}>$ avec le diagramme:
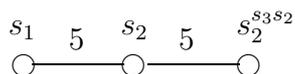
\[
\begin{picture}(150,88)
\put(29,42){\circle{7}}
\put(32,42){\line(1,0){30}}
\put(65,42){\circle{7}}
\put(70,42){\line(4,0){30}}
\put(103,42){\circle{7}}
\put(24,52){$s_{1}$}
\put(61,52){$s_{2}$}
\put(99,52){$s_{2}^{s_{3}s_{2}}$}
\put(44,47){$5$}
\put(80,47){$5$}
\end{picture}
\]
Posons $t_{2}:=s_{1}$, $t_{1}=s_{2}$, $t_{3}=s_{2}^{s_{3}s_{2}}$, alors $t_{2}t_{3}^{t_{1}}=s_{1}s_{2}^{s_{3}}$ qui est d'ordre 3 par hypothèse, donc d'après le 2), $G\simeq W(H_{3})$. En effet nous avons $C(s_{1},s_{2}^{s_{3}s_{2}})=(\alpha l+\beta(\gamma-1))(m+\gamma-1)$ d'où $\alpha l+\beta(\gamma-1)=0=m+(\gamma-1)$, donc $m=-(\gamma-1)$ puis, comme $lm=\gamma$, $l=-(\gamma-1)$ et la première relation donne $\beta=\alpha$. Nous avons
 $C(s_{1},s_{2}^{s_{3}})=1=\beta\gamma+\alpha+\alpha l+\beta m=\alpha(3-\gamma)$. Il en résulte que $\gamma=\alpha$. Nous obtenons la représentation $R(\tau,\tau,\tau;1-\tau)$ de $W(5,5,5)$.\\
 \textbf{Deuxième partie.}\\
 Nous pouvons remarquer que le sous-groupe $G^{+}\,(=D(G))$ où $G\simeq W(H_{3})$ est isomorphe au groupe alterné de degré $5$ donc il ne contient pas d'éléments d'ordre $4$, $6$ ou $10$. Si l'on veut obtenir $W(H_{3})$ comme quotient d'un groupe de Coxeter $W(p,q,r)$, on doit avoir $\{p,q,r\}\subset \{2,3,5\}$ avec $5\in \{p,q,r\}$ et aussi $2$ ne peut apparaitre qu'une fois au plus parmi $\{p,q,r\}$. Les triples à examiner sont donc l'un des suivants (à l'ordre près de $p$, $q$, $r$): $(3,5,2),(5,5,2),(3,3,5),(5,5,3),(5,5,5)$. on doit aussi avoir $\alpha l=\beta m$ puisqu'il existe une forme bilinéaire $G$-invariante non nulle. Dans toute la suite $\epsilon \in \{-1,+1\}$.\\
 Nous examinons successivement les différents cas.\\
 1) Pour $(p,q,r)=(3,5,2)$ il n'y a rien à dire.\\
 2) Pour $(p,q,r)=(5,5,2)$ on a $\{\alpha,\beta\}\subset \{\tau,3-\tau\}$ et $\gamma=l=m=0$. Nous avons $C(s_{2},s_{3}^{s_{1}})=\gamma+\alpha\beta+\alpha l+\beta m=\alpha\beta$ donc si le groupe $G$ est fini, nécessairement $\alpha\beta=1$, donc $\beta\neq \alpha$ et $s_{2}s_{3}^{s_{1}}$ est d'ordre $3$.\\
 3) Pour $(p,q,r)=(3,3,5)$, on a $\alpha=\beta=1$ et $\gamma=\tau$. Nous avons $\alpha l= \beta m=l=m$, $lm=\tau$ donc $l=m=\epsilon(\tau-1)$. Nous obtenons 
 \[
 C(s_{2},s_{3}^{s_{1}})=\tau +1+2\epsilon(\tau-1)=
 \begin{cases}
 3\tau-1, &\text{si $\epsilon=+1$}\\
 3-\tau,  &\text{si $\epsilon=-1$.}
 \end{cases}
 \]
 Il en résulte que $\epsilon=-1$ car sinon $s_{2}s_{3}^{s_{1}}$ est d'ordre infini.\\
 4)Pour $(p,q,r)=(5,5,3)$, on a $\{\alpha,\beta\}\subset \{\tau,3-\tau\}$, $\gamma=1$ et aussi $lm=1$ et $\alpha l=\beta m$. Distinguons deux cas suivant que $\alpha=\beta$ ou non.\\
 \textbf{Premier cas.} Si $\alpha=\beta$, alors $l=m=\epsilon$ et nous obtenons:
 \[
 C(s_{1},s_{2}^{s_{3}})=\alpha+\beta\gamma+\alpha l+\beta m=(\alpha+\beta m)(l+1)=
 \begin{cases}
4\alpha &\text{si $\epsilon=+1,$}\\
0 &\text{si $\epsilon=-1.$}
\end{cases}
 \]
 Si $\epsilon=+1$, $s_{1}s_{2}^{s_{3}}$ est d'ordre infini;\\
 si $\epsilon=-1$, $\alpha+\beta m=l+1=0$ donc $s_{1}s_{2}^{s_{3}}$ est d'ordre $2$.\\
 \textbf{Deuxième cas.} Si $\beta \neq \alpha$, alors $\alpha=\tau$, $\beta=3-\tau$, $\gamma=1$, $\alpha l\beta m=1$ et $\alpha l=\beta m=\epsilon$. Nous obtenons:
 \[
 C(s_{1},s_{2}^{s_{3}})=\alpha+\beta\gamma+\alpha l+\beta m=3+2\epsilon=
 \begin{cases}
5 &\text{si $\epsilon=+1$,}\\
1&\text{si $\epsilon=-1$.}
\end{cases}
 \]
 Si $\epsilon=+1$, $s_{1}s_{2}^{s_{3}}$ est d'ordre infini;\\
 si $\epsilon=-1$, $s_{1}s_{2}^{s_{3}}$ est d'ordre $3$.\\
 5) Pour $(p,q,r)=(5,5,5)$, on a $\{\alpha,\beta,\gamma\}\subset\{\tau,3-\tau\}$ et nous pouvons supposer que $\alpha=\beta=\tau$ et $\gamma \in \{\tau,3-\tau\}$. Comme $\alpha l=\beta m$, nous avons $l=m$ et $l^{2}=m^{2}=\gamma$.
 On a $(\tau-1)^{2}=\tau$ et $(\tau-2)^{2}=3-\tau$ et nous obtenons:
 \[
 l=m=
 \begin{cases}
\epsilon(\tau-1) &\text{si $\gamma=\tau$,}\\
\epsilon(\tau-2) &\text{si $\gamma=3-\tau$}.
\end{cases}
 \]
 Nous avons $C(s_{2},s_{3}^{s_{1}})=\alpha\beta +\gamma +\alpha l+\beta m$.
 Distinguons deux cas suivant que $\gamma=\tau$ ou $\gamma=3-\tau$.\\
 \textbf{Premier cas.} Si $\gamma=\tau$, alors $l=m=\epsilon(\tau-1)$ et $C(s_{2},s_{3}^{s_{1}})=\tau(\tau+1+2\epsilon\tau-2\epsilon)$, si $\epsilon=+1$, $C(s_{2},s_{3}^{s_{1}})=\tau(3\tau-1)=\tau^{3}$ et $s_{2}s_{3}^{s_{1}}$ est d'ordre infini; si $\epsilon=-1$, $C(s_{2},s_{3}^{s_{1}})=\tau(3-\tau)=1$ et $s_{2}s_{3}^{s_{1}}$ est d'ordre $3$.\\
 \textbf{Deuxième cas.} Si $\gamma=3-\tau$, alors $l=m=\epsilon(\tau-2)$ et $C(s_{2},s_{3}^{s_{1}})=\tau^{2}+3-\tau+2\epsilon(\tau^{2}-\tau)$; si $\epsilon=+1$, $C(s_{2},s_{3}^{s_{1}})=3(\tau^{2}-\tau+1)=6\tau=3(3+\sqrt{5})>4$ et $s_{2}s_{3}^{s_{1}}$ est d'ordre infini; si $\epsilon=-1$, $C(s_{2},s_{3}^{s_{1}})=-\tau^{2}+\tau+3=-2\tau+4=1-\sqrt{5}<0$ et $s_{2}s_{3}^{s_{1}}$ est d'ordre infini.
\end{proof}
\subsection{Présentations du groupe de Coxeter $W(H_{4})$.}
Nous donnons maintenant quelques résultats pour le groupe de Coxeter $W(H_{4})$ sans prétendre être exhaustif.
\begin{proposition}
Des présentations de $W(H_{4})$ comme quotient d'un groupe de Coxeter de rang $4$ sont les suivantes:\\
1) 
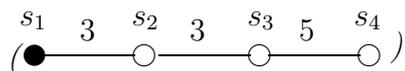
\[
\begin{picture}(150,70)
\put(20,38){(}
\put(29,42){\circle*{7}}
\put(32,42){\line(1,0){30}}
\put(65,42){\circle{7}}
\put(70,42){\line(4,0){30}}
\put(103,42){\circle{7}}
\put(106,42){\line(7,0){30}}
\put(139,42){\circle{7}}
\put(24,52){$s_{1}$}
\put(61,52){$s_{2}$}
\put(99,52){$s_{3}$}
\put(134,52){$s_{4}$}
\put(44,47){$3$}
\put(80,47){$3$}
\put(116,47){$5$}
\put(146,42){)}
\end{picture}
\]
C'est la présentation de $W(H_{4})$ comme groupe de Coxeter.\\
2) 
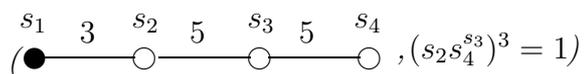
\[
\begin{picture}(150,70)
\put(20,38){(}
\put(29,42){\circle*{7}}
\put(32,42){\line(1,0){30}}
\put(65,42){\circle{7}}
\put(70,42){\line(4,0){30}}
\put(103,42){\circle{7}}
\put(106,42){\line(7,0){30}}
\put(139,42){\circle{7}}
\put(24,52){$s_{1}$}
\put(61,52){$s_{2}$}
\put(99,52){$s_{3}$}
\put(134,52){$s_{4}$}
\put(44,47){$3$}
\put(80,47){$5$}
\put(116,47){$5$}
\put(148,42){,}
\put(152,42){$(s_{2}s_{4}^{s_{3}})^{3}=1$}
\put(204,42){)}
\end{picture}
\]
On obtient la représentation $R(1,\tau,3-\tau;0)$ de $W(H_{4})$.\\
3) 
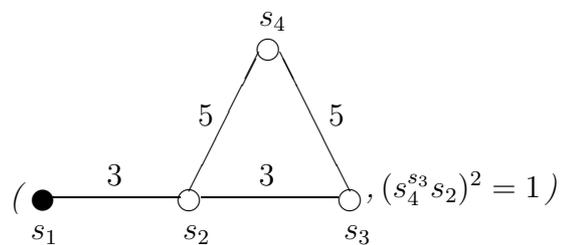
\[
\begin{picture}(150,78)
\put(28,3){(}
\put(110,63){$s_{4}$}
\put(113,55){\circle{7}}
\put (87,5){\circle{7}}
\put (140,5){\circle{7}}
\put(39,5){\circle*{7}}
\put(85,-8){$s_{2}$}
\put(138,-8){$s_{3}$}
\put(91,6){\line(1,0){45}}
\put(87,8){\line(1,2){23}}
\put(42,6){\line(1,0){42}}
\put(140,8){\line(-1,2){23}}
\put(35,-8){$s_{1}$}
\put(60,10){$3$}
\put(110,10){$3$}
\put(90,30){$5$}
\put(133,30){$5$}
\put(145,6){,}
\put(150,6){$(s_{4}^{s_{3}}s_{2})^{2}=1$}
\put(204,6){)}
\end{picture}
\]
On obtient la représentation $R(1,1,\tau,\tau;-\tau)$ de $W(H_{4})$.\\
4)
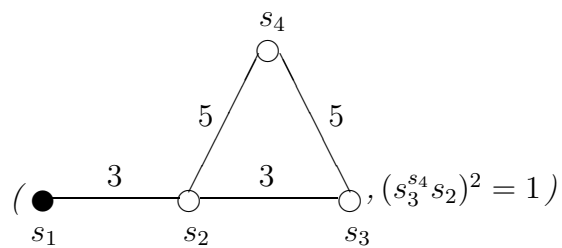
\[
\begin{picture}(150,78)
\put(28,3){(}
\put(110,63){$s_{4}$}
\put(113,55){\circle{7}}
\put (87,5){\circle{7}}
\put (140,5){\circle{7}}
\put(39,5){\circle*{7}}
\put(85,-8){$s_{2}$}
\put(138,-8){$s_{3}$}
\put(91,6){\line(1,0){45}}
\put(87,8){\line(1,2){23}}
\put(42,6){\line(1,0){42}}
\put(140,8){\line(-1,2){23}}
\put(35,-8){$s_{1}$}
\put(60,10){$3$}
\put(110,10){$3$}
\put(90,30){$5$}
\put(133,30){$5$}
\put(145,6){,}
\put(150,6){$(s_{3}^{s_{4}}s_{2})^{2}=1$}
\put(204,6){)}
\end{picture}
\]
On obtient la représentation $R(1,1,\tau,3-\tau;-1)$ de $W(H_{4})$.\\
5)
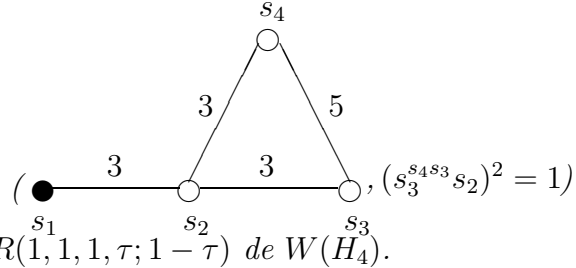
\[
\begin{picture}(150,78)
\put(28,3){(}
\put(110,63){$s_{4}$}
\put(113,55){\circle{7}}
\put (87,5){\circle{7}}
\put (140,5){\circle{7}}
\put(39,5){\circle*{7}}
\put(85,-8){$s_{2}$}
\put(138,-8){$s_{3}$}
\put(91,6){\line(1,0){45}}
\put(87,8){\line(1,2){23}}
\put(42,6){\line(1,0){42}}
\put(140,8){\line(-1,2){23}}
\put(35,-8){$s_{1}$}
\put(60,10){$3$}
\put(110,10){$3$}
\put(90,30){$3$}
\put(133,30){$5$}
\put(145,6){,}
\put(150,6){$(s_{3}^{s_{4}s_{3}}s_{2})^{2}=1$}
\put(209,6){)}
\end{picture}
\]
On obtient la représentation $R(1,1,1,\tau;1-\tau)$ de $W(H_{4})$.
\end{proposition}
\begin{proof}
Dans tous les cas, la racine de l'arbre couvrant est $s_{1}$. Dans les cas 3), 4) et 5), on enlève l'arête $(s_{3},s_{4})$ pour obtenir l'arbre couvrant.

Nous n'avons rien à dire sur le 1).\\
2) On a $<s_{3},s_{4}>=<s_{4}^{s_{3}},s_{4}>$ donc $G=<s_{1},s_{2},s_{4}^{s_{3}},s_{4}>$ avec le diagramme:
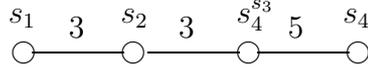
\[
\begin{picture}(150,70)
\put(29,42){\circle{7}}
\put(32,42){\line(1,0){30}}
\put(65,42){\circle{7}}
\put(70,42){\line(4,0){30}}
\put(103,42){\circle{7}}
\put(106,42){\line(7,0){30}}
\put(139,42){\circle{7}}
\put(24,52){$s_{1}$}
\put(61,52){$s_{2}$}
\put(99,52){$s_{4}^{s_{3}}$}
\put(134,52){$s_{4}$}
\put(44,47){$3$}
\put(80,47){$3$}
\put(116,47){$5$}
\end{picture}
\]
d'où le résultat: $G\simeq W(H_{4})$.\\
Soit $\tau' \in \{\tau,3-\tau\}$. Un calcul simple montre que $s_{4}^{s_{3}}(a_{2})=a_{2}+(\tau'a_{3}+a_{4})$ et $s_{2}(\tau'a_{3}+a_{4})=\tau\tau'a_{2}+(\tau'a_{3}+a_{4})$, d'où $C(s_{2},s_{4}^{s-3})=\tau'\tau$. Ensuite $(s_{2}s_{4}^{s_{3}})^{3}=1\leftrightarrow C(s_{2},s_{4}^{s-3})=1=\tau'\tau$ donc $\tau'=3-\tau$. On obtient la représentation $R(1,\tau,3-\tau;0)$.\\
3) On a $<s_{3},s_{4}>=<s_{4}^{s_{3}},s_{3}>$ donc $G=<s_{1},s_{2},s_{3},s_{4}^{s_{3}}>$ et nous terminons la démonstration comme dans le cas précédent.\\
4) On a $<s_{3},s_{4}>=<s_{3}^{s_{4}},s_{3}>$ donc $G=<s_{1},s_{2},s_{3},s_{3}^{s_{4}}>$ et nous terminons la démonstration comme dans le cas précédent.\\
5) On a $<s_{3},s_{4}>=<s_{3}^{s_{4}s_{3}},s_{3}>$ donc $G=<s_{1},s_{2},s_{3},s_{3}^{s_{4}s_{3}}>$ et nous terminons la démonstration comme dans le cas précédent.
\end{proof}
\section{Exemples lorsque le corps $K$ est complexe.}
\subsection{Les groupes de réflexion complexes $G(p,p,n)$ comme quotients du groupe de Weyl affine $W(\tilde{A}_{n-1})$.}
Nous montrons maintenant comment la théorie précédente permet de trouver la famille des groupes de réflexion complexe $G(p,p,n)$ comme quotients du groupe de Weyl affine $W(\tilde{A}_{n-1})$.

1) On pose $G:=W(\tilde{A}_{n-1})$, $S=\{s_{1}, \cdots,s_{n}\}$ et le graphe $\Gamma=\Gamma(G)$ est un circuit sans corde. On prend $s_{1}$ comme racine et comme arbre couvrant $\{s_{1}, \cdots,s_{n}\}$. On cherche toutes les représentations de réflexion de $G$ satisfaisant aux conditions (R). On a $m_{s_{i}s_{i+1}}=3$ pour $(1\leqslant i \leqslant n, \text{les indices étant pris}\mod n)$. Si $e$ est une arête de $\Gamma$ alors  $\alpha_{e}$ est égal à $1$. On a $K_{0}=\mathbb{Q}$. On a $(s_{1},s_{n})\in E'(\Gamma)$. Soient $l$ et $m$ dans une extension $K$ de $\mathbb{Q}$ avec $lm=1$ et on suppose que $K=\mathbb{Q}(l)$. On pose $l_{1,n}=l$ et $l_{n,1}=m$.

Dans la base $\mathcal{A}$ de $M$, on a les relations: toutes celles obtenues grâce à la construction fondamentale et $s_{1}(a_{n})=la_{1}+a_{n}$, et $s_{n}(a_{1})=a_{1}+ma_{n}$
On pose $s_{0}:=s_{1}^{s_{2}s_{3}\cdots s_{n-1}}$ et $a_{0}:= \sum_{i=1}^{n-1}a_{i}$ ($a_{0}$ est l'opposée de la plus grande racine du système de racines associé au groupe de Weyl $W(A_{n-1})$ engendré par les $s_{i} (1\leqslant i \leqslant n-1)$). On a alors $s_{0}(a_{0})=-a_{0}$, $s_{0}(a_{1})=a_{1}+a_{0}$, $s_{0}(a_{i})=a_{i}\,(2\leqslant i\leqslant n-2)$, $s_{0}(a_{n-1})=a_{n-1}+a_{0}$, $s_{0}(a_{n})=a_{n}-(l+1)a_{0}$. De plus $s_{n}(a_{0})=a_{0}-(m+1)a_{n}$, donc $C(s_{0},s_{n})=(l+1)(m+1)=l+m+2$.\\ Dans la représentation géométrique de $W(\tilde{A}_{n-1})$, on a $l=m=1$ donc $C(s_{0},s_{n})=4$, $s_{0}s_{n}$ est une application unipotente et $s_{0}s_{n}$ et ses conjugués engendrent le sous-groupe des translations de $W(\tilde{A}_{n-1})$.\\
Grâce à $s_{0}$ nous obtenons une autre présentation du groupe $G$. On a:
\[
<s_{1},s_{2},\cdots, s_{n-1}>=<s_{2},\cdots, s_{n-1},s_{0}>
\]
car $s_{0}\in <s_{1},s_{2},\cdots, s_{n-1}>$ et $s_{1}\in <s_{2},\cdots, s_{n-1},s_{0}>$ car $s_{1}=s_{0}^{s_{n-1}s_{n-2}\cdots s_{2}}$. Il en résulte que $G=<s_{2},\cdots, s_{n-1},s_{o},s_{n}>$ avec le diagramme\\ $\Gamma(s_{2},s_{3},\cdots,s_{n-2},s_{n-1},s_{n},s_{0})$:
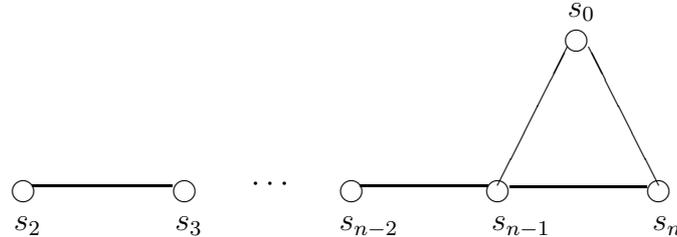
\[
\begin{picture}(150,78)
\put(150,63){$s_{0}$}
\put(153,55){\circle{7}}
\put (127,5){\circle{7}}
\put (180,5){\circle{7}}
\put(24,5){\circle{7}}
\put(-29,5){\circle{7}}
\put(79,5){\circle{7}}
\put(46,5){$\cdots$}
\put(125,-8){$s_{n-1}$}
\put(178,-8){$s_{n}$}
\put(131,6,5){\line(1,0){45}}
\put(127,6,8){\line(1,2){23}}
\put(180,6,8){\line(-1,2){23}}
\put(82,6,8){\line(1,0){42}}
\put(-26,6,8){\line(1,0){46}}
\put(75,-8){$s_{n-2}$}
\put(-32,-8){$s_{2}$}
\put(21,-8){$s_{3}$}
\end{picture}
\]
On a la relation supplémentaire $(s_{0}^{s_{n-1}}s_{n})^{3}=1$. En effet $(s_{0}^{s_{n-1}}s_{n})^{3}=(s_{1}^{s_{2}\cdots s_{n-2}}s_{n})^{3}=(s_{1}s_{n})^{3}=1$. Il résulte de ceci qu'il y a une \textbf{inclusion canonique de $W(\tilde{A}_{i})$ dans $W(\tilde{A}_{n-1})$ pour $3\leqslant i \leqslant n-1$.}\\
Le diagramme précédent avec la relation supplémentaire $(s_{0}^{s_{n-1}}s_{n})^{3}=1$ donne une autre présentation du groupe $G$.\\
Soit $\Gamma=\Gamma(t_{1},t_{2},\cdots,\,t_{n-3},t_{n-2},t_{n},t_{n-1})$
\[
H:=<t_{1},\cdots, t_{n}| \Gamma ,(t_{n-1}^{t_{n-2}}t_{n})^{3}=1>
\]
Nous montrons que $H\simeq G$. Il est clair que $G$ est isomorphe à un quotient de $H$, donc $H$ est infini.

Posons $t_{0}:=t_{n-1}^{t_{n-2}\cdots,\,t_{2}t_{1}}$. Alors $H=<t_{0},t_{1},\cdots,\,t_{n-2},t_{n}>$, de plus 
\[
(t_{0}t_{n})^{3}=(t_{n-1}^{t^{n-2}\cdots\,t_{2}t_{1}}t_{n})^{3}=(t_{n-1}^{t_{n-2}}t_{n})^{3}
\]
et l'on a le diagramme:
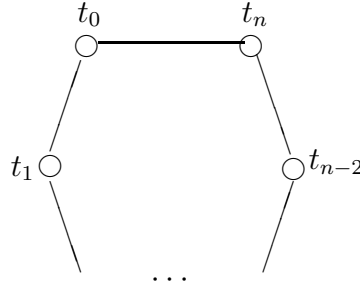
\[
\begin{picture}(150,68)
\put(56,63){$t_{0}$}
\put(59,55){\circle{7}}
\put(63,56,0){\line(1,0){48}}
\put(110,63){$t_{n}$}
\put(113,55){\circle{7}}
\put(47,20,0){\line(1,3){10}}
\put(47,15){\circle{7}}
\put(34,12){$t_{1}$}
\put(115,52,0){\line(1,-3){11}}
\put(127,14){\circle{7}}
\put(132,14){$t_{n-2}$}
\put(47,10,0){\line(1,-3){10}}
\put(117,-20,0){\line(1,3){10}}
\put(80,-25){$\cdots$}
\end{picture}
\]
\[
\]
donc $H$ est isomorphe à un quotient du groupe $G$. Comme tous les quotients propres de $G$ sont des groupes finis, d'après la remarque précédente, on voit que $H\simeq G$.

2) Soit maintenant un quotient propre de $G$. On sait que ce quotient est fini, donc l'ordre de $s_{0}s_{n}$ est fini, $p$ par exemple. On appelle $G_{p}$ un tel quotient. On sait aussi que cela est possible si et seulement si $C(s_{0},s_{n})$ est racine du polynôme $v_{p}(X)$: il existe $k\in \mathbb{N}$, $k$ premier à $p$ tel que $l+m+2=4\cos^{2} \frac{k\pi}{p}$. Il en résulte que $l$ et $m$ sont les racines du polynôme:
\[
Q(X)=X^{2}-2\cos \frac{2k\pi}{p}X+1
\]
et l'on a $\{ l,m\}=\{\zeta_{p},\zeta_{p}^{-1}\}$ où $\zeta_{p}=\exp (\frac{2ik\pi}{p})$.\\
Dans ces conditions, on obtient un groupe de réflexion complexe isomorphe à $G(p,p,n)$ et  $K=\mathbb{Q}(\zeta_{p})$.\\
Si $p=2$, on obtient le groupe $W(D_{n})$ (on a $l=m=-1$).\\
On a ainsi les deux présentations suivantes de $G(p,p,n)$:
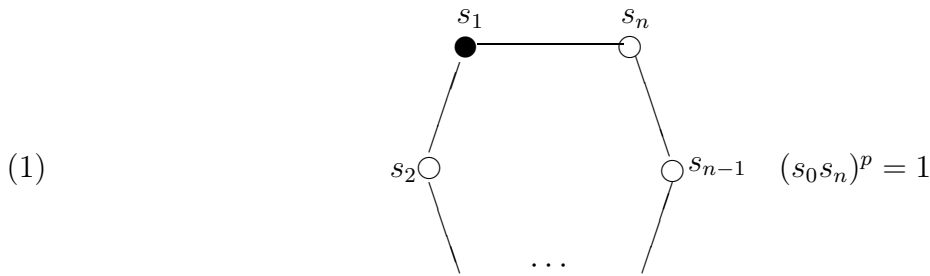
\[
\begin{picture}(150,68)
\put(56,63){$s_{1}$}
\put(59,55){\circle*{7}}
\put(63,56,0){\line(1,0){48}}
\put(110,63){$s_{n}$}
\put(113,55){\circle{7}}
\put(47,20,0){\line(1,3){10}}
\put(47,15){\circle{7}}
\put(34,12){$s_{2}$}
\put(115,52,0){\line(1,-3){11}}
\put(127,14){\circle{7}}
\put(132,14){$s_{n-1}$}
\put(47,10,0){\line(1,-3){10}}
\put(117,-20,0){\line(1,3){10}}
\put(80,-20){$\cdots$}
\put(-92,12){(1)}
\put(162,12){$(s_{0}s_{n})^{p}=1$}
\end{picture}
\]
\[
\]
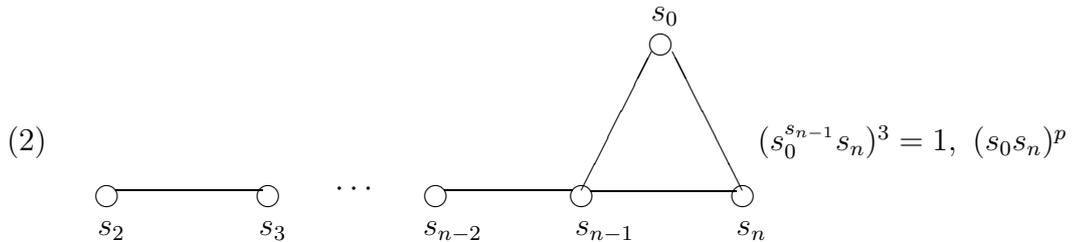
\[
\begin{picture}(130,78)
\put(110,63){$s_{0}$}
\put(113,55){\circle{7}}
\put (87,5){\circle{7}}
\put (140,5){\circle{7}}
\put(-16,5){\circle{7}}
\put(-69,5){\circle{7}}
\put(39,5){\circle{7}}
\put(6,5){$\cdots$}
\put(85,-8){$s_{n-1}$}
\put(138,-8){$s_{n}$}
\put(90,6,5){\line(1,0){48}}
\put(87,6,8){\line(1,2){23}}
\put(140,6,8){\line(-1,2){23}}
\put(42,6,8){\line(1,0){42}}
\put(-66,6,8){\line(1,0){48}}
\put(35,-8){$s_{n-2}$}
\put(-72,-8){$s_{2}$}
\put(-19,-8){$s_{3}$}
\put(-102,20){(2)}
\put(145,20){$(s_{0}^{s_{n-1}}s_{n})^{3}=1$,\, $(s_{0}s_{n})^{p}$}
\end{picture}
\]
\[
\]
La présentation (2) de $G(p,p,n)$ montre que \textbf{si $2 \leqslant n'\leqslant n$ alors $G(p,p,n')\subset G(p,p,n)$.}

Si $C(s_{0},s_{n})$ n'est racine d'aucun polynôme $v_{p}(X)$, alors $s_{0}s_{n}$ est d'ordre infini et l'on obtient une représentation fidèle de réflexion de $G$ qui n'est pas équivalente à la représentation géométrique.
\subsection{Compléments sur les groupes $W(\tilde{A}_{2})$.}
En préparation de ce qui suit, nous complétons les résultats précédents sur le groupe $W(\tilde{A}_{2})$.
\begin{proposition}
On a les présentations suivantes pour le groupe de réflexion complexe $G(n,n,3)$:
\begin{enumerate}
  \item $(w(3,3,3),\quad (s_{1}s_{2}^{s_{3}})^{n}=1)$
  \item $(w(3,3,n), \quad (s_{1}s_{2}s_{3})^{2}=(s_{2}s_{3}s_{1})^{2})$
 \end{enumerate}
\end{proposition}
\begin{proof}
Le 1) est clair d'après l'étude de $W(\tilde{A}_{n-1})$ faite dans la section précédente.\\
Pour le 2), appelons, pour cette démonstration uniquement, $G_{n}$ le groupe $G(n,n,3)$ et $H_{n}$ un groupe ayant la présentation de l'énoncé:
\[
G_{n}=<t_{1},t_{2},t_{3}|w(3,3,3),(t_{1}t_{3}^{t_{2}})^{n}=1>
\]
Comme $<t_{1},t_{2}>=<t_{1},t_{2}^{t_{1}}>$, nous avons $G_{n}=<t_{1},t_{2}^{t_{1}},t_{3}>$. Posons $t'_{2}:=t_{2}^{t_{1}}$. Alors $t'_{2}t_{3}=t_{2}^{t_{1}}t_{3}=t_{1}^{t_{2}}t_{3}=(t_{1}t_{3}^{t_{2}})^{t_{2}}$, donc $t'_{2}t_{3}$ est d'ordre $n$.\\
Ensuite $(t_{1}t'_{2}t_{3})^{2}=(t_{2}t_{1}t_{3})^{2}$ et $(t'_{2}t_{3}t_{1})^{2}=(t_{2}^{t_{1}}t_{3}t_{1})^{2}$ d'où 
\begin{multline}
(t_{2}^{t_{1}}t_{3}t_{1})^{2}=t_{2}^{t_{1}}t_{3}t_{1}t_{2}^{t_{1}}t_{3}t_{1}=t_{1}t_{2}t_{1}t_{3}t_{2}t_{1}t_{3}t_{1}\\
=t_{2}t_{1}(t_{2}t_{3}t_{2}t_{3})t_{1}t_{3}=t_{2}t_{1}(t_{3}t_{2})t_{1}t_{3}=(t_{2}t_{1}t_{3})^{2}\notag
\end{multline}
et nous obtenons $(t_{1}t'_{2}t_{3})^{2}=(t'_{2}t_{3}t_{1})^{2}$. Il en résulte que le groupe $G_{n}$ est isomorphe à un quotient du groupe $H_{n}$.\\
Partant du groupe $H_{n}$ de l'énoncé, posons $s'_{2}:=s_{2}^{s_{1}}$ et $u:=(s_{1}s_{2}s_{3})^{2}$. Nous pouvons remarquer que $u$ commute avec $s_{1}$. Posons $a:=s'_{2}s_{3}s'_{2}$, $b:=s_{3}s'_{2}s_{3}$. Alors $a=s_{3}^{s_{1}s_{2}s_{1}}=s_{1}^{s_{3}s_{2}s_{1}}$ et $b=s_{2}^{s_{1}s_{3}}=s_{1}^{s_{2}s_{3}}$, d'où $a^{s_{1}s_{2}s_{1}}=s_{1}$ et $b^{s_{1}s_{2}s_{3}}=s_{1}^{s_{2}s_{3}s_{1}s_{2}s_{3}}=s_{1}^{u}=s_{1}$ d'après la remarque précédente. Il en résulte que $a=b$ et que $s'_{2}s_{3}$ est d'ordre $3$.\\
Enfin $s_{1}s'_{2}s_{3}s'_{2}=s_{2}s_{1}(s_{3}s_{2})s_{1}s_{2}$, donc $s_{1}s'_{2}s_{3}s'_{2}$ a le même ordre que $s_{3}s_{2}$ c'est à dire $n$. Comme $H_{n}=<s_{1},s'_{2},s_{3}>$, nous voyons que $H_{n}$ est isomorphe à un quotient de $G_{n}$. En combinant les deux morphismes  ainsi construits entre $G_{n}$ et $H_{n}$, nous voyons que ce sont des isomorphismes et nous avons le résultat.
\end{proof}
\begin{proposition}
Soit $R(1,1,\gamma;l)$ une représentation de réflexion de $W(3,3,n)$. Alors les conditions suivantes sont équivalentes:
\begin{enumerate}
  \item $(s_{1}s_{2}s_{3})^{2}=(s_{2}s_{3}s_{1})^{2}$;
  \item $l+m=-\gamma$;
  \item $\Delta=4-\gamma$.
\end{enumerate}
Si ces conditions sont satisfaites, on a $(s_{1}s_{2}s_{3})^{2n}=id_{M}$ et\\ $\{l,m\}=\{-1-\exp (\frac{2ki\pi}{n}),-1-\exp(\frac{-2ki\pi}{n})\}$ pour un certain $k$ premier à $n$.
\end{proposition}
\begin{proof}
On a:
\[
s_{1}=\begin{pmatrix}
-1 & 1 & 1\\
0 & 1 & 0\\
0 & 0 &1
\end{pmatrix},
s_{2}=\begin{pmatrix}
1 & 0 & 0\\
1 & -1 & l\\
0 & 0 & 1
\end{pmatrix},
s_{3}=\begin{pmatrix}
1 & 0 & 0\\
0 & 1 & 0\\
1 & m & -1
\end{pmatrix}
\]
et un calcul simple montre que:
\[
(s_{1}s_{2}s_{3})^{2}=\begin{pmatrix}
l^{2}+l\gamma+2\gamma+m-1 & \gamma^{2}+\gamma(l+m-1)-(l+m) & -l^{2}-l\gamma-\gamma\\
l^{2}+l\gamma+\gamma & \gamma^{2}+l\gamma-\gamma-l+m & -l^{2}-l\gamma-1\\
\gamma+l+m & m\gamma+\gamma-m-1 & -\gamma-l
\end{pmatrix}
\]
et 
\[
(s_{2}s_{3}s_{1})^{2}=\begin{pmatrix}
-l-1 & \gamma+l+m & 0\\
-l^{2}-l\gamma-\gamma & \gamma^{2}+2l\gamma+l^{2}+m-l & -l-1\\
-\gamma-l-m & m\gamma+2\gamma+l-1 & m
\end{pmatrix}
\]
Si $(s_{1}s_{2}s_{3})^{2}=(s_{2}s_{3}s_{1})^{2}$, nous devons avoir $\gamma +l+m=0$ et nous vérifions immédiatement que cette condition est suffisante, donc les conditions 1. et 2. sont équivalentes. Comme $\Delta=4-2\gamma -(l+m)$, nous voyons tout de suite que 2. et 3. sont équivalentes.
\end{proof}
\subsection{Le groupe de réflexion complexe $G_{24}$.}
Dans cette partie nous montrons comment obtenir le groupe de réflexion complexe $G_{24}$ comme quotient d'un groupe de Coxeter de rang $3$. Plus précisément, nous obtenons tous les triples $(p,q,r)$ (à l'ordre près) tels que $G_{24}$ soit un quotient de $W(p,q,r)$. Nous donnons des présentations de $G_{24}$ en utilisant les résultats précédents.

Nous connaissons la structure de $G_{24}$: $G_{24}\simeq C_{2}\times PSL_{2}(7)$. Nous connaissons aussi son corps de définition en tant que groupe de réflexion complexe: si $\zeta:=\frac{1+i\sqrt{7}}{2}$, $\zeta$ racine du polynôme $X^{2}-X+2$ l'autre étant $\zeta':=\frac{1-i\sqrt{7}}{2}$, alors $K_{0}=\mathbb{Q}$ et $K=\mathbb{Q}(\zeta)$ puisque $\zeta$ n'est racine d'aucun polynôme $u_{n}(X)$. De plus l'anneau des entiers de $K$ est $\mathbb{Z}[\zeta]$ qui est un anneau principal..

Le reste de cette partie va être consacrée à la démonstration du théorème suivant:
\begin{theorem}\hfill
\begin{enumerate}
  \item Les triples $(p,q,r)$ tels que $G_{24}$ soit un quotient de $W(p,q,r)$ sont les suivants (à permutation près de $p,q,r$):
  \[
  (p,q,r)\in \{(3,3,4),(4,4,3),(4,4,4)\}.
  \]
  On pose dans les trois cas $t:=s_{1}s_{2}s_{3}$.
  \item Des présentations de $G_{24}$ sont:
  \begin{enumerate}
  \item $(w(3,3,4),(s_{2}^{s_{1}}s_{3})^{4}=1)$.
  On a $t^{7}=-id_{M}$ et $\Delta =1$. On obtient la représentation $R(1,1,2;-\zeta)$ de $W(3,3,4)$.
  \item $(w(4,4,3),(s_{2}s_{1}^{s_{3}})^{3}=1)$. On a $t^{7}=-id_{M}$ et $\Delta =1$. On obtient la représentation $R(2,2,1;-1+\frac{\zeta}{2}(=\frac{\zeta}{\zeta'}))$ de $W(4,4,3)$.
  \item $(w(4,4,4),(s_{1}s_{2}^{s_{3}})^{3}=1)$. On a $t^{7}=-id_{M}$ et $\Delta =1$. On obtient la représentation $R(2,2,2;-1-\frac{\zeta}{2})$ de $W(4,4,4)$.
\end{enumerate}
\end{enumerate}
\end{theorem}
\begin{proof}
Elle divisée en deux parties.\\ D'abord nous montrons que les groupes obtenus dans le 2) sont isomorphes au groupe $G_{24}$, puis nous montrons que ce sont les seules possibilités.

A. 1) Considérons la représentation $R(1,1,2;l)$ de $W(3,3,4)$ et exprimons que $(s_{2}^{s_{1}}s_{3})$ est d'ordre $4$. Nous avons $(s_{2}^{s_{1}}s_{3})$ d'ordre $4$ si et seulement si $(1+l)(1+m)=2$ donc, tenant compte de $lm=2$, nous voyons que $l$ et $m$ sont les racines du polynôme :
\[
Q(X)=X^{2}+X+2.
\]
Nous obtenons $l=-\zeta$ et $m=-1+\zeta$ (l'autre choix correspondant à la représentation conjuguée).
\\Posons $c:=s_{1}s_{2}$ et $d:=s_{1}s_{3}$. Alors le sous-groupe $G^{+}$ de $G$ est engendré par $c$ et $d$ et nous avons 
$c^{3}=d^{3}=1,c^{2}d=s_{2}s_{3}$ donc $(c^{2}d)^{3}=1$ et nous venons de voir que $(cd)^{4}=1$. Il en résulte que $<c,d>$ est isomorphe au groupe $PSL_{2}(7)$ (voir \cite{CM}).\\
D'après la proposition 18 1), nous avons $\theta=l=-\zeta$, $\theta'=-m=\bar{\theta}$, d'où $P_{t}(X)=x^{3}-\bar{\zeta}X^{2}-\zeta X+1$ et un calcul simple montre que $t^{7}=-id_{M}$. Comme $t\not \in G^{+}$, nous avons $t^{7}\not \in G^{+}$ et comme $G^{+}$ est d'indice $2$ dans $G$, nous obtenons $G\simeq C_{2}\times PSL_{2}(7)\simeq G_{24}$.\\
2) Considérons la représentation $R(2,2,1;l)$ de $W(4,4,3)$.Nous avons $s_{2}s_{1}^{s_{3}}$ d'ordre $3$ si et seulement si $\beta\gamma+\alpha+l\alpha+m\beta=1$, si et seulement si $2l+2m=-3$. Comme $2l2m=4$, $2l$ et $2m$ sont les racines du polynôme:
\[
Q(X)=x^{2}+3X+4
\]
et nous obtenons $\{l,m\}=\{-1+\frac{\zeta}{2},-1+\frac{\bar{\zeta}}{2}\}$ ce qui détermine $l$ à conjugaison près.\\
Nous avons $<s_{1},s_{3}>=<s_{3},s_{1}^{s_{3}}>$, donc $G=<s_{2},s_{3},s_{1}^{s_{3}}>$ avec le diagramme:
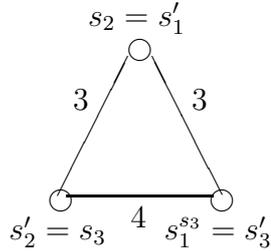
\[
\begin{picture}(150,68)

\put(46,63){$s_{2}=s'_{1}$}
\put(63,55){\circle{7}}
\put (37,5){\circle{7}}
\put (90,5){\circle{7}}
\put(20,-8){$s'_{2}=s_{3}$}
\put(71,-8){$s_{1}^{s_{3}}=s'_{3}$}
\put(39,6,5){\line(1,0){48}}
\put(36,6,7){\line(1,2){23}}
\put(90,6,8){\line(-1,2){23}}
\put(41,35){$3$}
\put(80,35){$3$}
\put(60,-4){$4$}
\end{picture}
\]
Enfin $s'_{1}s'_{2}s'_{1}s'_{3}=s_{3}^{s_{2}}s_{1}^{s_{3}}=s_{2}^{s_{3}}s_{1}^{s_{3}}=(s_{2}s_{1})^{s_{3}}$ donc $s'_{1}s'_{2}s'_{1}s'_{3}$ est d'ordre $4$. Nous appliquons le 1) pour avoir le résultat: $G\simeq G_{24}$.\\
3) Considérons la représentation $R(2,2,2;l)$ de $W(4,4,4)$. Nos avons $s_{1}s_{2}^{s_{3}}$ d'ordre $3$ si et seulement si $\alpha+\beta\gamma +l\alpha+m\beta =1$ si et seulement si $2l+2m=-5$. Comme $2l2m=8$, $2l$ et $2m$ sont les racines du polynôme:
\[
Q(X)=X^{2}+5X+8
\]
et nous obtenons $\{l,m\}=\{-1-\frac{\zeta}{2},-1-\frac{\bar{\zeta}}{2}\}$ ce qui détermine $l$ à conjugaison près.\\
Nous avons $<s_{2},s_{3}>=<s_{3},s_{2}^{s_{3}}>$ donc $G=<s_{1},s_{3},s_{2}^{s_{3}}>$ avec le diagramme:
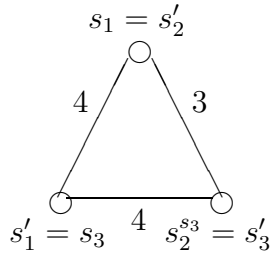
\[
\begin{picture}(150,68)

\put(46,63){$s_{1}=s'_{2}$}
\put(63,55){\circle{7}}
\put (37,5){\circle{7}}
\put (90,5){\circle{7}}
\put(20,-8){$s'_{1}=s_{3}$}
\put(71,-8){$s_{2}^{s_{3}}=s'_{3}$}
\put(39,6,5){\line(1,0){48}}
\put(36,6,7){\line(1,2){23}}
\put(90,6,8){\line(-1,2){23}}
\put(41,35){$4$}
\put(80,35){$3$}
\put(60,-4){$4$}
\end{picture}
\]
Enfin $s'_{2}s'_{3}s'_{1}s'_{3}=s_{1}(s_{2}^{s_{3}})s_{3}(s_{2}^{s_{3}})=s_{1}s_{3}^{s_{2}}$. Un calcul simple montre que $s_{1}s_{3}^{s_{2}}$ est d'ordre $3$, donc $s'_{2}s'_{3}s'_{1}s'_{3}$ est d'ordre $3$ et $G\simeq G_{24}$ d'après le 2).

B. Comme $K_{0}=\mathbb{Q}$, nous avons $\{p,q,r\}\subset \{3,4,6\}$ et $G$ est un quotient d'un $W(p,q,r)$, donc $\{\alpha,\beta,\gamma\}\subset \{1,2,3\}$. Nous savons que l'ordre des éléments réels de $PSL_{2}(7)$ est $3$ ou $4$. Un produit de deux réflexions de $G_{24}$ ne peut donc être d'ordre que $2$, $3$ ou $4$; en particulier $6 \not \in \{p, q, r\}$. Nous avons donc les possibilités (à permutation près) ($2$ est aussi exclu car sinon nous obtenons des groupes de Coxeter: $W(A_{3}),W(B_{3}), W(\tilde{B_{2}}))$:
\[
(p,q,r)\in \{(3,3,4),(4,4,3),(4,4,4)\}.
\]
Examinons ces trois cas successivement.

1) Si $(p,q,r)=(3,3,4)$, d'après la remarque précédente $s_{2}^{s_{1}}s_{3}$ est d'ordre $2$, $3$ ou $4$.\\
- Si $s_{2}^{s_{1}}s_{3}$ est d'ordre $2$, comme $C(s_{3},s_{2}^{s_{1}})=(l+1)(m+1)$, nous avons $l=m=-1$, $\gamma=lm=1\neq 2$ ce cas ne peut pas se produire.\\
- Si $s_{2}^{s_{1}}s_{3}$ est d'ordre $3$, comme $<s_{1},s_{2}^{s_{1}}>=<s_{1},s_{2}>$ nous obtenons\\ $G=<s_{1},s_{2}^{s_{1}},s_{3}>$ avec le diagramme:
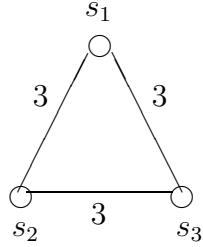
\[
\begin{picture}(150,68)

\put(58,65){$s_{1}$}
\put(63,55){\circle{7}}
\put (37,5){\circle{7}}
\put (90,5){\circle{7}}
\put(34,-8){$s_{2}$}
\put(88,-8){$s_{3}$}
\put(39,6,5){\line(1,0){48}}
\put(36,6,7){\line(1,2){23}}
\put(90,6,8){\line(-1,2){23}}
\put(41,35){$3$}
\put(80,35){$3$}
\put(60,-4){$3$}
\end{picture}
\]
et $G$ est isomorphe à $W(\tilde{A}_{2})$.
Ce cas ne peut pas se produire. Il en résulte que $s_{2}^{s_{1}}s_{3}$ est d'ordre $4$.

2) Si $(p,q,r)=(4,4,3)$, comme ci-dessus $s_{2}^{s_{1}}s_{3}$ est d'ordre $2$, $3$ ou $4$.\\
- Si $s_{2}^{s_{1}}s_{3}$ est d'ordre $2$, $l=m=-1$, donc $G=<s_{1},s_{2},s_{3}^{s_{2}}>$ avec le diagramme:
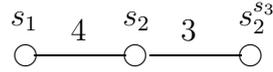
\[
\begin{picture}(150,88)
\put(29,42){\circle{7}}
\put(32,42){\line(1,0){30}}
\put(65,42){\circle{7}}
\put(70,42){\line(4,0){30}}
\put(103,42){\circle{7}}
\put(24,52){$s_{1}$}
\put(61,52){$s_{2}$}
\put(99,52){$s_{2}^{s_{3}}$}
\put(44,47){$4$}
\put(80,47){$3$}
\end{picture}
\]
et $G\simeq W(B_{3})$. Ce cas ne peut pas se produire.\\
- Si $s_{2}^{s_{1}}s_{3}$ est d'ordre $4$, on a $C(s_{2},s_{1}^{s_{3}})=2=4+2(l-m)$, donc $l+m=-1$ et $\Delta=8-4-4-2+2=0$. Ce cas ne peut pas se produire. Il en résulte que $s_{2}^{s_{1}}s_{3}$ est d'ordre $3$.

3) Si $(p,q,r)=(4,4,4)$, comme ci-dessus, $s_{1}s_{2}^{s_{3}}$ est d'ordre $2$, $3$ ou $4$.\\
- Si $s_{1}s_{2}^{s_{3}}$ est d'ordre $2$, on a $l+1=0=2(m+1)$, donc $l=m=-1$ et $\gamma=1\neq 2$. Ce cas ne peut pas se produire.\\
- Si $s_{1}s_{2}^{s_{3}}$ est d'ordre $4$, on a $2=\beta\gamma+\alpha+l\alpha+m\beta=6+2(l+m)$ donc $2(l+m)=-4$ et $\Delta=8-4-4-4+4=0$. Ce cas ne peut pas se produire.\\
Il en résulte que $s_{1}s_{2}^{s_{3}}$ est d'ordre $3$ et ceci achève la démonstration.
\end{proof}
\subsection{Le groupe de réflexion complexe $G_{27}$.}
Dans cette partie nous montrons comment obtenir le groupe de réflexion complexe $G_{27}$ comme quotient d'un groupe de Coxeter de rang $3$. Plus précisément, nous obtenons tous les triples $(p,q,r)$ (à l'ordre près) tels que $G_{27}$ soit un quotient de $W(p,q,r)$. Nous donnons aussi des présentations de $G_{27}$.

Nous connaissons la structure de $G_{27}$: $G_{27}\simeq C_{2}\times (C_{3}*A_{6})$ où $A_{6}$ désigne le groupe alterné de degré $6$ et $C_{3}*A_{6}$ est une extension centrale non scindée de $A_{6}$ par $C_{3}$.

Nous connaissons aussi son corps de définition en tant que groupe de réflexion complexe: si $\omega=\frac{-1+i\sqrt{3}}{2}$ est une racine primitive troisième de l'unité, on a $K=\mathbb{Q}(\tau,\omega)$.

Le reste de cette partie va être consacrée à la démonstration du résultat suivant:
\begin{theorem}\hfill
\begin{enumerate}
  \item Les triples $(p,q,r)$ tels que $G_{27}$ soit isomorphe à un quotient de $W(p,q,r)$ sont les suivants (à permutation près):
  \[
  (p,q,r)\in \{(3,3,5),(3,4,5),(5,5,3),(5,5,4),(4,4,5),(3,3,4)\}.
  \]
  \item Des présentations de $G_{27}$ sont:\begin{enumerate}
  \item $(w(3,3,5),(s_{1}s_{2}^{s_{3}})^{4}=1)$. On a $t^{5}=-\omega id_{M}$, $t^{15}=-id_{M}$ et $\Delta=3-\tau$. On obtient la représentation $R(1,1,\tau;\omega(1-\tau))$ de $W(3,3,5)$ et $K_{0}=\mathbb{Q}(\tau)$. 
  \item $(w(3,4,5),(s_{1}s_{2}^{s_{3}})^{4}=1)$. On a $t^{4}=\omega id_{M}$, $t^{12}=id_{M}$ et $\Delta=3$. On obtient la représentation $R(1,2,\tau;-\omega-\tau)$ de $W(3,4,5)$ et $K_{0}=\mathbb{Q}(\tau)$.
  \item $(w(3,4,5),(s_{1}s_{2}^{s_{3}})^{5}=1)$. On a $t^{5}=-\omega id_{M}$, $t^{15}=-id_{M}$ et $\Delta=3-\tau$. On obtient la représentation $R(1,2,\tau;\omega^{2}+\omega\tau)$ de $W(3,4,5)$ et $K_{0}=\mathbb{Q}(\tau)$.
  \item $(w(5,5,3),(s_{3}s_{2}^{s_{1}})^{3}=1,(s_{1}s_{2}^{s_{3}})^{4}=1)$. On a $t^{5}=-\omega id_{M}$, $t^{15}=-id_{M}$ et $\Delta=3-\tau$. On obtient la représentation $R(\tau,3-\tau,1;\omega(3-\tau))$ de $W(5,5,3)$ et $K_{0}=\mathbb{Q}(\tau)$.
  \item $(w(5,5,4),(s_{1}s_{2}^{s_{3}})^{3}=1, (s_{1}s_{3}^{s_{2}})^{3}=1)$. On a $t^{5}=-\omega id_{M}$, $t^{15}=-id_{M}$ et $\Delta=3-\tau$. On obtient la représentation $R(\tau,\tau,2;\omega(\tau-1)+\omega^{2}))$ de $W(5,5,4)$ et $K_{0}=\mathbb{Q}(\tau)$. 
  \item $(w(4,4,5),(s_{1}s_{2}^{s_{3}})^{3}=1)$. On a $t^{4}=\omega^{2}id_{M}, t^{12}=id_{M}$ et $\Delta=1$. On obtient la représentation $R(2,2,\tau;\frac{1}{2}(\omega-\tau))$ de $W(4,4,5)$ et $K_{0}=\mathbb{Q}(\tau)$.
  \item $(w(3,3,4),(s_{1}s_{2}^{s_{3}})^{5}=1)$. On a $t^{5}=-\omega^{2}id_{M},t^{15}=-id_{M}$ et $\Delta=\tau$. On obtient la représentation $R(1,1,2;\omega(\tau-1)+\omega^{2}))$ de $W(3,3,4)$ et $K_{0}=\mathbb{Q}$.
\end{enumerate}
  
\end{enumerate}

\end{theorem}
\begin{proof}
Elle divisée en deux parties.\\ D'abord nous montrons que les groupes obtenus dans le 2) sont isomorphes au groupe $G_{27}$, puis nous montrons que ce sont les seules possibilités.

A. 1) Considérons la représentation $R(1,1,\tau;l)$ de $W(3,3,5)$ et exprimons que l'élément $s_{1}s_{2}^{s_{3}}$ est d'ordre $4$. Nous avons $s_{1}s_{2}^{s_{3}}$ est d'ordre $4$ si et seulement si $C(s_{1},s_{2}^{s_{3}})=2$ donc $2=1+\tau +l+m$ et comme $lm=\gamma=\tau$, nous voyons que $l$ ,$m$ sont les racines du polynôme:
\[
Q(X)=X^{2}-(1-\tau)X+\tau
\]
dont le discriminant est $-3\tau=-3(\tau-1)^{2}$, d'où $l=\omega(1-\tau)$ et $m=\omega^{2}(1-\tau)$ et $K=\mathbb{Q}(\tau,\omega)$.\\
Nous devons montrer que $G$ possède la bonne structure. Posons $c:=s_{1}s_{2}$ et $d:=s_{2}s_{3}$. Alors $D(G)=G^{+}=<c,d>$ et nous avons $c^{3}=d^{5}=1$, $cd=s_{1}s_{3}$, donc $(cd)^{3}=1$. De plus $c^{2}d=s_{2}s_{1}s_{2}s_{3}=s_{1}^{s_{2}}s_{3}=(s_{1}s_{3}^{s_{2}})^{s_{2}}$. Comme $\alpha=\beta=1$, nous avons $C(s_{1},s_{3}^{s_{2}})=C(s_{1},s_{2}^{s_{3}})=2$, donc $c^{2}d$ est d'ordre $4$.

Nous montrons maintenant que $G^{+}/Z(G^{+})$ est isomorphe au groupe alterné de degré $6$ en utilisant la présentation de ce groupe donnée dans \cite{Bu}:
\[
A_{6}=<c,d| c^{3}=d^{5}=(cd)^{3}=(c^{2}d)^{4}=[c,d^{3}]^{2}=1>.
\]
(Par exemple on prend $c=(2,4,6)$ et $d=(1,2,3,4,5)$.)\\
Il nous reste à calculer l'ordre de $[c,d^{3}]$. Nous avons:
 \[
 [c,d^{3}]=(s_{1}s_{2})(s_{2}s_{3})^{3}(s_{2}s_{1})(s_{3}s_{2})^{3}=(s_{1}(s_{2}s_{3})^{2})^{2}=t_{2}^{2} 
 \]
  où $t_{2}=s_{1}(s_{2}s_{3})^{2}$ (notation de la section 2). Nous utilisons les formules de la proposition 18.
Nous avons:
\begin{align}
\theta &= -4+1+1+\tau+l &= l+\tau-2 &= -1-m \notag\\
\theta' &= 4-1-1-\tau-m &= 2-m-\tau &= l+1 \notag
\end{align}
Nous obtenons $\theta\theta'=-(1+l)(1+m)=-2$ et $\theta+\theta'=l-m$. Il en résulte que $P_{t^{2}}(X)=X^{3}-uX^{2}+u'X+1$ avec $u=1+\theta\tau-(\theta+\theta')$ et $u'=-1+\theta'\tau-(\theta+\theta')$, puis 
\begin{align*}
u&=1-\tau-\omega^{2}\tau(\tau-1)-\omega(\tau-1)+\omega^{2}(\tau-1)\\
&=(\tau-1)(-1-\omega^{2}\tau-\omega+\omega^{2})\\
&=(\tau-1)\omega^{2}(\tau-2)=-\omega^{2}
\end{align*}
\begin{align*}
u' &= -1+\omega\tau(\tau-1)+\tau-\omega(\tau-1)+\omega^{2}(\tau-1)\\
&= (\tau-1)(1+\omega\tau-\omega+\omega^{2})\\
&= (\tau-1)\omega(\tau-2)=\omega
\end{align*}
car $\tau^{2}-3\tau+1=0$.\\
Nous avons alors:
\begin{align*}
t_{2}^{3} &= ut_{2}^{2}-u't_{2}-id_{M},\\
t_{2}^{4} &= (u^{2}-u')t_{2}^{2}-(uu'+1)t_{2}-u id_{M} =\omega^{2}id_{M}
\end{align*}
d'où $t_{2}^{4}\in Z(G^{+})$, $t_{2}^{4}$ est d'ordre $3$ et $t_{2}$ est d'ordre $12$.\\
Maintenant $[c,d^{3}]=t_{2}^{2}$, donc $[c,d^{3}]$ est d'ordre $6$ et $[c,d^{3}]^{2}\in Z(G^{+})$. D'après le résultat de Burnside, nous voyons que $D(G^{+})/Z(G^{+})$ est isomorphe au groupe alterné $A_{6}$ et que $Z(G^{+})=<[c,d^{3}]^{2}>$.\\
Comme $D(G)$ est son propre groupe dérivé, il en résulte que $D(G)\simeq C_{3}*A_{6}$ extension centrale non scindée de $A_{6}$.

Nous avons $P_{t}(X)=X^{3}-(1+\theta)X^{2}+(-1+\theta')X+1=X^{3}+mX^{2}+lX+1$ et nous obtenons:
\begin{align*}
t^{3} &= -mt^{2}-lt-id_{M}\\
t^{4} &= (m^{2}-l)t^{2}+(\tau-1)t+mid_{M}\\
t^{5} &= (-m^{3}+2\tau-1)t^{2}+(-lm^{2}+l^{2}+m)t-(m^{2}-l)id_{M}.
\end{align*}
Comme $l=\omega(\tau-1)$ et $m=\omega^{2}(\tau-1)$, nous avons $m^{2}-l=\omega(\tau-1)^{2}-\omega(\tau-1)=\omega$, $m^{3}=2\tau-1$, donc $-m^{3}+2\tau-1=0$ et de même on voit que $-lm^{2}+l^{2}+m=0$. Finalement, $t^{5}=-\omega id_{M}\in Z(G)$ et $t^{15}=-id_{M}$. Comme $-id_{M}\not \in D(G)$ nous avons le résultat: $G\simeq C_{2}\times(C_{3}*A_{6})\simeq G_{27}$.

2) Considérons la représentation $R(1,2,\tau;l)$ de $W(3,4,5,)$. Exprimons que l'élément $s_{1}s_{2}^{s_{3}}$ est d'ordre $4$.

(i) Nous avons $s_{1}s_{2}^{s_{3}}$ d'ordre $4$ si et seulement si $C(s_{1},s_{2}^{s_{3}})=2$ donc $2=1+2\tau+l+2m$ et comme $lm=\gamma=\tau$, nous voyons que $l$ et $2m$ sont les racines du polynôme:
\[
Q(X)=x^{2}-(1-2\tau)X+2\tau
\]
dont le discriminant est $-3$, d'où $K=\mathbb{Q}(\tau,\omega)$ et $\{l,2m\}=\{-\omega-\tau,-\omega^{2}-\tau$\}.\\
Nous avons $C(s_{1},s_{2}^{s_{3}s_{2}})=\beta\gamma+\alpha(\gamma-1)^{2}+(\gamma-1)(l\alpha+m\beta)=3\tau+(\tau-1)(1-2\tau)$ donc 
$C(s_{1},s_{2}^{s_{3}s_{2}})=-2\tau^{2}+6\tau-1=1$. Nous obtenons $G=<s_{1},s_{2},s_{2}^{s_{3}s_{2}}>$ avec le diagramme:
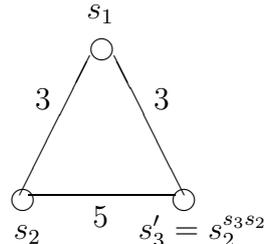
\[
\begin{picture}(150,68)

\put(58,65){$s_{1}$}
\put(63,55){\circle{7}}
\put (37,5){\circle{7}}
\put (90,5){\circle{7}}
\put(34,-8){$s_{2}$}
\put(75,-8){$s'_{3}=s_{2}^{s_{3}s_{2}}$}
\put(39,6,5){\line(1,0){48}}
\put(36,6,7){\line(1,2){23}}
\put(90,6,8){\line(-1,2){23}}
\put(41,35){$3$}
\put(80,35){$3$}
\put(60,-4){$5$}
\end{picture}
\]
et comme $s_{2}^{s_{3}s_{2}}s_{2}s_{2}^{s_{3}s_{2}}=s_{3}$, nous avons $(s_{1}s'_{3}s_{2}s'_{3})^{4}=(s_{1}s_{3})^{4}=1$ donc, d'après le 1), $G$ est isomorphe à $G_{27}$. Les autres calculs ne présentent pas de difficultés.

(ii) Exprimons maintenant que l'élément $s_{1}s_{2}^{s_{3}}$ est d'ordre $5$.\\
 Nous avons $C(s_{1},s_{2}^{s_{3}})=1+2\tau+(l+2m)\in \{\tau,3-\tau\}$ d'où deux possibilités.\\
 - $1+2\tau+(l+2m)=\tau$, alors $l$ et $2m$ sont les racines du polynôme:
 \[
 Q(X)=X^{2}+(\tau+1)X+2\tau
 \]
 et nous voyons que $\{l,2m\}=\{\omega^{2}+\tau\omega,\omega+\tau\omega^{2}\}$ d'où $K=\mathbb{Q}(\tau,\omega)$.
 
 Nous avons $C(s_{1},s_{3}^{s_{2}})=2+\tau-1-\tau=1$, donc $s_{1}s_{3}^{s_{2}}$ est d'ordre $3$\\ et $G=<s_{1},s_{2},s_{3}^{s_{2}}>$ avec le diagramme:
 
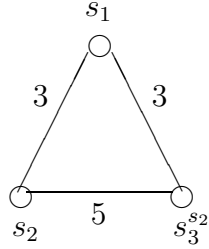
\[
\begin{picture}(150,68)

\put(58,65){$s_{1}$}
\put(63,55){\circle{7}}
\put (37,5){\circle{7}}
\put (90,5){\circle{7}}
\put(34,-8){$s_{2}$}
\put(87,-8){$s_{3}^{s_{2}}$}
\put(39,6,5){\line(1,0){48}}
\put(36,6,7){\line(1,2){23}}
\put(90,6,8){\line(-1,2){23}}
\put(41,35){$3$}
\put(80,35){$3$}
\put(60,-4){$5$}
\end{picture}
\]
et $s_{1}(s_{3}^{s_{2}})^{s_{2}}=s_{1}s_{3}$ avec $(s_{1}s_{3})^{4}=1$. D'après le 1) $G$ est isomorphe à $G_{27}$. Les autres calculs ne présentent pas de difficultés.\\
- $1+2\tau+l+2m=3-\tau$, donc $l+2m=2-3\tau$, $l2m=2\tau$ et $l$ et $2m$ sont les racines du polynôme 
\[
Q(X)=X^{2}+(3\tau-2)X+2\tau
\]
dont le discriminant est $(3\tau-2)^{2}-8\tau=7\tau-5$, mais $\sqrt{7\tau-5}$ n'est pas dans $K$. Ce cas ne peut pas se produire.

3) Considérons la représentation $R(\alpha,\beta,1;l)$ de $W(5,5,3)$ où $\alpha$ et $\beta$ sont des racines de $u_{5}(X)= X^{2}-3X+1$ et $\gamma=1$. Nous avons:
\[
(s_{3}s_{2}^{s_{1}})^{3}=1\iff C(s_{3},s_{2}^{s_{1}})=\gamma+\alpha\beta+l\alpha+m\beta=1,
\]
\[
(s_{1}s_{2}^{s_{3}})^{4}=1\iff C(s_{1},s_{2}^{s_{3}})=\alpha+\beta\gamma+l\alpha+m\beta=2.
\]
Comme $\gamma=1$ nous obtenons: $\alpha\beta=-(l\alpha+m\beta)$ et $\alpha+\beta=2-(l\alpha+m\beta)$ donc $\alpha+\beta-\alpha\beta=2$. Si $\beta=\alpha$, nous obtenons $\alpha^{2}-2\alpha+2=0$ ce qui n'est pas; donc $\beta\neq\alpha$ et nous pouvons choisir $\alpha=\tau$ et $\beta=3-\tau$. Dans ces conditions $l\tau+m(3-\tau)=-1$ et $l\tau.m(3-\tau)=1$: $l\tau$ et $m(3-\tau)$ sont les racines du polynôme 
\[
Q(X)=X^{2}+X+1
\]
donc $K=\mathbb{Q}(\tau,\omega)$ et$\{l\tau,m(3-\tau)\}=\{\omega,\omega^{2}\}$.\\
Nous avons $G=<s_{2}^{s_{3}},s_{3},s_{1}>$ avec le diagramme:
 
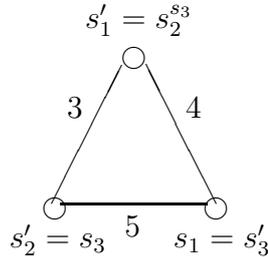
\[
\begin{picture}(150,68)

\put(47,65){$s'_{1}=s_{2}^{s_{3}}$}
\put(63,55){\circle{7}}
\put (37,5){\circle{7}}
\put (90,5){\circle{7}}
\put(22,-8){$s'_{2}=s_{3}$}
\put(76,-8){$s_{1}=s'_{3}$}
\put(39,6,5){\line(1,0){48}}
\put(36,6,7){\line(1,2){23}}
\put(90,6,8){\line(-1,2){23}}
\put(41,35){$3$}
\put(80,35){$4$}
\put(60,-4){$5$}
\end{picture}
\]
et $(s'_{1}s'_{3}s'_{2}s'_{3})^{4}=(s_{2}^{s_{3}}s_{3}^{s_{1}})^{4}$. Des calculs simples montrent que:
\[
s_{2}^{s_{3}}(a_{1})=a_{1}+(l+1)(a_{2}+ma_{3}),\quad s_{2}^{s_{3}}(a_{3})=a_{3}-l(a_{2}+ma_{3})
\]
\[
s_{3}^{s_{1}}(a_{2})=a_{2}+(\tau+m)((3-\tau)a_{1}+a_{3}), \quad s_{3}^{s_{1}}(a_{3})=a_{3}+(1-\tau)((3-\tau)a_{1}+a_{3})
\]
donc $C(s_{2}^{s_{3}},s_{3}^{s_{1}})=[(3-\tau)(l+1)-l][(\tau+m)+m(1-\tau)]$ et, tenant compte des égalités $l\tau=\omega$ et $2m(3-\tau)=2\omega^{2}$ nous voyons que $C(s_{2}^{s_{3}},s_{3}^{s_{1}})=2$ et $(s'_{1}s'_{3}s'_{2}s'_{3})^{4}=1$. d'après le 2), $G$ est isomorphe à $G_{27}$. Les autres calculs ne présentent pas de difficultés.

4) Considérons la représentation $R(\alpha,\beta,2;l)$ de $W(5,5,4)$ où $\alpha$ et $\beta$ sont racines de $u_{5}(X)$ et $\gamma=2
$. Nous avons $C(s_{1} ,s_{2}^{s_{3}})= 1= \alpha+2\beta+(\alpha l+\beta m)=1$, $C(s_{1},s_{3}^{s_{2}}) =\beta+2\alpha+(\alpha l+\beta m)=1$ et nous obtenons $\beta= \alpha=\tau$. Alors $l\tau+m\tau=1-3\tau$, d'où $l+m=-\tau$ et $lm=2$, donc $l$, $m$ sont racines du polynôme:
\[
Q(X)=X^{2}+\tau X+2
\]
dont le discriminant est $\tau^{2}-8=-3(3-\tau)=-3(\tau-2)^{2}$. Il en résulte que $K=\mathbb{Q}(\tau,\omega)$ et $\{l,m\}=\{\omega(\tau-1)+\omega^{2},\omega^{2}(\tau-1)+\omega\}$.\\
Comme $<s_{2},s_{3}>=<s_{2}^{s_{3}},s_{3}>$, on a $G=<s_{2}^{s_{3}},s_{1},s_{3}>$ avec le diagramme:
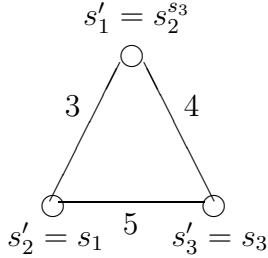
\[
\begin{picture}(150,68)

\put(47,65){$s'_{1}=s_{2}^{s_{3}}$}
\put(63,55){\circle{7}}
\put (37,5){\circle{7}}
\put (90,5){\circle{7}}
\put(22,-8){$s'_{2}=s_{1}$}
\put(76,-8){$s'_{3}=s_{3}$}
\put(39,6,5){\line(1,0){48}}
\put(36,6,7){\line(1,2){23}}
\put(90,6,8){\line(-1,2){23}}
\put(41,35){$3$}
\put(80,35){$4$}
\put(60,-4){$5$}
\end{picture}
\]
et $s'_{1}s'_{3}s'_{2}s'_{3}=s_{2}^{s_{3}}s_{1}^{s_{3}}=(s_{2}s_{1})^{s3}$, donc $s'_{1}s'_{3}s'_{2}s'_{3}$ est d'ordre $5$ et d'après le 2) $G$ est isomorphe à $G_{27}$. Les autres calculs ne présentent pas de difficultés.

5) Considérons la représentation $R(2,2,\tau;l)$ de $W(4,4,5)$.\\
 Nous avons $C(s_{1},s_{2}^{s_{3}})=1=2\tau+2+2(l+m)$ donc $2l+2m=-(2\tau+1)$, $2l2m=4\tau$ et $2l$, $2m$ sont les racines du polynôme 
\[
Q(X)=X^{2}+(2\tau+1)X+4\tau
\]
dont le discriminant est $(2\tau+1)^{2}-16\tau=-3$, donc $K=\mathbb{Q}(\tau,\omega)$ et nous obtenons:$\{2l,2m\}=\{\omega-\tau,\omega^{2}-\tau\}$.\\
On a $<s_{2},s_{3}>=<s_{2}^{s_{3}}, s_{3}>$, donc $G=<s_{1},s_{2}^{s_{3}},s_{3}>$ avec le diagramme:
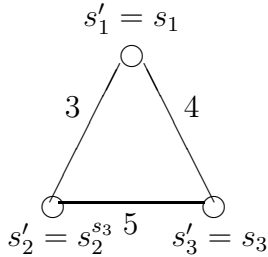
\[
\begin{picture}(150,68)

\put(47,65){$s'_{1}=s_{1}$}
\put(63,55){\circle{7}}
\put (37,5){\circle{7}}
\put (90,5){\circle{7}}
\put(22,-8){$s'_{2}=s_{2}^{s_{3}}$}
\put(76,-8){$s'_{3}=s_{3}$}
\put(39,6,5){\line(1,0){48}}
\put(36,6,7){\line(1,2){23}}
\put(90,6,8){\line(-1,2){23}}
\put(41,35){$3$}
\put(80,35){$4$}
\put(60,-4){$5$}
\end{picture}
\]
et $s'_{1}s'_{3}s'_{2}s'_{3}=s_{1}s_{2}$ est d'ordre $4$. D'après le 2), $G$ est isomorphe à $G_{27}$. Les autres calculs ne présentent pas de difficultés. 

6) Considérons la représentation $R(1,1,2;l)$ de $W(3,3,4)$\\
. Nous avons $C(s_{1},s_{2}^{s_{3}})=3-\tau=3+l+m$, $lm=2$ donc $l$ et $m$ sont les racines du polynôme 
\[
Q(X)=X^{2}+\tau X+2
\] dont le discriminant est $\tau^{2}-8=3\tau-9==-3(3-\tau)=-3(\tau-2)^{2}$ donc $K=\mathbb{Q}(\tau,\omega)$ et nous obtenons $\{l,m\}=\{\omega^{2}(\tau-2)-1,\omega(\tau-2)-1\}$.\\
On a $<s_{2},s_{3}>=<s_{2},s_{3}^{s_{2}}>$, donc $G=<s_{3},s_{1},s_{2}^{s_{3}}>$ avec le diagramme:
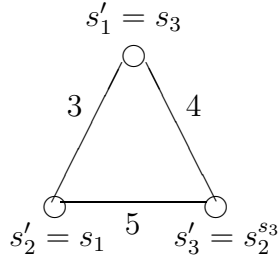
\[
\begin{picture}(150,68)

\put(47,65){$s'_{1}=s_{3}$}
\put(63,55){\circle{7}}
\put (37,5){\circle{7}}
\put (90,5){\circle{7}}
\put(22,-8){$s'_{2}=s_{1}$}
\put(76,-8){$s'_{3}=s_{2}^{s_{3}}$}
\put(39,6,5){\line(1,0){48}}
\put(36,6,7){\line(1,2){23}}
\put(90,6,8){\line(-1,2){23}}
\put(41,35){$3$}
\put(80,35){$4$}
\put(60,-4){$5$}
\end{picture}
\]
On a $s'_{1}s'_{3}s'_{2}s'_{3}=s_{3}(s_{2}^{s_{3}})s_{1}(s_{2}^{s_{3}})=(s_{2}^{s_{3}})(s_{3}^{s_{2}}s_{1})(s_{2}^{s_{3}})$. Comme $C(s_{1},s_{3}^{s_{2}})=3-\tau$, $s_{1}s_{3}^{s_{2}}$ est d'ordre $5$, donc aussi $s'_{1}s'_{3}s'_{2}s'_{3}$: $G$ est isomorphe à $G_{27}$ d'après le 2). Les autres calcule ne présentent pas de difficultés.

B. Nous montrons maintenant que les seules possibilités sont celles du théorème.\\
Nous commen\c cons par deux remarques:
\begin{itemize}
  \item Le groupe $G^{+}\simeq C*A_{6}$ ne contient aucun élément d'ordre $10$.
  \item Le groupe $G^{+}$ ne contient aucun élément réel d'ordre $6$, car dans $A_{6}$ il n'y a aucun élément d'ordre $6$.

\end{itemize}
Il résulte de ces deux remarques que si l'on veut obtenir $G_{27}$ comme quotient d'un groupe de Coxeter $W(p,q,r)$ on doit avoir $\{p,q,r\} \subset \{3,4,5\}$. Les triples $(p,q,r)$ à examiner sont donc les triples suivants:
\[
(3,3,4),(3,3,5),(3,4,4),(3,4,5),(4,4,4),(4,4,5),(5,5,3),(5,5,4),(5,5,5).
\]
Pour éliminer certains de ces triples, nous montrons soit que le corps obtenu n'est pas le bon, soit que certains éléments sont d'ordre infini. Pour cela nous exprimons que l'ordre de $s_{1}s_{2}^{s_{3}}$ est d'ordre $3$, $4$ ou $5$  (cet ordre ne peut pas être $2$, car alors le groupe $G$ est un groupe de réflexion réel: en effet, on a $<s_{2},s_{3}>=<s_{3},s_{2}^{s_{3}}>$, donc $G=<s_{1},s_{3},s_{2}^{s_{3}}>$ avec le diagramme:
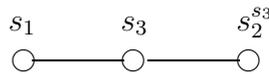
\[
\begin{picture}(150,88)
\put(29,42){\circle{7}}
\put(32,42){\line(1,0){30}}
\put(65,42){\circle{7}}
\put(70,42){\line(4,0){30}}
\put(103,42){\circle{7}}
\put(24,52){$s_{1}$}
\put(61,52){$s_{3}$}
\put(99,52){$s_{2}^{s_{3}}$}
\end{picture}
\]
d'où le résultat.)

Nous éliminons les triples $(3,4,4), (4,4,4),(5,5,5)$, les autres triples se traitant de la même manière.\\
- Si $(p,q,r)=(4,4,3)$, alors $\alpha=\beta=2,\gamma=1$ et nécessairement $s_{1}s_{2}^{s_{3}}$ est d'ordre $5$, donc $C(s_{1},s_{2}^{s_{3}})=4+2(l+m)=\tau$ et $2l$ et $2m$ sont les racines du polynôme:
\[
Q(X)=X^{2}-(4-\tau)X+4
\]
dont le discriminant est $\tau^{2}-8=1-5\tau$, mais $\sqrt{1-5\tau}\not \in \mathbb{Q}(\tau,\omega)$. Ce cas ne peut pas se produire.\\
- Si $(p,q,r)=(4,4,4)$, alors $\alpha=\beta=2=\gamma$ et comme ci-dessus $s_{1}s_{2}^{s_{3}}$ est d'ordre $5$, donc $C(s_{1},s_{2}^{s_{3}})=6+2(l+m)=\tau$ et $2l$ et $2m$ sont les racines du polynôme 
\[
Q(X)=X^{2}-(\tau-6)X+8
\]
dont le discriminant est $\tau^{2}-12\tau+4=3-9\tau$, mais $\sqrt{3-9\tau}\not \in \mathbb{Q}(\tau,\omega)$. Ce cas ne peut pas se produire.\\
- Si $(p,q,r)=(5,5,5)$, alors nous pouvons sans perte de généralité supposer que $\alpha=\beta=\tau$ et $\gamma=lm \in \{\tau,3-\tau\}$.\\
Nous avons $C(s_{1},s_{2}^{s_{3}})=\alpha+\beta\gamma+\tau(l+m)=\tau+\tau\gamma+\tau(l+m)$.\\
Considérons deux cas suivant la valeur de $\gamma$.\\
\textbf{Premier Cas.} $\gamma=\tau$, alors $C(s_{1},s_{2}^{s_{3}})=\tau^{2}+\tau+\tau(l+m)=4\tau-1+\tau(l+m)$.\\
-- Si $s_{1}s_{2}^{s_{3}}$ est d'ordre $3$, on a $\tau(l+m)=2-4\tau$, d'où $l+m=2(5-2\tau),lm=\tau$, donc $l$ et $m$ sont les racines du polynôme
\[
Q(X)=x^{2}-2(5-2\tau)X+\tau
\]
dont le discriminant est $21-9\tau$, mais $\sqrt{21-9\tau} \not \in \mathbb{Q}(\tau,\omega)$. Ce cas ne peut pas se produire.\\
--Si $s_{1}s_{2}^{s_{3}}$ est d'ordre $4$, on a $\tau(l+m)=3-4\tau$, d où $l+m=3-4\tau,lm=\tau$ et $l$ et $m$ sont les racines du polynôme
\[
Q(X)=x^{2}+(3\tau-5)X+\tau
\]
dont le discriminant est $16-17\tau$, mais $\sqrt{16-17\tau} \not \in \mathbb{Q}(\tau,\omega)$. Ce cas ne peut pas se produire.\\
-- Si $s_{1}s_{2}^{s_{3}}$ est d'ordre $5$, alors $4\tau-1+\tau(l+m)\in \{\tau,3-\tau\}$ d'où deux possibilités:\\
--- si $4\tau-1+\tau(l+m)=\tau$, on a $l+m=-\tau,lm=\tau$, donc $l$ et $m$ sont les racines du polynôme
\[
Q(X)=x^{2}+\tau X+\tau
\]
dont le discriminant est $\tau^{2}-4\tau$, mais $\sqrt{\tau^{2}-4\tau} \not \in \mathbb{Q}(\tau,\omega)$. Ce cas ne peut pas se produire.\\
--- si $4\tau-1+\tau(l+m)=3-\tau$, on a $l+m=7-4\tau,lm=\tau$, donc $l$ et $m$ sont les racines du polynôme
\[
Q(X)=x^{2}+(4\tau-7) X+\tau
\]dont le discriminant est $33-12\tau$, mais $\sqrt{33-12\tau}\not \in \mathbb{Q}(\tau,\omega)$.\\
On ne peut donc pas avoir $\gamma=\tau$.\\
\textbf{Deuxième cas.} $\gamma=3-\tau$, alors $C(s_{1},s_{2}^{s_{3}})=\tau+1+\tau(l+m)$.\\
- Si $s_{1}s_{2}^{s_{3}}$ est d'ordre $3$, on a $\tau+1+\tau(l+m)=1$ d'où $l+m=-1,lm=3-\tau$ donc $l$ et $m$ sont les racines du polynôme
\[
Q(X)=x^{2}+ X+3-\tau
\]
dont le discriminant est $4\tau-11$, mais $\sqrt{4\tau-11}\not \in \mathbb{Q}(\tau,\omega)$.\\
- Si $s_{1}s_{2}^{s_{3}}$ est d'ordre $4$, on a $\tau+1+\tau(l+m)=2$ d'où $l+m=2,\tau,lm=3-\tau$ donc $l$ et $m$ sont les racines du polynôme
\[
Q(X)=x^{2}-(2-\tau) X+3-\tau
\]
dont le discriminant est $\tau^{2}-8=-3(3-\tau)=-3(\tau-2)^{2}$, donc on a le bon corps dans ce cas, mais $¢(s_{2},s_{3}^{s_{1}})=3+\tau$, ce qui implique que $s_{2}s_{3}^{s_{1}}$ est d'ordre infini, ce qui ne convient pas.\\
- Si $s_{1}s_{2}^{s_{3}}$ est d'ordre $5$, on a $\tau+1+\tau(l+m)\in \{\tau,3-\tau\}$ d'où deux possibilités:\\
-- Si $\tau+1+\tau(l+m)=\tau$, on a $l+m=\tau-3,lm=\tau-3$ donc $l$ et $m$ sont les racines du polynôme
\[
Q(X)=x^{2}+(3-\tau) X+3-\tau
\]
dont le discriminant est: $2-\tau$, mais $\sqrt{2-\tau}\not \in \mathbb{Q}(\tau,\omega)$.\\
-- Si $\tau+1+\tau(l+m)=3-\tau$, on a $l+m=2(2-\tau),lm=3-\tau$ donc $l$ et $m$ sont les racines du polynôme
\[
Q(X)=x^{2}-2(2-\tau) X+3-\tau=(X-(2-\tau))^{2}
\]
On obtient $\Delta =0$, avec comme corps $K=\mathbb{Q}(\sqrt{5})$ ce qui n'est pas le bon corps. On ne peut pas avoir ce cas.

Nous verrons dans la dernière partie de ce travail à quoi correspond cette possibilité.
La fin de cette démonstration ne présente pas de difficultés et se fait comme dans le théorème précédent.
\end{proof}

\begin{center}
Université de Picardie Jules Verne\\
Pôle Scientifique\\
Laboratoire LAMFA, UMR CNRS 7352\\
33, rue Saint Leu\\
80039 Amiens Cedex\\
francois.zara@u-picardie.fr
\end{center}

\end{document}